\documentclass[10pt,letterpaper,reqno]{amsart}
\usepackage{amssymb,a4wide,amsthm,amsmath,amsfonts,xcolor}
\usepackage{mathrsfs}
\usepackage{color}
\usepackage{enumitem}
\usepackage[utf8]{inputenc}
\usepackage[T1]{fontenc}
\usepackage[colorlinks,citecolor=blue,urlcolor=blue]{hyperref}
\numberwithin{equation}{section}

\usepackage{bbm}
\usepackage{verbatim}

\usepackage{amsaddr}

\newtheorem{theorem}{Theorem}[section]
\newtheorem{lemma}[theorem]{Lemma}
\newtheorem{proposition}[theorem]{Proposition}
\newtheorem{assumption}[theorem]{Assumption}
\newtheorem{corollary}[theorem]{Corollary}
\newtheorem{definition}[theorem]{Definition}
\newtheorem{main result}{Main Result}

\newtheorem{remark}[theorem]{Remark}

\newcommand\dela[1]{}

\def\m{m^{\prime}}

\def\l{\left}
\def\r{\right}

\def\p{\prime}

\title[Well-Posedness and large deviations for the SLLB equation]{Well-Posedness and large deviations for the \\stochastic Landau-Lifshitz-Bloch equation}

\author[S. Gokhale]{Soham Gokhale}
\address{School of Mathematics\\ Indian Institute of Science Education and Research Thiruvananthapuram\\ Trivandrum 695551, INDIA}	\email{gokhalesoham16@iisertvm.ac.in}
\author[U.~Manna]{Utpal Manna}
\address{School of Mathematics\\ Indian Institute of Science Education and Research Thiruvananthapuram\\ Trivandrum 695551, INDIA}
\email{manna.utpal@iisertvm.ac.in}

\keywords{Large deviations principle, Stochastic Landau Lifshitz Bloch equation, Ferromagnetism, Wentzell-Freidlin type large deviations principle}

\subjclass{60H15 
}

\begin{document}
	\maketitle
	\textbf{Abstract:} The stochastic Landau-Lifshitz-Bloch equation in dimensions $1,2,3$ perturbed by pure jump noise (in the Marcus canonical form) is considered. The proof for the existence of a martingale solution uses the classical Faedo-Galerkin approximation, followed by compactness and tightness arguments, including those of Aldous, and Jakubowski's version of the Skorohod representation theorem, among others. Pathwise uniqueness and the theory of Yamada and Watanabe give the existence of a strong solution (for dimensions $1,2$). Later, a weak convergence method is used to establish a Wentzell-Freidlin type large deviations principle for the small noise asymptotic of solutions (for dimensions $1,2$).

	\section{Introduction}\label{section Introduction}
	Initiated by Weiss (see \cite{brown1963micromagnetics} and references therein) and further developed by Landau and Lifshitz (\cite{Landau+Lifshitz_1935_TTheoryOf_MagneticPermeability_Ferromagnetic,Landau1992theory}) and Gilbert \cite{Gilbert}, the study of the theory of ferromagnetism is one of the areas of rapidly increasing importance and applications. 
	The proposed equation of motion for a ferromagnetic medium is now a basis for many studies of magnetic structures, aiding substantially in the field of information storage and processing.
	For temperatures ($\mathbb{T}$) below the Curie temperature ($\mathbb{T}_c$), the magnetization can be modeled by the Landau-Lifshitz-Gilbert (LLG) equation. The LLG equation assumes that the magnetization length remains a constant, which is a serious restriction. This approach is not suitable for higher temperatures. For instance, in heat assisted magnetic recording, the electronic temperature can rise. Also, the magnetization is an average over some distribution function, and its length can change \cite{garanin1997fokker}. \\
	\subsection{The Landau-Lifshitz-Bloch Equation}\label{The Landau-Lifshitz-Bloch Equation}
	Garanin \cite{Garanin_1991_Generalized_EquationOfMotion_Ferromagnet} (see also \cite{garanin1997fokker,garanin2004thermal}) developed a thermodynamically consistent approach, the Landau-Lifshitz-Bloch (LLB) equation for ferromagnetism. The LLB equation is valid for temperatures both below and above the Curie temperature and essentially interpolates between the LLG equation at low temperature and the Ginzburg-Landau theory for phase transition. LLB micromagnetics has become a real alternative to LLG micromagnetics for temperatures that are close to the Curie temperature $\l( \mathbb{T} \geq \frac{3\mathbb{T}_c}{4}\r)$. For instance, in light-induced demagnetization with powerful fs lasers, the electronic temperature is normally raised higher than $\mathbb{T}_c$. While LLG micromagnetics cannot work under these, micromagnetics based on the LLB equation has been proven to describe the observed fs magnetization dynamics \cite{garanin1997fokker}.\\
	The average spin polarization $m$ in a domain $\mathcal{O}\subset\mathbb{R}^d,d=1,2,3$ for $t>0$ satisfies the following LLB equation.
	\begin{align}\label{deterministic SLLBE}
		\begin{cases}
			&\frac{\partial m}{\partial t} =  \gamma m \times H_{\text{eff}} + L_1\frac{1}{|m|_{\mathbb{R}^3}^2}(m \cdot H_{\text{eff}})m - L_2\frac{1}{|m|_{\mathbb{R}^3}^2} m\times (m\times H_{\text{eff}}),\\
			&\frac{\partial m}{\partial \eta}(t,x) =  0,\ t>0,\ x\in \partial\mathcal{O}, \\
			&m(0) =  m_0.
		\end{cases}
	\end{align}
	Here $\gamma>0$ is the gyromagnetic ratio and $L_1,L_2$ are the longitudinal and transverse damping parameters respectively and $H_{\text{eff}}$ is the effective field. $\eta$ denotes the outward pointing normal to the boundary $\partial \mathcal{O}$.\\
	Using the identity
	\begin{align*}
		a \times (b \times c) = b(a \cdot c) - c(a \cdot b)\ \text{for}\ a,b,c\in\mathbb{R}^3,
	\end{align*}
	for the last term gives
	\begin{align*}
		m \times (m \times H_{\text{eff}}) = (m \cdot H_{\text{eff}})m - H_{\text{eff}} \l|m\r|_{\mathbb{R}^3}^2.
	\end{align*}
	For temperature above the Curie temperature, we have $L_1 = L_2 = \kappa_1$ (say).
	Therefore the resulting equation is
	\begin{align}\label{Definition of LLBE}
		\frac{\partial m}{\partial t} = \gamma m \times H_{\text{eff}} + \kappa_1 H_{\text{eff}} .
	\end{align}
	The effective field is given by
	\begin{equation}\label{Definition of H eff}
		H_{\text{eff}} = \Delta m - \frac{1}{\mathcal{X}_{||}}\l( 1 + \frac{3}{5}\frac{\mathbb{T}}{\mathbb{T}-\mathbb{T}_c}\l|m\r|_{\mathbb{R}^3}^2 \r) m.
	\end{equation}	
	
	$\mathcal{X}_{||}$ is the longitudinal susceptibility. On the right hand side of \eqref{Definition of H eff}, the first term denotes the exchange field, the second denotes the entropy correction field.
	
 Using \eqref{Definition of H eff}, we can write  equation \eqref{Definition of LLBE} as
	\begin{align}
		\nonumber \frac{\partial m}{\partial t} = &  \gamma m \times \Delta m -    \frac{\gamma}{\mathcal{X}_{||}}\l( 1 + \frac{3}{5}\frac{\mathbb{T}}{\mathbb{T}-\mathbb{T}_c}\l|m\r|_{\mathbb{R}^3}^2 \r) \l( m \times m \r) + \kappa_1 \Delta m \\
		& - \frac{\kappa_1}{\mathcal{X}_{||}}\l( 1 + \frac{3}{5}\frac{\mathbb{T}}{\mathbb{T}-\mathbb{T}_c}\l|m\r|_{\mathbb{R}^3}^2 \r) m  .
	\end{align}
	Denoting $\mu = \frac{3}{5}\frac{\mathbb{T}}{\mathbb{T}-\mathbb{T}_c} $ and $\kappa =  \frac{\kappa_1}{\chi_{||}}$ results in the following equality.
	\begin{align}\label{LLB equation with constants}
		\frac{\partial m}{\partial t} = &   \kappa_1 \Delta m + \gamma m \times \Delta m -  
		\kappa\l( 1 + \mu \l|m\r|_{\mathbb{R}^3}^2 \r) m 
		.
	\end{align}
	Le \cite{LE_Deterministic_LLBE} showed the existence of a weak solution in a bounded domain.	
	For some recent developments on the LLB equation, we refer the reader to  \cite{Ayouch+Benmouane+Essoufi_2022_RegularSolution_LLB,Hadamache+Hamroun_2020_LargeSolution_LLBE,Li+Guo+Zeng_2021_SmoothSolution_LLBE,Pu+Yang_2022_GlobalSmoothSolutions_LLBE,Wang+EtAl_SmoothSolutionLLBE} among others.
	Note that the said (deterministic) LLB equation turns out to be insufficient, for instance, to capture the dispersion of individual trajectories at high temperatures. 
	Hence, according to Brown  \cite{brown1963micromagnetics,Brown_Thermal_Fluctuations} the LLB equation is modified (see \cite{Evans_etal_2012_Stochastic_Form_LLBE_Physics,garanin2004thermal} for discussions on stochastic form of LLB equation) in order to incorporate random fluctuations and to describe noise induced transitions between equilibrium states of the ferromagnet.
	
	The study of thermal fluctuations in magnetic materials is one of the fundamental issues of modern micromagnetism, and its applications \cite{Atxitia_2016_fundamentals_LLBE}. For instance, magnetic nanostructures are very susceptible to thermal excitations. Also, many new writing techniques have thermal excitation playing an important role, either as a byproduct or triggering magnetization switching (for example, thermally induced magnetization switching).
	Brown \cite{brown1963micromagnetics,Brown_Thermal_Fluctuations} incorporated thermal fluctuation into the Landau-Lifshitz model as formal random fields. 
	Garanin and Chubykalo-Fesenko \cite{garanin2004thermal} suggested the treatment of LLB equation following Brown's treatment of the LLG equation. However, the requirement of the Boltzmann distribution in equilibrium is not fulfilled in the vicinity of the Curie temperature.
	Evans \textit{et. al.} in \cite{Evans_etal_2012_Stochastic_Form_LLBE_Physics} introduced a different form of the stochastic LLB equation, which is consistent with the Boltzmann distribution at arbitrary temperatures. For more details about the LLB equation and its applications, we refer the reader, for instance, to \cite{Atxitia_2016_fundamentals_LLBE,Vogler_2014_LLBE_Exchange_Couple_Grains} among others and references therein.

	Following the pioneering work of Le \cite{LE_Deterministic_LLBE}, Brze\'zniak, Goldys, and Le \cite{ZB+BG+Le_SLLBE} considered the stochastic counterpart of the LLB equation (perturbed by Gaussian noise) and showed the existence of a weak martingale solution in dimensions $1,2,3$ along with pathwise uniqueness of the obtained solution for dimensions $1,2$, also proving the existence of invariant measures. 
	The authors in \cite{Jiang+Ju+Wang_MartingaleWeakSolnSLLBE} show the existence of a weak martingale solution to the stochastic LLB equation (perturbed by Gaussian noise).
	Similarly, the authors in \cite{Qiu+Tang_Wang_AsymptoticBehaviourSLLBE} consider the stochastic LLB equation in 1 dimension and establish large deviations, along with the central limit theorem.

	Our study also has another physical motivation other than understanding the phenomenon of phase transition between equilibrium states induced by thermal fluctuations. Consider a magnet that is traversed by a domain wall. There can be random imperfections (for example, defects, impurities, etc.) within the material. Hence the domain wall is not free to move and is pinned to the imperfections, which act as local potential barriers. When an external magnetic field is applied to move the domain wall away from the imperfection, the free wall moves for a while and is then trapped by another imperfection. The hysteresis loop is therefore distinguished by a series of jumps, known as Barkhausen jumps. Mayergoyz, \textit{et. al.}, in \cite{mayergoyz_2011_LandauLifshitzMagnetizationDynamics_JumpNoise,mayergoyz_2011_MagnetizationDynamics_JumpNoise} introduced jump noise process into magnetization dynamics equations in order to account for the random thermal effects. The exact mechanism of the pinning (imperfection) is imperfectly known \cite{klik+Chang_2013_thermal,puppin+Zani_2004_magnetic}. Our work might help in a better understanding of the said effects.
	
	The aim of this work is twofold. The first is to show the existence and uniqueness of the solution for the stochastic LLB equation (with pure jump noise).
	Taking motivation from \cite{ZB+BG+Le_SLLBE} and \cite{ZB+UM_WeakSolutionSLLGE_JumpNoise,ZB+UM+AAP_LargeDeviations_StochasticNematic} (see also  \cite{ZB+UM_SLLGE_JumpNoise,ZB+UM+Zhai_Preprint_LDP_LLGE_JumpNoise,ZB+Xuhui+Zhai_2019_WellPosedness_LargeDeviations_2DSNSE_Jumps,garanin2004thermal}), we perturb the effective field $H_{\text{eff}}$.
	A key technical issue while considering a jump noise is that the noise must preserve the invariance property under coordinate transformation. This is important, for example, in preserving the constraint condition for the LLG equation.
	Marcus \cite{Marcus_1981_ModellingSDE_Semimartingales}, Kunita \cite{Kunita_2004book_SDE_LevyProcesses_StochasticFlows} (see also \cite{Chechkin+Pavlyukevich_2014_MarcusVsStatonovichNoise_JumpNoise,Applebaum_2009Book_LevyProcesses_StochasticCalculus}) provide a framework to resolve this issue.
	
	Another important point is the following. The proof for the existence of a martingale solution uses the classical Faedo-Galerkin approximation. To show the tightness of the family of laws of the approximate solutions, the work of Flandoli and Gatarek \cite{Flandoli_Gatarek} is usually followed. But this method is not applicable when one considers jump noise.
	The approach by M\'etivier \cite{Metivier_SPDE_InfDimensions_Book}, Brze\'zniak, \textit{et. al.} \cite{ZB+EH_2009_MaximalRegularlty_StochasticConvolutions_LevyNoise,ZB+EH+Paul_2018_StochasticReactionDiffusion_JumpProcesses},  Motyl \cite{Motyl_2013_SNSE_LevyNoise} depends on a deterministic compactness result, which further depends on some energy bounds and the Aldous condition, which is a stochastic version of the equicontinuity result of Arz\'ela and Ascoli. A similar idea has been used in \cite{ZB+UM_WeakSolutionSLLGE_JumpNoise,ZB+UM_SLLGE_JumpNoise} for the stochastic LLG equation.
	
	We first state the problem that we have considered. Let $\mathcal{O}\subset \mathbb{R}^d,d=1,2,3$ be a bounded domain with a smooth boundary $\partial \mathcal{O}$.  The meaning of some of the terms and functions will be given at a later stage. Let $B$ denote the unit ball in $\mathbb{R}^{N}$ (excluding the center) for a fixed $N\in\mathbb{N}$. We consider the following LLB equation with pure jump noise in the Marcus canonical sense. Let $T>0$.
	\begin{align}\label{eqn problem considered with L}
		\nonumber m(t) = & m_0 + \int_{0}^{t} \l[ \Delta m(s) + m(s) \times \Delta m(s) - \l( 1 + \l| m(s) \r|_{\mathbb{R}^3}^2 \r) m(s) \r] \, ds \\
		& + \int_{0}^{t}\int_{B} \l( m(s) \times h + h \r) \diamond dL(s),\ t\in[0,T].
	\end{align}
	with the Neumann boundary condition. Here $L(t) = \l(L_1(t), \dots, L_N(t)\r), N\in\mathbb{N}$ is a $\mathbb{R}^N$-valued L\'evy process with pure jump. For the sake of simplicity, we assume $N = 1$. We postpone more details about the process $L$ and equation \eqref{eqn problem considered with L} to 
	Sections \ref{Some Notations, Assumptions}-\ref{section The Equation in Marcus Canonical Form} .

	The second part of the work aims to establish a Wentzell-Freidlin type large deviations principle for the small noise asymptotics of solutions of \eqref{eqn problem considered with L}. We use the weak convergence method based on the work \cite{Budhiraja+Chen+Dupuis_2013_LDPDrivenByPoissonProcess} (see also \cite{Budhiraja+Dupuis_2000_VariationalRepresentation_PositiveFunctionals_InfDimBrownianMotion,Budhiraja+Dupuis_2019_Book_AnalysisApproximationOfRareEvents,Budhiraja+Dupuis+Ganguly_2016_ModerateDeviations_Jumps,Budhiraja+Dupuis+Maroulas_2011_VariationalRepresentationsContinuousTimeProcesses,Dupuis+Ellis_Book_WeakConvergenceApproach_LargeDeviations}).
	We take motivation from the work of Brezniak, Goldys, and Jegaraj \cite{ZB+BG+TJ_LargeDeviations_LLGE}, wherein the authors first show the existence and uniqueness of a solution to the one dimensional stochastic LLG equation, followed by large deviations principle.\\
	Brze\'zniak and Manna in \cite{ZB+UM_WeakSolutionSLLGE_JumpNoise}  consider the LLG equation perturbed by pure jump noise (in the Marcus canonical form) and show the existence of a weak martingale solution taking values in the unit sphere $\mathbb{S}^2\subset \mathbb{R}^3$ (see also \cite{ZB+UM_SLLGE_JumpNoise} for non-zero Anisotropy energy)
	The same authors, along with Panda in \cite{ZB+UM+AAP_2019_MartingaleSolutions_NematicLiquidCrystals_PureJumpNoise} show the existence of a martingale solution to the nematic liquid crystal model, driven by pure jump noise. In \cite{ZB+UM+AAP_LargeDeviations_StochasticNematic}, the same authors establish large deviations for the liquid crystal model driven by multiplicative Gaussian noise.
	Manna and Panda in \cite{UM+AAP_2021_LargeDeviationsSNSELevyNoise} show the well posedness and large deviations for the constrained stochastic Navier-Stokes equation in 2 dimensions, perturbed by L\'evy noise in the Marcus canonical form.	
	For more related works and some works concerning large deviations, we refer the reader to the following (non-exhaustive) list of works and references therein. \cite{Jianhai+Chenggui_2015_LargeDevitaionsSDEJumps,ZB+Xuhui+Zhai_2019_WellPosedness_LargeDeviations_2DSNSE_Jumps, Chueshov+Millet_2010_Stochastic2DHydrodynamic_WellPosedness_LargeDeviations,Acosta_2000_LargeDeviationStochasticEquations,Emanuela+Hocquet,Sritharan+Sundar_2006_LargeDeviations_2DSNSE_MultiplicativeNoise,Xu+Zhang_LDP2DSNSE_LevyProcess,Zhai+Zhang_2015_LargeDeviations_2DSNSE_MultiplicativeLevyNoise}  among others.
	
	We show the existence of a solution and large deviations for the stochastic LLB equation perturbed by pure jump noise. In the spirit of \cite{UM+AAP_2021_LargeDeviationsSNSELevyNoise,Xu+Zhang_LDP2DSNSE_LevyProcess,Zhai+Zhang_2015_LargeDeviations_2DSNSE_MultiplicativeLevyNoise} among others, one may be able to extend the result to the case of L\'evy noise.
	
	The layout of the work is as follows. Section \ref{section Existence of a Weak Martingale Solution} describes the considered problem \eqref{eqn problem considered with L} in more detail.
	The main aim of the section is to show the existence of a weak martingale solution. First, the equation \eqref{eqn problem considered with L} is written using the Marcus mapping defined in Section \ref{section The Marcus Mapping}. The section also outlines some basic properties of the defined Marcus map.
	To prove the existence of a weak martingale solution, we use the Faedo-Galerkin approximation, followed by some compactness arguments and the use of a generalized Skorokhod Theorem. Section \ref{Section Pathwise Uniqueness and Existence Of Strong Solution} shows that (for $d=1,2$) the obtained solution is pathwise unique, and hence when combined with the theory of Yamada and Watanabe, the existence of a strong solution is shown (Theorem \ref{theorem existence of strong solution}).
	Section \ref{section The Large Deviations Principle} describes the second aim of the work, namely the large deviations principle. We formulate the problem and give some necessary background. The equivalence between the Laplace principle and the large deviations principle for Polish spaces is well known \cite{Dupuis+Ellis_Book_WeakConvergenceApproach_LargeDeviations}. Budhiraja, Chen, and Dupuis in \cite{Budhiraja+Chen+Dupuis_2013_LDPDrivenByPoissonProcess} establish two sufficient conditions for a sequence of laws to satisfy the large deviations principle.
	In Section \ref{section Sufficient Conditions for LDP}, we state the concerned sufficient conditions for the sequence of laws (on an appropriate space) of the solutions to the problem \eqref{eqn problem considered with L} (in an appropriate form) to satisfy the large deviations principle.
	Section \ref{section Verification of Conditions 1 and 2} is dedicated to the verification (for $d=1,2$) of the said conditions. The verification is done in a number of lemmas.
	
	\section{Some Preliminaries, Notataions}\label{Section Some Preliminaries, Notataions}

	We introduce some operators and some spaces that will be used in the work.
	
	\subsection{Some Notations, Assumptions}\label{Some Notations, Assumptions}
	\begin{remark}[\textbf{The operator }$A$]
		Let $A = -\Delta$ denote the negative Neumann Laplacian on $\mathcal{O}\subset \mathbb{R}^d,\ d=1,2,3$. Let $\l\{ e_i \r\}_{n\in\mathbb{N}}$ denote the orthonormal basis of $L^2$, consisting of eigenfunctions of $A$, corresponding to eigenvalues $\lambda_i,i\in\mathbb{N}$ \cite{Evans}. That is for each $i$
		\begin{align}
			Ae_i = \lambda e_i,\ \text{in}\ \partial \mathcal{O}\ \frac{\partial e_i}{\partial n} = 0\ \text{on}\ \partial \mathcal{O}.
		\end{align}
		Here $n$ denotes the outward pointing normal vector on the boundary $\partial \mathcal{O}$.
		Let us define $A_1$ on the space $D(A)\subset L^2$ by 
		\begin{equation*}
			A_1 = I_{L^2} + A,
		\end{equation*}
		where $I_{L^2}$ denotes the identity operator on the space $L^2$. For $\beta > 0$, we define the Hilbert space $X^{\beta} := \text{dom}(A_1^{\beta})$, with the norm
		\begin{equation}
			\l| v \r|_{X^{\beta}} := \l| \l( I + A \r)^{\beta} v \r|_{L^2}.
		\end{equation}
		We assume that $X^{0} = L^2$. The space $X^{-\beta}$ is understood as the dual space of $X^{\beta}$. More details about the spaces can be found in, for example \cite{ZB+BG+TJ_Weak_3d_SLLGE}.
		In particular, the following is a Gelfand triple (for $\beta > 0$).
		\begin{equation}
			X^{\beta} \hookrightarrow L^2 \hookrightarrow X^{-\beta}.
		\end{equation}
		For $0 \leq \beta < \frac{3}{4}$,
		\begin{equation}
			X^{\beta} = H^{2\beta}.
		\end{equation}
		Here $H^{2\beta}$ is the standard Sobolev space.
	\end{remark}
	\begin{assumption}\label{assumption main assumption}
		\hfill
		\begin{enumerate}
			\item The given probability space $\l( \Omega , \mathcal{F} , \mathbb{F} , \mathbb{P} \r)$ satisfies the usual conditions.
			\item The process $\{L(t)\}_{t\in[0,T]}$ is $\mathbb{R}$ valued, $\l(\mathcal{F}_t\r)$-adapted L\'evy process of pure jump type, defined on the above probability space, with $0$ drift. The corresponding time homogeneous Poisson Random Measure (PRM) is $\eta$.
			\item The intensity measure $Leb\otimes \nu$ is such that $\text{supp} \l(\nu\r) \subset B(0,1)\backslash \{0\} =: B\subset\mathbb{R}$.
			\item The function $h\in W^{2,\infty}$.
			\item The initial data is a $\mathcal{F}_0$ -measurable $H^1$-valued random variable.
		\end{enumerate}
	\end{assumption}
	The meaning of the stochastic integral in \eqref{eqn problem considered with L} is explained in the following few lines.
	$L(t)$ is a $\mathbb{R}$-valued L\'evy process with pure jump (i.e., the continuous part of the process is 0). The L\'evy process $L$ is understood as follows.
	\begin{equation}\label{eqn meaning of L eqn 1}
		L(t) = \int_{0}^{t} \int_{B} l \, \tilde{\eta}(ds,dl) + \int_{0}^{t} \int_{B^c} l \, \eta(ds,dl),\ t\geq 0.
	\end{equation}
	Here $\eta$ denotes a time homogeneous Poisson random measure and $\tilde{\eta}$ denotes the corresponding compensated time homogeneous Poisson random measure, with a compensator $\text{Leb} \otimes \nu$ $\l( \tilde{\eta} = \eta - \text{Leb} \otimes \nu \r)$.
	We assume that $\nu$ and $\eta$ are concentrated on $B$. That is, $\text{supp}\l(\nu\r)\subset B$ and $\eta = 0$ on $B^c$. In short, we only consider small jumps (\cite{ZB+UM_WeakSolutionSLLGE_JumpNoise,Ikeda+Watanabe}).\\
	Effectively, equation \eqref{eqn meaning of L eqn 1} becomes
	\begin{equation}\label{eqn meaning of L eqn 2}
		L(t) = \int_{0}^{t} \int_{B} l \, \tilde{\eta}(ds,dl).
	\end{equation}

	\subsection{The Marcus Mapping}\label{section The Marcus Mapping}
	This subsection describes the Marcus mapping and the resulting form of the equation \eqref{eqn problem considered with L}.
	Let \begin{equation*}
		\bar{g}:H^1 \ni v \mapsto g(v) + h \in H^1,
	\end{equation*}
	with 
	\begin{equation*}
		g:H^1 \ni v \mapsto v \times h \in H^1,
	\end{equation*}
	with fixed (given) $h\in W^{2,\infty}$.
	We observe that the map $g$, and hence the map $\bar{g}$ is a bounded map.
	We now define a map 
	\begin{equation*}
		\Phi : \mathbb{R}_{+}\times \mathbb{R} \times H^1 \to L^1, 
	\end{equation*}
	such that for each $x\in H^1$, the mapping
	\begin{equation*}
		t \mapsto \Phi(t,l,x),
	\end{equation*}
	is a $C^1$ solution to the ordinary differential equation
	\begin{equation}\label{eqn ode Marcus Mapping eqn 1}
		\frac{d\Phi}{dt}\l(t,l,x\r) = l\bar{g}\l( \Phi(t,l,x)\r),
	\end{equation}
	with
	$\Phi(0,l,x) = x$.\\
	The equation \eqref{eqn ode Marcus Mapping eqn 1} can be written as
	\begin{equation}
		\Phi(t,l,x) = x + \int_{0}^{t} l \bar{g}\l(\Phi(s,l,x)\r) \, ds.
	\end{equation}
	Arguing similarly, we can show that the map
	\begin{equation*}
		\bar{g}:H^1 \ni v \mapsto g(v) + h \in H^1,	
	\end{equation*}
	is bounded. Therefore, the mappings
	\begin{equation*}
		\mathbb{R}_{+}\times \mathbb{R} \times L^2 \ni \l( t,l,x \r) \mapsto \Phi\l( t,l,x \r) \in L^2,
	\end{equation*}
	and
	\begin{equation*}
		\mathbb{R}_{+}\times \mathbb{R} \times H^1 \ni \l( t,l,x \r) \mapsto \Phi\l( t,l,x \r) \in H^1,
	\end{equation*}
	are well defined.
	In what follows, we fix $t = 1$. $\Phi(l,\cdot) := \Phi(1,l,\cdot)$, $l\in\mathbb{R}$.
	
	\subsection{Some Properties of the Marcus Mapping and some corresponding operators}\label{section Some Properties of the Marcus Mapping and some corresponding operators}
	This subsection enlists some properties of the mapping $\Phi$.

	\begin{lemma}\label{lemma linear growth Phi}
		Let $X \in \l\{ L^2, H^1 \r\}$. Then there exists a constant $C>0$ such that for every $l\in B$, the following hold.

		\begin{equation}\label{eqn linear growth for Phi}
			\l| \Phi(l,x) \r|_{X} \leq C  \l( 1 + \l| x \r|_{X} \r),\ x\in X.
		\end{equation}
		
		\begin{equation}\label{eqn linear growth 2 for Phi}
			\l| \Phi(l,x) \r|_{X}^2 \leq C  \l( 1 + \l| x \r|_{X}^2 \r),\ x\in X.
		\end{equation}
		
	\end{lemma}
	\begin{proof}[Proof of Lemma \ref{lemma linear growth Phi}]
		We give a brief proof for the second inequality with $X = L^2$. The other cases can be done similarly.		
		For the first inequality,
		\begin{align}
			\Phi(l,x) = x + \int_{0}^{1} l\bar{g}\bigl(\Phi\l(s,l,x\r)\bigr) \, ds.
		\end{align}
		Taking the $L^2$ norm of both sides gives
		\begin{align}
			\nonumber \l| \Phi(l,x) \r|_{L^2} \leq &  \l| x \r|_{L^2} + \l| \int_{0}^{1} l\bar{g}\bigl(\phi\l(s,l,x\r)\bigr) \, ds \r|_{L^2} \\
			\nonumber \leq &  \l| x \r|_{L^2} + \int_{0}^{1} \l| l \r| \l|\bar{g}\bigl(\phi\l(s,l,x\r)\bigr) \r|_{L^2} \, ds \\
			\leq &  \l| x \r|_{L^2} + C \int_{0}^{1} \l| l \r| \l( 1 + \l| \phi\l(s,l,x\r) \r|_{L^2} \r) \, ds.
		\end{align}
		The use of the Gronwall inequality gives the desired result.
		
		For the second inequality, we again have
		\begin{align}
			\Phi(l,x) = x + \int_{0}^{1} l\bar{g}\bigl(\phi\l(s,l,x\r)\bigr) \, ds.
		\end{align}
		Therefore by Young's inequality, we can show the following.	
		\begin{align*}
			\nonumber \frac{1}{2} \l| \Phi(l,x) \r|_{L^2}^2 = &  \frac{1}{2} \l|  x  \r|_{L^2}^2 + \int_{0}^{1} l\l\langle \bar{g}\l( \Phi(s,l,x) \r) , \Phi(s,l,x) \r\rangle_{L^2} \, ds \\
			= &  \frac{1}{2} \l|  x  \r|_{L^2}^2 + \int_{0}^{1} l \l[ \l\langle g \l( \Phi(s,l,x) \r) , \Phi(s,l,x) \r\rangle_{L^2} + \l\langle h , \Phi(s,l,x) \r\rangle_{L^2}\r] \, ds \\
			= &  \frac{1}{2} \l|  x  \r|_{L^2}^2 + \int_{0}^{1} l \l[ \l\langle  \Phi(s,l,x) \times h , \Phi(s,l,x) \r\rangle_{L^2} + \l\langle h , \Phi(s,l,x) \r\rangle_{L^2}\r] \, ds \\
			= &  \frac{1}{2} \l|  x  \r|_{L^2}^2 + \int_{0}^{1} l  \l\langle h , \Phi(s,l,x) \r\rangle_{L^2} \, ds \\
			\leq & \frac{1}{2}\l|  x  \r|_{L^2}^2 + \int_{0}^{1}  \l| l \r|  \l[ \frac{1}{2} \l| h \r|_{L^2}^2 + \l| \Phi(s,l,x) \r|_{L^2}^2 \r] \, ds \\
			\leq & \frac{1}{2}\l|  x  \r|_{L^2}^2 + \frac{1}{2} \l| h \r|_{L^2}^2 + \int_{0}^{1}  \l| l \r| \l| \Phi(s,l,x) \r|_{L^2}^2 \, ds.
		\end{align*}	
		Use of the Gronwall inequality, there exists constants $\tilde{C},C>0$ such that
		\begin{equation*}
			\l| \Phi(l,x) \r|_{L^2}^2 \leq \tilde{C} \l( 1 +  \l| x \r|_{L^2}^2 \r)  e^{\int_{0}^{1} \tilde{C} \, ds} \leq C \l( 1 +  \l| x \r|_{L^2}^2 \r).
		\end{equation*}		
		Note that all of the previous calculations go through for any $l\in\mathbb{R}$. In particular, choosing $l\in B$ can make the constants independent of $l$.		
	\end{proof}
	Consider the following mappings on $L^2$.
	\begin{align}\label{eqn definition of G}
		G(l,v) = \Phi(l,v) - v,\ v\in L^2:
	\end{align}
	\begin{align}\label{eqn definition of H}
		H(l,v) = \Phi(v) - v - l\bar{g}(v),\ v\in L^2.
	\end{align}
	With the above two mappings in mind, we state the following lemma.
	
	\begin{lemma}\label{lemma linear growth Lipschitz G and H}
		There exists a constant $C>0$ such that, for every $v\in L^2$ and $l\in B$, the following inequalities hold.
		\begin{align}\label{eqn G linear growth}
			\l|G(l,v)\r|_{L^2} \leq C \l( 1 + \l| v \r|_{L^2}\r),
		\end{align}
		\begin{align}\label{eqn G Lipschitz}
			\l| G(l,u) - G(l,v) \r|_{L^2} \leq C \l| u - v \r|_{L^2},
		\end{align}
		\begin{align}\label{eqn H linear growth}
			\l|H(l,v)\r|_{L^2} \leq C \l( 1 + \l| v \r|_{L^2}\r),
		\end{align}
		\begin{align}\label{eqn H Lipschitz}
			\l| H(l,u) - H(l,v) \r|_{L^2} \leq C \l| u - v \r|_{L^2}.
		\end{align}
	\end{lemma}
	\begin{proof}[Proof of Lemma \ref{lemma linear growth Lipschitz G and H}]
		We give proofs for Lipschitz continuity (\eqref{eqn G Lipschitz},\eqref{eqn H Lipschitz}). Linear growth follows as a consequence. See also Lemma \ref{lemma linear growth G,H,b with epsilon} for Linear Growth.
		\begin{align}
			\nonumber \l|G(l,u) - G(l,v)\r|_{L^2} = & \l| \Phi(l,u) - u - \Phi(l,v) + v \r|_{L^2}\\
			\nonumber \leq & \int_{0}^{1} \l| l \r| \l| \bar{g}(\Phi(s, l, u)) - \bar{g}(\Phi(s,l,v)) \r|_{L^2} \, ds \\
			\leq &  \int_{0}^{1} \l| h \r|_{L^{\infty}} \l| l \r| \l| \Phi(s, l, u) - \Phi(s,l,v) \r|_{L^2} \, ds .
		\end{align}
		We take a small detour here. By the triangle inequality,
		\begin{align}
			\nonumber  \l| \Phi(l,u)  - \Phi(l,v)  \r|_{L^2} - \l| u - v \r|_{L^2} \leq & \l| \Phi(l,u) - u - \Phi(l,v) + v \r|_{L^2} \\
			\leq & \int_{0}^{1} \l| h \r|_{L^{\infty}} \l| l \r| \l| \Phi(s, l, u) - \Phi(s,l,v) \r|_{L^2} \, ds.
		\end{align}
		In particular,
		\begin{equation}
			\l| \Phi(l,u)  - \Phi(l,v)  \r|_{L^2} \leq  \l| u - v \r|_{L^2} + \int_{0}^{1} \l| h \r|_{L^{\infty}} \l| l \r| \l| \Phi(s, l, u) - \Phi(s,l,v) \r|_{L^2} \, ds.
		\end{equation}
		By the Gronwall inequality, we have,
		\begin{equation}
			\l| \Phi(l,u)  - \Phi(l,v)  \r|_{L^2} \leq  \l| u - v \r|_{L^2} e^{1\l| h \r|_{L^{\infty}} \l| l \r|}.
		\end{equation}
		Therefore
		\begin{align}
			\nonumber \l|G(l,u) - G(l,v)\r|_{L^2} \leq &  \int_{0}^{1} \l| h \r|_{L^{\infty}} \l| l \r| \l| \Phi(s, l, u) - \Phi(s,l,v) \r|_{L^2} \, ds \\
			\nonumber \leq & \int_{0}^{1} \l| h \r|_{L^{\infty}} \l| l \r| e^{s\l| h \r|_{L^{\infty}} \l| l \r|} \l| u - v \r|_{L^2} \, ds \\
			\leq & \l| h \r|_{L^{\infty}} \l| l \r| \l( e^{\l| h \r|_{L^{\infty}} \l| l \r|} - 1 \r) \l| u - v \r|_{L^2}.
		\end{align}
		Therefore there exists a constant $C>0$ such that
		\begin{equation}
			\l|G(u) - G(v)\r|_{L^2} \leq C \l|u - v\r|_{L^2}.
		\end{equation}
		Now,
		\begin{align}
			\nonumber \l| H(l,u) - H(l,v) \r|_{L^2} = & \l| \Phi(l,u) - u - l\bar{g}(u) - \Phi(l,u) + u + l\bar{g}(u) \r|_{L^2} \\
			\nonumber \leq & \l| u - v \r|_{L^2} + \l| G(l,u) - G(l,v) \r|_{L^2} +  \l| l \r|\l|\bar{g}(u) - \bar{g}(v) \r|_{L^2}\\
			\leq & C \l| u - v \r|_{L^2}.
		\end{align}
		Hence both $G,H$ are Lipschitz continuous. Hence they also have linear growth. This concludes the proof of the Lemma \ref{lemma linear growth Lipschitz Gn and Hn}.
	\end{proof}
	We define another operator $b$ on $L^2$ as follows.
	\begin{equation}\label{eqn definition of operator b}
		b(v) = \int_{B} H(l,v) \, \nu(dl),\ v\in L^2.
	\end{equation}
	
	\subsection{The Equation in Marcus Canonical Form}\label{section The Equation in Marcus Canonical Form}
	The equation \eqref{eqn problem considered with L} is understood as follows.
	\begin{align}\label{eqn problem considered Marcus Form}
		\nonumber m(t) = & m_0 + \int_{0}^{t} \l[ \Delta m(s) + m(s) \times \Delta m(s) - \l( 1 + \l| m(s) \r|_{\mathbb{R}^3}^2 \r) m(s) \r] \, ds \\
		& + \int_{0}^{t} b\bigl(m(s)\bigr) \, ds + \int_{0}^{t} \int_{B} G\bigl(l,m(s)\bigr) \, \tilde{\eta}(dl,ds),\ t\in[0,T].
	\end{align}

	\subsection{Some Functional Spaces}\label{Section Some Functional Spaces}
	
	As pointed out in Section \ref{section Introduction}, we use certain compactness results from \cite{Metivier_SPDE_InfDimensions_Book,Aldous_1978_Stopping_times_and_tightness,Motyl_2013_SNSE_LevyNoise}, among others.
	Towards that, we define some spaces which will be used, primarily in Section \ref{section tightness}. For more details, we refer the reader to Billingsley \cite{Billingsley_1999_ConvergenceOfProbabilityMeasures_Book}, Parthasarthy \cite{Parthasarathy_1967_Book_ProbabilityMeasuresOnMetricSpaces_Book}, M\'etivier \cite{Metivier_SPDE_InfDimensions_Book}, Aldous \cite{Aldous_1978_Stopping_times_and_tightness,Aldous_1989_StoppingTimes2} among others.

	\begin{enumerate}
		\item The space $\mathbb{D}([0,T] : X^{-\beta})$ : This is the space of all c\`adl\`ag functions $v: [0,T] \to X^{-\beta}$ with the topology induced by the Skorohod metric $\delta_{T,X^{-\beta}}$ given by
		
		\begin{align}
			\nonumber \delta_{T,X^{-\beta}}(x,y) : = \inf_{\lambda\in\Lambda_T} \bigg[ & \sup_{t\in[0,T]}  \rho\bigl( x(t) , y( \lambda(t) ) \bigr) + \sup_{t\in[0,T]} \l| t - \lambda(t) \r| \\
			& + \sup_{s<t}\l| \log \l( \frac{ \lambda(t) - \lambda(s) }{ t - s } \r) \r| \bigg].
		\end{align}
		Here $\Lambda_T$ is the set of all increasing homeomorphisms $[0,T]$ and $\rho$ denotes the norm generated metric on the space $X^{-\beta}$.

		\item The space $L^2_{\text{w}}(0,T:H^2)$ : This is the space $L^2(0,T:H^2)$ endowed with the weak topology.
		\item The space $\mathbb{D}([0,T] : H^1_{\text{w}})$ : This is the space of all weakly c\'adl\'ag functions $v: [0,T] \to H^1$, with the weakest topology such that for all $u\in H^1$, the maps
		\begin{align*}
			\mathbb{D}([0,T] : H^1_{\text{w}}) \ni v \mapsto \l\langle v (\cdot) , u \r\rangle_{H^1} \in \mathbb{D}([0,T] : \mathbb{R}),
		\end{align*}
		are continuous.
	\end{enumerate}

	\section{Existence of a Weak Martingale Solution}\label{section Existence of a Weak Martingale Solution}
	The section is dedicated to the existence of a weak martingale solution for the problem \eqref{eqn problem considered Marcus Form}. We first define what we mean by a weak martingale solution. Then Theorem \ref{theorem existence of weak martingale solution} formulates certain conditions under which the said problem admits a weak martingale solution. The rest of the section is committed to the proof of Theorem \ref{theorem existence of weak martingale solution}.
	
	We now give a definition for a weak martingale solution to the problem \eqref{eqn problem considered Marcus Form}.
	
	\begin{definition}[Weak Martingale solution]\label{definition weak martingale solution}
		A weak martingale solution to the problem \eqref{eqn problem considered Marcus Form} is a tuple
		\begin{align*}
			\l( \Omega^\p , \mathcal{F}^\p , \mathbb{F}^\p , \mathbb{P}^\p , M^\p , \eta^\p\r),
		\end{align*}
		such that
		\begin{enumerate}
			\item $\l( \Omega^\p , \mathcal{F}^\p , \mathbb{F}^\p , \mathbb{P}^\p\r)$ is a probability space (with filtration $\mathbb{F}^\p = \l\{ \mathcal{F}_t^\p \r\}_{t\in[0,T]}$) satisfying the usual hypotheses.
			\item $\eta^\p$ is a time homogeneous Poisson random measure on the measurable space $\l(B,\mathcal{B}\r)$ over $\l( \Omega^\p , \mathcal{F}^\p , \mathbb{F}^\p , \mathbb{P}^\p\r)$ with intensity measure $Leb \otimes \nu$.
			\item $\m:[0,T] \times \Omega ^\p \to H^1$ is an $\mathbb{F}^\p$-progressively measurable weakly c\'adl\'ag process which satisfies the following estimates. There exists a constant $C>0$ that depends upon $T,m_0,h$ such that
			\begin{enumerate}
				\item \begin{equation}\label{eqn weak martingale solution definition bound 1}
					\mathbb{E}^\p\sup_{t\in[0,T]}\l|\m(t)\r|_{H^1}^2 \leq C,
				\end{equation}
				\item \begin{equation}\label{eqn weak martingale solution definition bound 2}
					\mathbb{E}^\p\int_{0}^{T} \l|\m(t)\r|_{H^2}^2 \, dt \leq C.
				\end{equation}
			\end{enumerate}
			\item For each $V\in L^4(\Omega^\p: H^{1})$, \eqref{eqn problem considered Marcus Form} is satisfied $\mathbb{P}^\p$-a.s. in the weak (PDE) sense as follows.
			\begin{align}\label{eqn weak martingale solution definition eqn weak form}
				\nonumber \m(t) = & \, m_0 - \int_{0}^{t}  \l\langle \nabla \m(s) , \nabla V \r\rangle_{L^2} \, ds +  \int_{0}^{t}  \l\langle \m(s) \times \nabla \m(s) , \nabla V \r\rangle_{L^2} \, ds \\
				\nonumber & -  \int_{0}^{t} \int_{\mathcal{O}} \l( 1 + \l|\m(s,x)\r|_{\mathbb{R}^3}^2 \r) \l\langle \m(s,x) ,  V(x) \r\rangle_{\mathbb{R}^3} \, dx \, ds \\
				& + \int_{0}^{t} \l\langle b\bigl(l,\m(s)\bigr) , V \r\rangle_{L^2} \, ds + \int_{0}^{t} \int_{B} \l\langle G\bigl(l,\m(s)\bigr) , V \r\rangle_{L^2} \, \tilde{\eta}(dl,ds),\ t\in[0,T].
			\end{align}
		\end{enumerate}
	\end{definition}
	
	\begin{remark}\label{Remark weak form of Laplacian and cross product terms}
		Since we are allowing the solution $\m$ to be $H^1$- valued, a natural question is how the terms $\Delta \m$, $\m \times \Delta \m$ are understood, in spite of not assuming the existence of a second derivative. We understand the terms as elements of $(H^1)^\p$.
		First, let $v\in D(A)$ and $\phi\in H^1$.
		Using Stokes theorem (see for example \cite{Temam}), we have the following.
		\begin{enumerate}
			\item \begin{align}
				\l\langle \Delta v , \psi \r\rangle_{L^2} = - \l\langle \nabla v , \nabla \psi \r\rangle_{L^2},
			\end{align}
			\item \begin{align}
				\nonumber \l\langle v \times \Delta v , \psi \r\rangle_{L^2} 
				= & -\l\langle  v \times \nabla v , \nabla \psi \r\rangle_{L^2}.
			\end{align}
		\end{enumerate}
		Therefore for $v\in H^1$, the terms $\Delta v$, $v \times \Delta v$ are understood as elements of $(H^1)^\p$ as follows.
		\begin{enumerate}
			\item \begin{equation}
				\ _{(H^1)^\p}\l\langle \Delta v , \phi \r\rangle_{H^1} = - \l\langle \nabla v , \nabla \phi \r\rangle_{L^2},\ \phi\in H^1,
			\end{equation}
			\item \begin{equation}
				\ _{(H^1)^\p}\l\langle v \times \Delta v , \phi \r\rangle_{H^1} = -\l\langle v \times \nabla v , \nabla \phi \r\rangle_{L^2},\ \phi\in H^1.
			\end{equation}
		\end{enumerate}
	\end{remark}

	\begin{theorem}[Existence of a Weak Martingale Solution]\label{theorem existence of weak martingale solution}
		Let $d=1,2,3$. Fix $T>0$. Let the initial data $m_0$ and the given function $h$ satisfy Assumption \ref{assumption main assumption}. Then the problem \eqref{eqn problem considered Marcus Form} admits a weak martingale solution.
	\end{theorem}

	The proof is structured as follows. We approximate the equation \eqref{eqn problem considered Marcus Form} using the classical Faedo-Galerkin approximation (Section \ref{section Faedo Galerkin Approximations}). We obtain uniform energy estimates on the sequence of approximates in Section \ref{Uniform Energy Estimates}. This is followed by Section \ref{section tightness}, in which we use the Aldous condition, along with some compactness criterion from \cite{Metivier_SPDE_InfDimensions_Book,Motyl_2013_SNSE_LevyNoise} to obtain tightness for the laws of the approximates on an appropriate nonmetrizable locally convex topological space. A generalized Skorohod Representation Theorem \cite{ZB+EH+Paul_2018_StochasticReactionDiffusion_JumpProcesses} is then used (Lemma \ref{lemma use of Skorohod Theorem}) to obtain another probability space, along with another sequence of random variables, which converge pointwise. Sections \ref{Section Properties of the limiting process}, \ref{Section Convergence of the Approximate Solutions} and \ref{Section Identifying the Driving Jump Process} show that the obtained limit is a weak martingale solution to the problem \eqref{eqn problem considered Marcus Form}, according to Definition \ref{definition weak martingale solution}.

	\subsection{Faedo-Galerkin Approximations}\label{section Faedo Galerkin Approximations}
	For $n\in\mathbb{N}$, let $H_n\subset L^2$ denote the linear span of the eigenfunctions $(e_1,\dots,e_n)$ corresponding to the first $n$ eigenvalues ($\lambda_1,\dots,\lambda_n$) of the Neumann Laplacian operator. We approximate the equation \eqref{eqn problem considered Marcus Form} in $H_n$.
	Let $P_n:L^2\to H_n$ denote the orthogonal projection operator. We will now define the following operators for $n\in\mathbb{N}$.
	\begin{align*}
		F_n^1 &: H_n \ni v \mapsto \Delta v \in H_n, \\
		F_n^2 &: H_n \ni v \mapsto P_n\l( v \times \Delta v \r) \in H_n, \\
		F_n^3 &: H_n \ni v \mapsto P_n\l( \l( 1 + \l| v \r|_{\mathbb{R}^3}^2 \r) v \r) \in H_n, \\
		g_n &: H_n \ni v \mapsto P_n\l(v \times h\r) \in H_n, \\
		\bar{g}_n &: H_n \ni v \mapsto P_n\l(v \times h + h\r) \in H_n.
	\end{align*}
	Let 
	\begin{equation}
		F_n(v) = F_n^1(v) + F_n^2(v) - F_n^3(v).
	\end{equation}
	\begin{lemma}\label{lemma locally Lipschitz property of Fn,gn}
		For each $n\in\mathbb{N}$, the mappings $F_n^i,i=1,2,3$, and hence also the mapping $F_n$, are locally Lipschitz on $H_n$. For $n\in\mathbb{N}$, maps $g_n,\bar{g}_n$ are Lipschitz continuous on $H_n$.
	\end{lemma}
	We skip the proof of the lemma as it is standard, see for instance, Lemma 3.1 in \cite{ZB+BG+Le_SLLBE}.
	
	\subsubsection{Marcus Mapping for Faedo-Galerkin approximations}
	We define a map $\Phi_n$ (similar to \eqref{eqn ode Marcus Mapping eqn 1}) as the solution to the following ordinary differential equation. Let $l\in\mathbb{R}$.
	\begin{equation}
		\frac{d\Phi_n}{dt}(t,l,x) = l\bar{g}_n\l(\Phi_n(t,l,x)\r),\ t\geq0.
	\end{equation}
	with $\Phi_n(0,l,x) = x ,\ x\in H_n$.\\
	Similarly to Subsection \ref{section The Marcus Mapping}, we define
	\begin{align}\label{eqn definition of Gn}
		G_n(l,v) = \Phi_n(l,v) - v,
	\end{align}
	\begin{align}\label{eqn definition of Hn}
		H_n(l,v) = \Phi_n(l,v) - v - l\bar{g}_n(v),
	\end{align}
	and
	\begin{equation}\label{eqn definition of operator bn}
		b_n(v) = \int_{B} H_n(l,v) \, \nu(dl).
	\end{equation}
	The equation \eqref{eqn problem considered Marcus Form} is approximated by the following equation
	\begin{align}\label{eqn FG approximation of problem considered Marcus Form}
		\nonumber m_n(t) = & m_n(0) + \int_{0}^{t} \l[ \Delta m_n(s) + m_n(s) \times \Delta m_n(s) - \l( 1 + \l| m_n(s) \r|_{\mathbb{R}^3}^2 \r) m_n(s) \r] \, ds \\
		& + \int_{0}^{t} b_n(m_n(s)) \, ds + \int_{0}^{t} \int_{B} G_n(l,m_n(s)) \, \tilde{\eta}(dl,ds),\ t > 0,
	\end{align}
	with $m_n(0) = P_n(m_0)$.
	
	In the remainder of the section, we replace the notation $\Phi_n$ by $\Phi$.
	The following Lemma \ref{lemma linear growth Lipschitz Gn and Hn} gives some growth estimates for the mappings $G_n$ and $H_n$. We skip the proof as it is similar to the proof of Lemma \ref{lemma linear growth Lipschitz G and H}.
	\begin{lemma}\label{lemma linear growth Lipschitz Gn and Hn}
		There exists a constant $C>0$ such that, for any $u,v\in H_n$ and $l\in B$, the following inequalities hold. 
		\begin{align}\label{eqn Gn linear growth}
			\l|G_n(l,v)\r|_{L^2} \leq C \l( 1 + \l| v \r|_{L^2}\r),
		\end{align}
		\begin{align}\label{eqn Gn Lipschitz}
			\l| G_n(l,u) - G_n(l,v) \r|_{L^2} \leq C \l| u - v \r|_{L^2},
		\end{align}
		\begin{align}\label{eqn Hn linear growth}
			\l|H_n(l,v)\r|_{L^2} \leq C \l( 1 + \l| v \r|_{L^2}\r),
		\end{align}
		\begin{align}\label{eqn Hn Lipschitz}
			\l| H_n(l,u) - H_n(l,v) \r|_{L^2} \leq C \l| u - v \r|_{L^2}.
		\end{align}
	\end{lemma}

	Lemma \ref{lemma linear growth Lipschitz Gn and Hn} implies that the mappings $G_n.H_n$ (and hence also $b_n$) have Lipschitz continuity properties. The required linear growth properties follow. Hence by Lemma \ref{lemma linear growth Lipschitz Gn and Hn} along with Lemma \ref{lemma locally Lipschitz property of Fn,gn}, the equation \eqref{eqn FG approximation of problem considered Marcus Form} admits a unique strong solution in $H_n$ (see \cite{ZB+Albeverio_2010_ExistenceGlobalSolutionSDEPoissonNoise}).
	\subsubsection{Some Properties of the Marcus Map $\Phi$}
	\begin{lemma}\label{lemma Lipschitz Continuity Linear Growth Phi L2}
		There exists a constant $C>0$ such that for every $n\in\mathbb{N}$, $l\in B$ and $v\in H_n$, the following hold.
		\begin{equation}
			\l| \Phi(l,v) \r|_{L^2} \leq C \l( 1 + \l| v \r|_{L^2} \r),
		\end{equation}
		and
		\begin{equation}
			\l| \Phi(l,v) \r|_{L^2}^2 \leq C  \l( 1 + \l| v \r|_{L^2}^2 \r).
		\end{equation}
	\end{lemma}
	\begin{proof}[Proof of Lemma \ref{lemma Lipschitz Continuity Linear Growth Phi L2}]
		We give a proof only for the second inequality.
		Recall that $\Phi$ is a solution to the ordinary differential equation
		\begin{align}
			d\Phi (l,v) = l\bar{g}_n\bigl( \Phi (l,v) \bigr) \, dt,
		\end{align}
		with intial data $v$.
		Therefore
		\begin{align}
			\l| \Phi (l,v) \r|_{L^2}^2 
			\nonumber = & \l| v \r|_{L^2}^2 + 2\int_{0}^{1} l \l\langle \bar{g}_n\bigl( \Phi (s,l,v) \bigr) , \Phi(s,l,v) \r\rangle_{L^2}\, ds \\
			\nonumber \leq & \l| v \r|_{L^2}^2 + 2\int_{0}^{1} \l| l \r| \l| \bar{g}_n\bigl( \Phi (s,l,v) \bigr) \r|_{L^2} \l| \Phi(s,l,v) \r|_{L^2}\, ds \\
			\nonumber \leq & \l| v \r|_{L^2}^2 + 2 C \l| l \r| \int_{0}^{1}  \l( 1 + \l|  \Phi (s,l,v)  \r|_{L^2} \r) \l| \Phi(s,l,v) \r|_{L^2}\, ds \\
			\nonumber \leq & \l| v \r|_{L^2}^2 + 2 C \l| l \r| \int_{0}^{1}  \l( \l| \Phi(s,l,v) \r|_{L^2} + \l|  \Phi (s,l,v)  \r|_{L^2}^2 \r) \, ds \\
			\nonumber \leq & \l| v \r|_{L^2}^2 + 2 C \l| l \r| \int_{0}^{1}  \l( 1 + 2\l|  \Phi (s,l,v)  \r|_{L^2}^2 \r) \, ds \\
			\nonumber \leq & \l| v \r|_{L^2}^2 + 4 C \l| l \r| \int_{0}^{1}  \l( 1 + \l|  \Phi (s,l,v)  \r|_{L^2}^2 \r) \, ds\\
			\leq & \l| v \r|_{L^2}^2 + 4 C \l| l \r| + 4 C \l| l \r| \int_{0}^{1}  \l|  \Phi (s,l,v)  \r|_{L^2}^2  \, ds.
		\end{align}
		In the fifth inequality, we have used the fact that for any real number $a$, we have $ \l| a \r| \leq 1 + a^2$.
		
		Therefore by the Gronwall Lemma, we have
		\begin{align}
			\l| \Phi (l,v) \r|_{L^2}^2 \leq \l[ \l| v \r|_{L^2}^2 + 4 C \l| l \r| \r] e^{\int_{0}^{1} 4C\l| l \r| \, ds}.
		\end{align}
		This concludes the proof of the lemma.
	\end{proof}	
	
	The proof of the following lemma is similar to the proof of Lemma \ref{lemma Lipschitz Continuity Linear Growth Phi L2}, and hence we skip it.
	\begin{lemma}\label{lemma Lipschitz Continuity Linear Growth Phi H1}
		There exists a constant $C>0$ such that for every $n\in\mathbb{N}$, $l\in B$ and $v\in H_n$, the following hold.
		\begin{equation}
			\l|  \Phi(l,v) \r|_{H^1} \leq C \l( 1 + \l| v \r|_{H^1} \r),
		\end{equation}
		and
		\begin{equation}
			\l| \Phi(l,v) \r|_{H^1}^2 \leq C  \l( 1 + \l| v \r|_{H^1}^2 \r).
		\end{equation}
	\end{lemma}
	
	In the following subsection, we establish some uniform energy estimates for the sequence of approximates $m_n$.

	\subsubsection{Uniform Energy Estimates}\label{Uniform Energy Estimates}
	We state a proposition that will be used in the course of the proof of Theorem \ref{theorem existence of weak martingale solution}.

	\begin{proposition}\label{proposition properties of F existence of weak martingale solution}
 
		Let $v\in H_n$. Then the following hold.
		\begin{enumerate}
			\item \begin{equation}
				\l\langle F_n^1(v) , v \r\rangle_{L^2} = - \l| \nabla v \r|_{L^2}^2,
			\end{equation}
			
			\item \begin{equation}
				\l\langle F_n^2(v) , v \r\rangle_{L^2} = 0,
			\end{equation}
			
			\item \begin{equation}
				\l\langle F_n^3(v) , v \r\rangle_{L^2} = \l| v \r|_{L^4}^4 + \l| v \r|_{L^2}^2  ,
			\end{equation}
			
			\item \begin{equation}
				\l| \l\langle \bar{g}_n(v) , v \r\rangle_{L^2} \r| \leq \frac{1}{2} \l| h \r|_{L^2}^2 + \frac{1}{2} \l| v \r|_{L^2}^2,
			\end{equation}
			
			\item \begin{equation}
				\l\langle F_n^1(v) , - \Delta v \r\rangle_{L^2} = - \l| \Delta v \r|_{L^2}^2,
			\end{equation}
			
			\item \begin{equation}
				\l\langle F_n^2(v) , - \Delta v \r\rangle_{L^2} = 0,
			\end{equation}
			
			\item 
			\begin{equation}
				\l\langle F_n^3(v) , - \Delta v \r\rangle_{L^2} =  \l( 1 + \l| v \r|_{L^2}^2 \r) \l| \nabla v \r|_{L^2}^2 + 2\l\langle v , \nabla v \r\rangle_{L^2}^2 \leq 0,
			\end{equation}

			\item \begin{equation}
				\l| \l\langle \bar{g}_n(v) , -\Delta v \r\rangle_{L^2}  \r| \leq \l| h \r|_{W^{1,\infty}} \l| v \r|_{H^1}^2 + \frac{1}{2}\l| h \r|_{H^1}^2 + \frac{1}{2} \l| \nabla v \r|_{L^2}^2.
			\end{equation}
			In particular, there exists a constant $C>0$ such that for each $n\in\mathbb{N}$,
			\begin{equation}
				\l| \l\langle \bar{g}_n(v) , -\Delta v \r\rangle_{L^2}  \r| \leq C\l( 1 + \l| v \r|_{H^1}^2 \r).
			\end{equation}
		\end{enumerate}

	\end{proposition}
	\begin{proof}
		Proofs of the inequalities 1, 2, 3, 5, 6, 7 are straightforward. We refer the reader to  \cite{ZB+BG+Le_SLLBE}, \cite{LE_Deterministic_LLBE} for proofs. The fourth  and the eight inequality follows from the use of Young's inequality.
	\end{proof}

	Using the results established so far, we proceed to obtain some uniform energy estimates on the approximate solutions $m_n$.
	\begin{lemma}\label{lemma bounds 1}
		There exists a constant $C>0$ such that for every $n\in\mathbb{N}$, the following hold.
		\begin{equation}\label{eqn L infinity L2 bound mn}
			\mathbb{E} \sup_{t\in[0,T]} \l| m_n(t) \r|_{L^2}^2 \leq C,
		\end{equation}
		\begin{equation}\label{eqn L 2 H1 bound mn}
			\mathbb{E}\int_{0}^{T} \l| m_n(t) \r|_{H^1}^2 \, dt \leq C,
		\end{equation}
		\begin{equation}\label{eqn L 4 L4 bound mn}
			\mathbb{E}\int_{0}^{T} \l| m_n(t) \r|_{L^4}^4 \, dt \leq C.
		\end{equation}
	\end{lemma}
	\begin{proof}[Proof of Lemma \ref{lemma bounds 1}]
		Consider the mapping 
		\begin{align*}
			\l\{ \psi : H_n \ni v\mapsto \frac{1}{2} \l|v\r|_{L^2}^2 \in\mathbb{R} \r\}.
		\end{align*}
		Applying the It\^o-L\'evy formula (see for example \cite{UM+AAP_2021_LargeDeviationsSNSELevyNoise}) to $\psi$ gives
		\begin{align}\label{eqn bounds lemma 1 inequality 6}
			\nonumber \psi\bigl(m_n(t)\bigr) - \psi(m_n(0)) = & \int_{0}^{t} \psi^\p(m_n(s)) \biggl( F_n\bigl(m_n(s)\bigr)\biggr) \, ds \\
			\nonumber & + \int_{0}^{t} \int_{B} \l[\psi\biggl(\Phi\bigl(l,m_n(s-)\bigl)\biggr) - \psi\bigl(m_n(s-)\bigr)\r] \tilde{\eta}(ds,dl) \\
			\nonumber & + \int_{0}^{t} \int_{B} \bigg[ \psi\biggl(\Phi\bigl(l,m_n(s-)\bigr)\biggr) - \psi\bigl(m_n(s-)\bigr) \\
			& \qquad- l \l\langle \psi^\p\bigl(m_n(s)\bigr) , \bar{g}\bigl(m_n(s)\bigr)  \r\rangle_{L^2} \bigg] \, \nu(dl) \, ds.
		\end{align}
		Using Proposition \ref{proposition properties of F existence of weak martingale solution} gives
		\begin{align}\label{eqn bounds lemma 1 inequality 2}
			\nonumber \psi(m_n(t)) + & \int_{0}^{t} \l| \nabla m_n(s) \r|_{L^2}^2 \, ds + \int_{0}^{t} \l|  m_n(s) \r|_{L^4}^4 \, ds + \int_{0}^{t} \l|  m_n(s) \r|_{L^2}^2 \, ds \\
			\nonumber \leq & \psi(m_n(0)) + \int_{0}^{t} \int_{B} \l[\psi\biggl(\Phi\bigl(l,m_n(s-)\bigr)\biggr) - \psi\bigl(m_n(s-)\bigr)\r] \tilde{\eta}(ds,dl) \\
			& + \int_{0}^{t} \int_{B} \l[ \psi\biggl(\Phi\bigl(l,m_n(s-)\bigr)\biggr) - \psi\l(m_n(s-)\r) - l \l\langle \psi^\p\bigl(m_n(s)\bigr) , \bar{g}\bigl(m_n(s)\bigr)  \r\rangle_{L^2} \r] \, \nu(dl) \, ds.
		\end{align}
		The second, third, and fourth terms on the left hand side of the above inequality are non-positive and hence can be neglected, still preserving the inequality. Moreover, the second and the fourth integrals combine to give the full $H^1$ norm. This will be used at a later stage.
		Therefore
		\begin{align}\label{eqn bounds lemma 1 inequality 3}
			\nonumber \psi(m_n(t)) \leq & \psi(m_n(0)) + \int_{0}^{t} \int_{B} \l[\psi\biggl(\Phi\bigl(l,m_n(s-)\bigr)\biggr) - \psi\bigl(m_n(s-)\bigr)\r] \tilde{\eta}(ds,dl) \\
			\nonumber & + \int_{0}^{t} \int_{B} \bigg[ \psi\biggl(\Phi\bigl(l,m_n(s-)\bigr)\biggr) - \psi\bigl(m_n(s-)\bigr) \\
			& \qquad- l \l\langle \psi^\p\bigl(m_n(s)\bigr) , \bar{g}\bigl(m_n(s)\bigr)  \r\rangle_{L^2} \bigg] \, \nu(dl) \, ds.
		\end{align}
		In particular, the last term of the above inequality can be bounded as follows.
		\begin{align}\label{eqn bounds lemma 1 inequality 4}
			\nonumber & \l|\int_{0}^{t} \int_{B} \l[ \psi\biggl(\Phi\bigl(l,m_n(s-)\bigr)\biggr) - \psi\bigl(m_n(s-)\bigr) - l \l\langle \psi^\p\bigl(m_n(s)\bigr) , \bar{g}\bigl(m_n(s)\bigr)  \r\rangle_{L^2} \r] \, \nu(dl) \, ds\r| \\
			\nonumber \leq & \int_{0}^{t} \int_{B} \l[ \psi\biggl(\Phi\bigl(l,m_n(s-)\bigr)\biggr) + \psi\bigl(m_n(s-)\bigr) + \l| l \r| \l| m_n(s) \r|_{L^2}  \l| \bar{g}(m_n(s)) \r|_{L^2}  \r] \, \nu(dl) \, ds \\
			\nonumber \leq & \int_{0}^{t} \int_{B} C \l[ \biggl(1 + \psi\bigl(m_n(s-)\bigr)\biggr) + \psi\bigl(m_n(s-)\bigr) + \l| l \r| \l| m_n(s) \r|_{L^2} \bigl( 1 +  \l| m_n(s) \r|_{L^2} \bigr)    \r] \, \nu(dl) \, ds \\
			\leq &  C + C \int_{0}^{t} \l| m_n(s) \r|_{L^2}^2 \, ds.
		\end{align}
		For the second term, we have the following by the Burkh\"older-Davis-Gundy Inequality.
		\begin{align}\label{eqn bounds lemma 1 noise term bound}
			\nonumber &\mathbb{E} \sup_{t\in[0,T]} \l| \int_{0}^{t} \int_{B} \bigl[\psi\bigl(\Phi(l,m_n(s-))\bigr) - \psi\bigl(m_n(s-)\bigr)\bigr] \tilde{\eta}(ds,dl) \r| \\
			\nonumber \leq & C \mathbb{E}  \l( \int_{0}^{T} \int_{B} \bigl[\psi\bigl(\Phi(l,m_n(s-))\bigr) - \psi\l(m_n(s-)\r)\bigr]^2 \, \nu(dl) \, ds \r)^{\frac{1}{2}}\\
			\nonumber \leq & C + C \mathbb{E} \l(\int_{0}^{T} \l| m_n(s) \r|_{L^2}^4 \, ds \r)^{\frac{1}{2}} \\
			\nonumber \leq & C + C \mathbb{E}\l[  \sup_{t\in[0,T]}\l| m_n(s) \r|_{L^2} \l(\int_{0}^{T} \l| m_n(s) \r|_{L^2}^2 \, ds \r)^{\frac{1}{2}} \r] \\
			\leq & C +  \frac{\delta}{2}\mathbb{E} \sup_{t\in[0,T]}\l| m_n(s) \r|_{L^2}^2 + \frac{C^2}{2\delta}\mathbb{E} \int_{0}^{T} \l| m_n(s) \r|_{L^2}^2 \, ds .
		\end{align}
		Here again, $\delta$ is later chosen small enough so as to keep the coefficient of the corresponding term on the left hand side positive. This will be done at a later stage. Combining \eqref{eqn bounds lemma 1 inequality 3} and \eqref{eqn bounds lemma 1 inequality 4}, we get
		\begin{align}\label{eqn bounds lemma 1 inequality 5}
			\nonumber \psi(m_n(t)) \leq & \psi(m_n(0)) + \l| \int_{0}^{t} \int_{B} \l[\psi\l(\Phi(l,m_n(s-))\r) - \psi\l(m_n(s-)\r)\r] \tilde{\eta}(ds,dl)\r| \\
			&+  C + C \int_{0}^{t} \l| m_n(s) \r|_{L^2}^2 \, ds.
		\end{align}
		With the inequality \eqref{eqn bounds lemma 1 inequality 2} in mind, we note that for $v\in H^1$,
		$$\l| v \r|_{H^1}^2 = \l| v \r|_{L^2}^2 + \l| \nabla v \r|_{L^2}^2.$$
		Combining this with the inequality \eqref{eqn bounds lemma 1 inequality 4} gives, for a suitable constant $C>0$,
		\begin{align}\label{eqn bounds lemma 1 inequality 1}
			\nonumber \l|m_n(t)\r|_{L^2}^2  & + \int_{0}^{t} \l| m_n(s) \r|_{H^1}^2 \,  ds + \int_{0}^{t} \l| m_n(s) \r|_{L^4}^4 \, ds \\
			\nonumber &  \leq  C \l| m_n(0) \r|_{L^2}^2\\
			\nonumber & + \l| \int_{0}^{t} \int_{B} \l[\psi\bigl(\Phi(l,m_n(s-))\bigr) - \psi\bigl(m_n(s-)\bigr)\r] \tilde{\eta}(ds,dl)\r|\\
			& +  C + C \int_{0}^{t} \l| m_n(s) \r|_{L^2}^2 \, ds.
		\end{align}
		The second term and the third term on the left hand side are non-negative and hence can be neglected, still preserving the inequality.\\
		We go back to equation \eqref{eqn bounds lemma 1 inequality 5}. Taking the supremum over $[0,T]$, followed by taking the expectation of both sides and using the estimate \eqref{eqn bounds lemma 1 noise term bound} gives
		\begin{align}\label{eqn bounds lemma 1 inequality 7}
			\nonumber \mathbb{E}\sup_{t\in[0,T]} \l|m_n(t)\r|_{L^2}^2 \leq & \mathbb{E}\l| m_n(0) \r|_{L^2}^2
			+ \mathbb{E}\sup_{t\in[0,T]}\l| \int_{0}^{t} \int_{B} \l[\psi\bigl(\Phi(l,m_n(s-))\bigr) - \psi\bigl(m_n(s-)\bigr)\r] \tilde{\eta}(ds,dl)\r|\\
			\nonumber & +  C + C \mathbb{E} \int_{0}^{T} \sup_{r\in[0,s]}\l| m_n(r) \r|_{L^2}^2 \, ds \\
			\leq & \mathbb{E}\l| m_0 \r|_{L^2}^2 + C + C  \mathbb{E} \int_{0}^{T} \sup_{r\in[0,s]}\l| m_n(r) \r|_{L^2}^2 \, ds.
		\end{align}
		Using the Gronwall inequality concludes the proof for the first inequality \eqref{eqn L infinity L2 bound mn}.
		For the second and third inequalities, we go back to the inequality \eqref{eqn bounds lemma 1 inequality 1}. The first and the third terms are non-negative, and hence can be neglected still keeping the inequality intact. Taking the supremum over $[0,T]$, taking the expectation of both sides, and then using the bounds \eqref{eqn bounds lemma 1 noise term bound} and \eqref{eqn L infinity L2 bound mn} gives us the required result (\eqref{eqn L 2 H1 bound mn}). The bound \eqref{eqn L 4 L4 bound mn} can be obtained similarly from the same inequality.
	\end{proof}
	
	\begin{lemma}\label{lemma bounds 2}
		There exists a constant $C>0$ such that for every $n\in\mathbb{N}$, the following hold.
		\begin{equation}\label{eqn L infinity H1 bound mn}
			\mathbb{E} \sup_{t\in[0,T]} \l| m_n(t) \r|_{H^1}^2 \leq C,
		\end{equation}

		\begin{equation}\label{eqn L 2 H2 bound mn}
			\mathbb{E}\int_{0}^{T} \l| \Delta m_n(t) \r|_{L^2}^2 \, dt \leq C.
		\end{equation}
	\end{lemma}
	\begin{proof}[Proof of Lemma \ref{lemma bounds 2}]
		Proof of the lemma is similar to the proof of Lemma \ref{lemma bounds 1}. The idea is to apply the It\^o L\'evy formula to the function $\psi_2$, given by
		\begin{equation}\label{eqn bounds lemma 2 definition of psi2}
			\l\{\psi_2: H_n \ni v \mapsto \frac{1}{2} \l| \nabla v \r|_{L^2}^2 \in \mathbb{R} \r\}.
		\end{equation}
		
		The resulting equation is
		\begin{align}\label{eqn bounds lemma 2 using Ito formula}
			\nonumber &\psi_2(m_n(t)) - \psi_2(m_n(0)) =  \int_{0}^{t} \psi_2^\p\bigl(m_n(s)\bigr) \bigl( F_n(m_n(s))\bigr) \, ds \\
			\nonumber & + \int_{0}^{t} \int_{B} \l[\psi_2\bigl(\Phi\l(l,m_n(s-)\r)\bigr) - \psi_2\bigl(m_n(s-)\bigr)\r] \tilde{\eta}(ds,dl) \\
			& + \int_{0}^{t} \int_{B} \l[ \psi_2\bigl(\Phi\l(l,m_n(s-)\r)\bigr) - \psi_2\bigl(m_n(s-)\bigr) - l \l\langle \psi_2^\p\bigl(m_n(s)\bigr) , \bar{g}(m_n(s)) \r\rangle_{L^2} \r] \, \nu(dl) \, ds.
		\end{align}
		Using Proposition \ref{proposition properties of F existence of weak martingale solution}, there exists a constant $C$ independent of $n$ such that for each $t\in[0,T]$
		
		\begin{align}\label{eqn bounds lemma 2 inequality 1}
			\nonumber  \frac{1}{2} \l| \nabla m_n(t) \r|_{L^2}^2 \leq & \frac{1}{2} \l| m_0 \r|_{H^1}^2  -\int_{0}^{t} \l| \Delta m_n(s) \r|_{L^2}^2 \, ds + C\l( 1 + \int_{0}^{t} \l| m_n(s) \r|_{H^1}^2  \, ds \r) \\
			& + \l| \int_{0}^{t} \int_{B} \bigl[\psi_2\bigl(\Phi(l,m_n(s-))\bigr) - \psi_2\bigl(m_n(s-)\bigr)\bigr] \tilde{\eta}(ds,dl) \r|.
		\end{align}
		That is, for each $t\in[0,T]$, the following inequality holds.
		\begin{align}\label{eqn bounds lemma 2 inequality 2}
			\nonumber \l| \nabla m_n(t) \r|_{L^2}^2 & +  \int_{0}^{t} \l| \Delta m_n(s) \r|_{L^2}^2 \, ds \leq  \frac{1}{2} \l| m_0 \r|_{H^1}^2 +  C\l( 1 + \int_{0}^{t} \l| m_n(s) \r|_{H^1}^2  \, ds \r) \\
			& + \l| \int_{0}^{t} \int_{B} \bigl[\psi_2\bigl(\Phi(l,m_n(s-))\bigr) - \psi_2\bigl(m_n(s-)\bigr)\bigr] \tilde{\eta}(ds,dl) \r|.
		\end{align}
		The second term on the left hand side of the above inequality is non-negative. Therefore that term can be neglected for now, keeping the inequality intact. Then taking supremum over $[0,T]$ of both sides and replacing $\l| m_n(s) \r|_{H^1}$ on the right hand side by $\sup_{r\in[0,s]}\l| m_n(r) \r|_{H^1}$ gives
		\begin{align}\label{eqn bounds lemma 2 inequality 3}
			\nonumber &\sup_{t\in[0,T]}\l| \nabla m_n(t) \r|_{L^2}^2 \leq   C\l( \l| m_0 \r|_{H^1}^2 + 1 + \int_{0}^{T} \sup_{r\in[0,s]}\l| m_n(r) \r|_{H^1}^2  \, ds \r) \\
			& + \sup_{t\in[0,T]} \l| \int_{0}^{t} \int_{B} \bigl[\psi_2\bigl(\Phi(l,m_n(s-))\bigr) - \psi_2\bigl(m_n(s-)\bigr)\bigr] \tilde{\eta}(ds,dl)  \r|.
		\end{align}		
		\textbf{For the Noise Term:} Using Lemma \ref{lemma Lipschitz Continuity Linear Growth Phi H1}, and calculating similar to \eqref{eqn bounds lemma 1 noise term bound} give
		
		\begin{align}\label{eqn bounds lemma 2 noise term}
			\nonumber \mathbb{E} \sup_{t\in[0,T]} \l| \int_{0}^{t} \int_{B} \bigl[\psi_2\bigl(\Phi(l,m_n(s-))\bigr) - \psi_2\bigl(m_n(s-)\bigr)\bigr] \tilde{\eta}(ds,dl) \r| \leq & C +  \frac{1}{2}\mathbb{E} \sup_{t\in[0,T]}\l| m_n(s) \r|_{H^1}^2 \\
			& + \frac{C^2}{2}\mathbb{E} \int_{0}^{T} \l| m_n(s) \r|_{H^1}^2 \, ds.
		\end{align}		
		Taking the expectation of both sides of \eqref{eqn bounds lemma 2 inequality 3} and using Fubini's theorem, along with the inequality \eqref{eqn bounds lemma 2 noise term} gives (for some constant $C>0$, independent of $n\in\mathbb{N}$)

		\begin{equation}\label{eqn bounds lemma 2 inequality 4}
			\mathbb{E}\sup_{t\in[0,T]}\l| m_n(t) \r|_{H^1}^2 \leq 2\,  C\l( \mathbb{E} \l| m_0 \r|_{H^1}^2 + 1 + \int_{0}^{T} \mathbb{E} \sup_{r\in[0,s]}\l| m_n(r) \r|_{H^1}^2  \, ds \r).
		\end{equation}
		Note here that the first term on the left hand side of the above inequality is the full $H^1$ norm. That is be obtained by noting that the $\l|\cdot\r|_{H^1}^2 = \l|\nabla \cdot\r|_{L^2}^2  +  \l|\cdot\r|_{L^2}^2$ and the $L^2$ norm can be bounded by a constant using Lemma \ref{lemma bounds 1}.
		Applying the Gronwall inequality and noting that $\mathbb{E}\l|m_0\r|_{H^1}^2 < \infty $ implies that there exists a constant $C>0$ such that
		\begin{equation}
			\mathbb{E}\sup_{t\in[0,T]}\l| m_n(t) \r|_{H^1}^2 \leq C.
		\end{equation}
		Going back to equation \eqref{eqn bounds lemma 2 inequality 2}, we have
		\begin{align}\label{eqn bounds lemma 2 inequality 5}
			\nonumber \l| \nabla m_n(t) \r|_{L^2}^2 & + \int_{0}^{t} \l| \Delta m_n(s) \r|_{L^2}^2 \, ds \leq   C\l( 1 + \int_{0}^{t} \l| m_n(s) \r|_{H^1}^2  \, ds \r) \\
			& + \l| \int_{0}^{t} \int_{B} \bigl[\psi_2\bigl(\Phi(l,m_n(s-))\bigr) - \psi_2\bigl(m_n(s-)\bigr)\bigr] \tilde{\eta}(ds,dl) \r|.
		\end{align}
		The first term on the left hand side is non-negative. Hence the first term can be neglected, keeping the inequality intact. The resulting inequality is
		\begin{align}\label{eqn bounds lemma 2 inequality 6}
			\nonumber \int_{0}^{T} \l| \Delta m_n(s) \r|_{L^2}^2 \, ds  \leq & 2\, C\l( 1 + \int_{0}^{t} \l| m_n(s) \r|_{H^1}^2  \, ds \r) \\
			& + \l| \int_{0}^{t} \int_{B} \l[\psi_2\bigl(\Phi\l(l,m_n(s-)\r)\bigr) - \psi_2\bigl(m_n(s-)\bigr)\r] \tilde{\eta}(ds,dl) \r|.
		\end{align}
		Taking supremum over $[0,T]$, followed by taking expectation of both sides and using \eqref{eqn bounds lemma 2 noise term} along with the inequality 
		\eqref{eqn L infinity H1 bound mn} gives the desired inequality \eqref{eqn L 2 H2 bound mn} for some constant $C>0$, independent of $n$.
	\end{proof}
	
	\begin{lemma}\label{lemma bounds 3}
		For any $p\geq 1$, there exists a constant $C>0$ such that for every $n\in\mathbb{N}$, the following hold.
		\begin{equation}\label{eqn Lp L infinity L2 bound mn}
			\mathbb{E} \sup_{t\in[0,T]} \l| m_n(t) \r|_{L^2}^{2p} \leq C,
		\end{equation}
		\begin{equation}\label{eqn Lp L2 H1 bound mn}
			\mathbb{E} \l( \int_{0}^{T} \l| m_n(t) \r|_{H^1}^2 \r)^{p} \, dt \leq C.
		\end{equation}
	\end{lemma}
	\begin{proof}

		The proof is similar to the proof of Lemma \ref{lemma bounds 1}. We give a brief idea here.
		Recall the inequality \eqref{eqn bounds lemma 1 inequality 5}. Taking the power $p$ of both sides (for $p\geq 1$), followed by taking the supremum over $[0,T]$ of both sides, followed by the use of Jensen's inequality, gives
		\begin{align}
			\nonumber &\sup_{t\in[0,T]}\l| m_n(t) \r|_{L^2}^{2p} \leq   C_p\l( \l| m_0 \r|_{H^1}^{2p} + 1 + \int_{0}^{T} \sup_{r\in[0,s]}\l| m_n(r) \r|_{H^1}^{2p}  \, ds \r) \\
			& + \sup_{t\in[0,T]} \l| \int_{0}^{t} \int_{B} \l[ \psi_2\bigl(\Phi\l(l,m_n(s-)\r)\bigr) - \psi_2\bigl(m_n(s-)\bigr) - l \l\langle \psi_2^\p\bigl(m_n(s)\bigr) , \bar{g}\bigl(m_n(s)\bigr) \r\rangle_{L^2} \r] \, \nu(dl) \, ds \r|^{p}.
		\end{align}
		The inequality \eqref{eqn Lp L infinity L2 bound mn} follows after taking the expectation of both sides, using Burkh\"older-Davis-Gundy inequality,  some simplification, and then the Gronwall inequality.		
		Going back to the inequalities \eqref{eqn bounds lemma 1 inequality 2}, \eqref{eqn bounds lemma 1 inequality 1}  gives
		\begin{align}
			\nonumber \l| m_n(t) \r|_{L^2}^2 & + \int_{0}^{t} \l| m_n(s) \r|_{H^1}^2 \, ds \leq   C\l( 1 + \int_{0}^{t} \l| m_n(s) \r|_{H^1}^2  \, ds \r) \\
			& + \l| \int_{0}^{t} \int_{B} \l[ \psi_2\bigl(\Phi\l(l,m_n(s-)\r)\bigr) - \psi_2\bigl(m_n(s-)\bigr) - l \l\langle \psi_2^\p\bigl(m_n(s)\bigr) , \bar{g}\bigl(m_n(s)\bigr) \r\rangle_{L^2} \r] \, \nu(dl) \, ds \r|.
		\end{align}
		As before, the first term is non-negative and hence can be neglected, still preserving the inequality. Taking power $p$ of both sides, followed by taking the supremum over $[0,T]$ of both sides, followed by taking the expectation of both sides, followed by Jensen's inequality and Burkh\"older-Davis-Gundy inequality gives the inequality \eqref{eqn Lp L2 H1 bound mn}.

	\end{proof}
	
	\begin{lemma}\label{lemma bounds 4}
		For any $p\geq 1$, there exists a constant $C>0$ such that for every $n\in\mathbb{N}$, the following hold.
	
		\begin{equation}\label{eqn Lp L infinity H1 bound mn}
			\mathbb{E} \sup_{t\in[0,T]} \l| m_n(t) \r|_{H^1}^{2p} \leq C,
		\end{equation}
	
		\begin{equation}\label{eqn Lp L 2 H2 bound mn}
			\mathbb{E} \l( \int_{0}^{T} \l| \Delta m_n(t) \r|_{L^2}^2 \r)^{p} \, dt \leq C.
		\end{equation}
	\end{lemma}
	We skip the proof as it is similar to the proof of Lemma \ref{lemma bounds 3}, using the bounds from Lemma \ref{lemma bounds 2} instead of Lemma \ref{lemma bounds 1}.

	\subsubsection{Tightness}\label{section tightness}
	We state a few definitions and results first, particular cases of which will be used for obtaining tightness for the laws of $m_n$ on the space $\mathcal{Z}_T$. For more details, we refer the reader to Aldous \cite{Aldous_1978_Stopping_times_and_tightness,Aldous_1989_StoppingTimes2}, Motyl \cite{Motyl_2013_SNSE_LevyNoise}, M\'etivier \cite{Metivier_SPDE_InfDimensions_Book}, Parthasarathy \cite{Parthasarathy_1967_Book_ProbabilityMeasuresOnMetricSpaces_Book}.	
	We now state the Aldous condition for tightness from \cite{Metivier_SPDE_InfDimensions_Book}.
	\begin{definition}[Aldous Condition]\label{definition Aldous Condition}
		Let $\{ X_n \}_{n\in\mathbb{N}}$ be a sequence of c\'adl\'ag, adapted stochastic processes in a Banach space $E$. We say that the sequence $\{ X_n \}_{n\in\mathbb{N}}$ satisfies the Aldous condition on $E$ if for every $\varepsilon>0$ and $\eta>0$ there is $\delta>0$ such that for every sequence $\l\{ \tau_n\r\}_{n\in\mathbb{N}}$ of $\mathbb{F}$-stopping times with $\tau_n + \theta \leq T$ one has
		\begin{equation}
			\sup_{n\in\mathbb{N}} \sup_{0\leq \theta \leq \delta} \mathbb{P} \bigl\{ \| X_n\l( \tau_n + \theta \r) - X_n\l( \tau_n \r) \|_{E} \geq \eta \bigr\} \leq \varepsilon.
		\end{equation}
	\end{definition}	
	We now state a result (see M\'etivier \cite{Metivier_SPDE_InfDimensions_Book}, Motyl \cite{Motyl_2013_SNSE_LevyNoise}) stating a condition which guarantees that the sequence $\{ X_n \}_{n\in\mathbb{N}}$ satisfies the Aldous condition (as in Definition \ref{definition Aldous Condition}) in a separable Banach space.
	\begin{lemma}\label{Lemma alternate definition Aldous Condition}
		Let $E$ be a separable Banach space. Let $\{ X_n \}_{n\in\mathbb{N}}$ be a sequence of $E$-valued random variables. Assume that for every sequence $\l\{ \tau_n\r\}_{n\in\mathbb{N}}$ of $\mathbb{F}$-stopping times with $\tau_n \leq T$, and $\theta \geq 0$, with $\tau_n+\theta \leq T$, we have
		\begin{equation}
			\mathbb{E}\biggl[ \bigl\| X_n(\tau_n + \theta) - X_n(\tau_n) \bigr\|_{E}^{\gamma_1} \biggr] \leq C \theta^{\gamma_2},
		\end{equation}
		for some $\gamma_1,\gamma_2>0$ and some constant $C>0$. Then the sequence $\l\{ X_n \r\}_{n\in\mathbb{N}}$ satisfies the Aldous condition in $E$.
	\end{lemma}

	\begin{definition}[Modulus of Continuity, \cite{Metivier_SPDE_InfDimensions_Book}]\label{Definition Modulus of Continuity}
		Let $\l(\mathbb{S},\rho\r)$ be a complete separable metric space. Let $u\in \mathbb{D}\l( [0,T] : \mathbb{S}\r)$ and let $\delta>0$ be given. A modulus of continuity of $u$ is defined by
		\begin{equation}
			w_{[0,T],\mathbb{S}}\l(u,\delta\r) : = \inf_{\prod_{\delta}} \max_{t_i\in\bar{w}} \sup_{t_i\leq s<t< t_{i+1} 
				\leq T} \rho\bigl( u(t) - u(s) \bigr),
		\end{equation}
		where $\prod_{\delta}$ is the set of all increasing sequences $\bar{w} = \l( 0 = t_0 < t_1 < \dots < t_n = T \r)$ with the following property
		\begin{equation*}
			t_{i + 1} - t_i \geq \delta,\ i=0,1,\dots n - 1.
		\end{equation*}
	\end{definition}	
	We state the following criterion (see Chapter II in \cite{Metivier_SPDE_InfDimensions_Book}) for relative compactness of a subset of the space $\mathbb{D}([0,T]:\mathbb{S})$, where $\mathbb{S}$ is a complete separable metric space with metric $\rho$.
	\begin{theorem}\label{Theorem general tightness 1}
		A set $A\subset\mathbb{D}([0,T]:\mathbb{S})$ has compact closure if and only if it satisfies the following conditions
		\begin{enumerate}
			\item[(a)] there exists a dense subset $J \subset [0,T]$ such that for each $t\in J$ the set $\l\{u(t),u\in A\r\}$ has compact closure in $\mathbb{S}$.
			
			\item[(b)] 
			\begin{equation}
				\lim_{\delta\to 0}\sup_{u\in  A}w_{[0,T]}\l( u , \delta \r) = 0.
			\end{equation}
		\end{enumerate}
	\end{theorem}

	\begin{lemma}[Motyl \cite{Motyl_2013_SNSE_LevyNoise}]\label{Lemma general tightness 2}
		Let $\l\{X_n\r\}_{n\in\mathbb{N}}$ be a sequence of c\'adl\'ag, adapted stochastic processes in a separable Banach space $E$, which satisfies the Aldous condition. Then for every $\varepsilon>0$, there exists a measurable subset $A_{\varepsilon} \subset \mathbb{D}([0,T] : E)$ with 
		\begin{equation}
			\mathbb{P}^{X_n}(A_{\varepsilon}) \geq 1 - \varepsilon,
		\end{equation}
		and
		\begin{equation}
			\lim_{\delta \to 0} \sup_{u\in A_{\varepsilon}} \sup_{\l| t - s \r| < \delta} \| u(t)  - u(s) \|_{E} = 0.
		\end{equation}
	\end{lemma}	
	Let us fix the following notation for this section. For the definitions of some of the spaces, we refer the reader to Section \ref{Section Some Functional Spaces}. For $\beta\geq 0$, $p\geq 1$ and $1\leq q \leq 6$, let
	\begin{equation}\label{eqn definition of the space ZT}
		\mathcal{Z}_T = \mathbb{D}([0,T]:X^{-\beta})\cap L^p(0,T:L^q)\cap\mathbb{D}([0,T]:H^1_{\text{w}})\cap L^2_{\text{w}}(0,T:H^2).
	\end{equation}
	We endow $\mathcal{Z}_T$ with the smallest topology under which the inclusion maps (for the intersection into all the spaces) are continuous.\\
	Let 
	\begin{equation}
		Z_T : = \mathbb{D}([0,T]:X^{-\beta})\cap L^p(0,T:L^q)\cap\mathbb{D}([0,T]:H^1_{\text{w}}),
	\end{equation}
	equipped with the Borel $\sigma$-algebra (generated by the open sets in the locally convex topology of $Z_T$).

	Using Theorem \ref{Theorem general tightness 1} and Lemma \ref{Lemma general tightness 2}, we have the following lemma. For proof, we refer the reader to Proposition 5.11, \cite{ZB+UM_WeakSolutionSLLGE_JumpNoise}.
	\begin{lemma}\label{Lemma general tightness 3}
		Let $\l\{X_n\r\}_{n\in\mathbb{N}}$ be a sequence of c\'adl\'ag adapted $X^{-\beta}$-valued processes satisfying
		\begin{enumerate}
			\item[(a)] 
			\begin{equation}
				\sup_{n\in\mathbb{N}} \mathbb{E}\l[ 
				\l| X_n \r|_{L^{\infty}(0,T:H^1)}^2 \r] < \infty,
			\end{equation}
			\item[(b)] the Aldous condition in $X^{-\beta}$.
		\end{enumerate}
		Then the sequence $\l\{\mathbb{P}^{X_n}\r\}_{n\in\mathbb{N}}$ is tight in $Z_T$. That is for given $\varepsilon>0$, there exists a compact set $K_{\varepsilon} \subset Z_T$ with
		\begin{equation}
			\mathbb{P}^{X_n}(K_{\varepsilon}) \geq 1 - \varepsilon,
		\end{equation}
		for all $n\in\mathbb{N}$.
	\end{lemma}	
	We now come back to the proof of Theorem \ref{theorem existence of weak martingale solution}. The bound in $(a)$ of Lemma \ref{Lemma general tightness 3} can be given by Lemma \ref{lemma bounds 2}. Therefore, in order to show the tightness of the laws of $m_n$, it suffices to show that the Aldous condition is satisfied in the space $X^{-\beta}$, which is done in the following Lemma.
	
	\begin{lemma}\label{lemma Aldous Condition for mn}
		The sequence $\{ m_n \}_{n\in\mathbb{N}}$ satisfies the Aldous condition on the space $X^{-\beta}$ for $\beta > \frac{1}{4}$.
	\end{lemma}
	\begin{proof}[Proof of Lemma \ref{lemma Aldous Condition for mn}]
		We recall the equality \eqref{eqn FG approximation of problem considered Marcus Form}. satisfied by $m_n$.
		\begin{align}
			\nonumber m_n(t) = & m_n(0) + \int_{0}^{t} \l[ \Delta m_n(s) + m_n(s) \times \Delta m_n(s) - \l( 1 + \l| m_n(s) \r|_{\mathbb{R}^3}^2 \r) m_n(s) \r] \, ds \\
			& + \int_{0}^{t} b_n\bigl(m_n(s)\bigr) \, ds + \int_{0}^{t} \int_{B} G_n\bigl(l,m_n(s)\bigr) \, \tilde{\eta}(dl,ds).
		\end{align}
		Let the stopping time $\tau_n$ and $\theta$ be as in Lemma  \ref{Lemma alternate definition Aldous Condition}. Therefore
		\begin{align}
			\nonumber m_n(\tau_n + \theta) - m_n(\tau_n) = & \int_{\tau_n}^{\tau_n + \theta}  \Delta m_n(s)  \, ds \\
			\nonumber &
			+ \int_{\tau_n}^{\tau_n + \theta}  m_n(s) \times \Delta m_n(s)  \, ds 
			- \int_{\tau_n}^{\tau_n + \theta}  \l( 1 + \l| m_n(s) \r|_{\mathbb{R}^3}^2 \r) m_n(s)  \, ds \\
			\nonumber & + \int_{\tau_n}^{\tau_n + \theta} b_n\bigl(m_n(s)\bigr) \, ds 
			+ \int_{\tau_n}^{\tau_n + \theta} \int_{B} G_n\bigl(l,m_n(s)\bigr) \, \tilde{\eta}(dl,ds)\\
			& = \sum_{i=1}^{5} C_i I_i(t).
		\end{align}
		Before starting with the estimates, we state some embeddings that will be used in the calculations that follow. For $\beta>\frac{1}{4}$, the following embeddings are continuous.
		\begin{equation}
			X^{\beta} \hookrightarrow L^4 \hookrightarrow L^2,
		\end{equation}
		and
		\begin{equation}
			L^2 \hookrightarrow L^{\frac{4}{3}} \hookrightarrow X^{-\beta}.
		\end{equation}
		In particular, there exists a constant $C>0$ such that
		\begin{equation}
			\l|\cdot\r|_{X{-\beta}} \leq C \l|\cdot\r|_{L^2}.
		\end{equation}
		and
		\begin{equation}
			\l|\cdot\r|_{X{-\beta}} \leq C \l|\cdot\r|_{L^{\frac{4}{3}}}.
		\end{equation}
		\textbf{Calculations for $I_1$:}
		\begin{align}
			\nonumber & \mathbb{E} \l| \int_{\tau_n}^{\tau_n + \theta}  \Delta m_n(s)  \, ds  \r|_{X^{-\beta}} \leq
			C \mathbb{E}  \int_{\tau_n}^{\tau_n + \theta}  \l| \Delta m_n(s) \r|_{L^2} \, ds  \\
			\nonumber \leq & C \theta^{\frac{1}{2}} \l( \mathbb{E}  \int_{\tau_n}^{\tau_n + \theta}  \l| \Delta m_n(s) \r|_{L^2}^2 \, ds \r)^{\frac{1}{2}}  \\
			\leq & C \theta^{\frac{1}{2}}.
		\end{align}
		\textbf{Calculations for $I_2$:}
		By the Cauchy-Schwarz inequality,
		\begin{align}
			\nonumber \mathbb{E}    &    \l| \int_{\tau_n}^{\tau_n + \theta}  m_n(s) \times \Delta m_n(s)  \, ds  \r|_{X^{-\beta}} \\
			\nonumber  \leq  &  
			C \mathbb{E}  \int_{\tau_n}^{\tau_n + \theta}  \l| m_n(s) \times \Delta m_n(s) \r|_{L^{\frac{4}{3}}} \, ds  \\
			\nonumber \leq & C \mathbb{E}  \int_{\tau_n}^{\tau_n + \theta} \l| m_n(s) \r|_{L^4} \l| \Delta m_n(s) \r|_{L^{2}} \, ds  \\
			\nonumber \leq & \theta^{\frac{1}{4}} C \l( \mathbb{E}  \int_{\tau_n}^{\tau_n + \theta} \l| m_n(s) \r|_{L^4}^4  \, ds \r)^{\frac{1}{4}} \l( \mathbb{E} \int_{\tau_n}^{\tau_n + \theta} \l| \Delta m_n(s) \r|_{L^{2}}^2  \, ds \r)^{\frac{1}{2}} \\
			\nonumber \leq & \theta^{\frac{1}{4}} C \l( \mathbb{E}  \int_{0}^{T} \l| m_n(s) \r|_{L^4}^4  \, ds \r)^{\frac{1}{4}} \l( \mathbb{E} \int_{0}^{T} \l| \Delta m_n(s) \r|_{L^{2}}^2  \, ds \r)^{\frac{1}{2}} \\
			\leq &  C \theta^{\frac{1}{4}}.
		\end{align}		
		\textbf{Calculations for $I_3$:}
		Using the continuous embedding $X^{-\beta} \hookrightarrow L^{\frac{4}{3}}$ along with the Cauchy-Schwartz inequality, we have the following.
		\begin{align}
			\nonumber \mathbb{E} & \l| \int_{\tau_n}^{\tau_n + \theta}  \l( 1 + \l| m_n(s) \r|_{\mathbb{R}^3}^2 \r) m_n(s)   \, ds  \r|_{X^{-\beta}} \\
			\nonumber \leq & \mathbb{E} \l| \int_{\tau_n}^{\tau_n + \theta}   \l| m_n(s) \r|_{\mathbb{R}^3}^2  m_n(s)   \, ds  \r|_{X^{-\beta}}  
			+ \mathbb{E} \l| \int_{\tau_n}^{\tau_n + \theta}   m_n(s)   \, ds  \r|_{X^{-\beta}} \\
			\nonumber \leq & C \mathbb{E}  \int_{\tau_n}^{\tau_n + \theta}  \l| \l| m_n(s) \r|_{\mathbb{R}^3}^2  m_n(s) \r|_{L^{\frac{4}{3}}}  \, ds    
			+ C \mathbb{E}  \int_{\tau_n}^{\tau_n + \theta}  \l| m_n(s) \r|_{L^{2}}  \, ds   \\
			\nonumber \leq & C \mathbb{E}  \int_{\tau_n}^{\tau_n + \theta}  \l|  m_n(s) \r|_{L^{4}}^3  \, ds    
			+ C \mathbb{E}  \int_{\tau_n}^{\tau_n + \theta}  \l| m_n(s) \r|_{L^{2}}  \, ds   \\
			\nonumber \leq & C \theta^{\frac{1}{4}} \mathbb{E} \l( \int_{0}^{T}  \l|  m_n(s) \r|_{L^{4}}^4  \, ds \r)^{\frac{3}{4}}    
			+ C \theta^{\frac{1}{2}} \mathbb{E}  \l( \int_{0}^{T}  \l| m_n(s) \r|_{L^{2}}^2  \, ds \r)^{\frac{1}{2}} \\
			\leq & C\l( \theta^{\frac{1}{4}} + \theta^{\frac{1}{2}} \r).
		\end{align}		
		\textbf{Calculations for $I_4$:}
		Using the continuous embedding $X^{-\beta} \hookrightarrow L^{2}$ along with the Cauchy-Schwartz inequality, we have the following.
		\begin{align}
			\nonumber \mathbb{E} \l| \int_{\tau_n}^{\tau_n + \theta}  b_n(m_n(s))  \, ds  \r|_{X^{-\beta}} \nonumber \leq  & C \mathbb{E}  \int_{\tau_n}^{\tau_n + \theta}  \l| b_n(m_n(s)) \r|_{L^2} \, ds   \\
			\nonumber \leq & C \theta^{\frac{1}{2}} \l( \mathbb{E}  \int_{\tau_n}^{\tau_n + \theta}  \l| b_n(l,m_n(s)) \r|_{L^2}^2 \, ds \r)^{\frac{1}{2}}  \\
			\nonumber \leq & \theta^{\frac{1}{2}} C \l( \mathbb{E}  \int_{0}^{T}  \l( 1 + \l| m_n(s) \r|_{L^2}^2 \r) \, ds \r)^{\frac{1}{2}}  \\
			\leq & \theta^{\frac{1}{2}} C.
		\end{align}		
		\textbf{Calculations for $I_5$:}		
		By the  It\^o-L\'evy Isometry, we have
		\begin{align}
			\nonumber \mathbb{E} \l[ \l| \int_{\tau_n}^{\tau_n + \theta} \int_{B} G_n(l,m_n(s)) \, \tilde{\eta}(dl,ds) \r|_{X^{-\beta}}^2 \r] \leq &
			C \mathbb{E} \l[ \l| \int_{\tau_n}^{\tau_n + \theta} \int_{B} G_n(l,m_n(s)) \, \tilde{\eta}(dl,ds) \r|_{L^2}^2 \r] \\
			\nonumber = & C \mathbb{E} \l[  \int_{\tau_n}^{\tau_n + \theta} \int_{B} \l| G_n(l,m_n(s)) \r|_{L^2}^2 \, \nu(dl) \, ds  \r] \\
			\nonumber \leq & C \mathbb{E} \l[  \int_{\tau_n}^{\tau_n + \theta} \l( 1 + \l| m_n(s) \r|_{L^2}^2 \r) \, ds  \r] \\
			\nonumber \leq & \theta C \mathbb{E} \l[  \sup_{t\in[0,T]} \l( 1 + \l| m_n(s) \r|_{L^2}^2 \r) \, ds  \r] \\
			\leq & \theta C.
		\end{align}
		Combining the above estimates completes the proof for the lemma.

	\end{proof}

	\begin{lemma}\label{lemma tightness lemma}
		The sequence of laws $\l\{\mathcal{L}(m_n)\r\}_{n\in\mathbb{N}}$ is tight on the space $\mathcal{Z}_T$.
	\end{lemma}
	\begin{proof}[Proof of Lemma \ref{lemma tightness lemma}]
		It suffices to show that the sequence $\l\{m_n\r\}_{n\in\mathbb{N}}$ satisfies the Aldous condition (Definition \ref{definition Aldous Condition}) on the space $X^{-\beta},\ \beta > \frac{1}{4}$, which follows from Lemma \ref{lemma Aldous Condition for mn}.
	\end{proof}

	Lemma \ref{lemma Aldous Condition for mn} shows that the sequence of laws $\l\{\mathcal{L}(m_n)\r\}_{n\in\mathbb{N}}$ is tight on the space $\mathcal{Z}_T$. There are two parts to the proof.\\
	\textbf{Part 1:} Aldous condition on the space $X^{-\beta}$ (see Definition \ref{definition Aldous Condition}, Lemma \ref{lemma Aldous Condition for mn}) gives the tightness of the sequence of the laws of $m_n$ on the space $\mathbb{D}([0,T]:X^{-\beta})$. In particular, the Aldous condition gives equicontinuity (condition $(b)$ in Theorem \ref{Theorem general tightness 1}) for the sequence $\l\{m_n\r\}_{}n\in\mathbb{N}$ in the space $\mathbb{D}([0,T]:X^{-\beta})$, which along with the uniform bound on the space $L^{\infty}(0,T:H^1)$ (Lemma \ref{lemma bounds 2}) gives the compact embedding of a closed ball in the space $L^{\infty}(0,T:H^1)$ into the space $Z_T = \mathbb{D}([0,T]:X^{-\beta})\cap L^p(0,T:L^q)\cap\mathbb{D}([0,T]:H^1_{\text{w}})$.\\
	\textbf{Part 2:} Tightness of the sequence of laws of $m_n$ on $L^2_{\text{w}}(0,T:H^2)$ follows from the fact that closed balls in $L^2(0,T:H^2)$ are relatively compact in the space $L^2_{\text{w}}(0,T:H^2)$.\\
	Combining parts 1 and 2 gives the desired tightness of the sequence of laws of $m_n$ on the space $\mathcal{Z}_T$, which can also be written as $Z_T \cap L^2_{\text{w}}(0,T:H^2)$.

	In metric spaces, Prokhorov Theorem \cite{Parthasarathy_1967_Book_ProbabilityMeasuresOnMetricSpaces_Book} and Skorohod Theorem \cite{Billingsley_1999_ConvergenceOfProbabilityMeasures_Book} can be used to obtain convergence from the tightness. Since $Z_T$ is a nonmetrizable locally convex space, we use the following generalization to nonmetric spaces .

	\begin{theorem}[Skorohod-Jakubowski, \cite{ZB+UM_WeakSolutionSLLGE_JumpNoise} ]\label{THeorem Skorohod Jakubowski}
		Let $\mathcal{X}$ be a topological space such that there exists a sequence of continuous functions $f_m : \mathcal{X}\to\mathbb{C}$ that separates points of $\mathcal{X}$. Let $\mathcal{V}$ be the $\sigma$-algebra generated by $\l\{ f_m \r\}_{m\in\mathbb{N}}$. Then we have the following assertions.
		\begin{enumerate}
			\item Every compact set $\mathcal{K}\subset \mathcal{X}$ is metrizable.
			\item Let $\mu_n$ be a tight sequence of probability measures on $\l( \mathcal{X} , \mathcal{V} \r)$. Then there exists a subsequence $\l\{\mu_{n_{k}}\r\}_{k\in\mathbb{N}}$, random variables $X_{k},X$ for $k\in\mathbb{N}$ on a common probability space $\l( \bar{\Omega} , \bar{\mathbb{F}} , \bar{\mathbb{P}}\r)$ with $\bar{\mathbb{P}}^{X_{k}} = \mu_k$ for each $k\in\mathbb{N}$ and $X_k\to X$ $\bar{\mathbb{P}}$-a.s. as $k\to\infty$.
		\end{enumerate}
	\end{theorem}

	In order to talk about the convergence (of laws) for the pair $\l( m_n , \eta \r)$, we use a generalization of the Skorohod Theorem (see \cite{ZB+EH+Paul_2018_StochasticReactionDiffusion_JumpProcesses}). We use the said result to obtain another sequence of $\mathcal{Z}_T$-valued random variables (possibly on a different probability space). Let $\eta_n := \eta,n\in\mathbb{N}$. With this notation and Lemma \ref{lemma tightness lemma} and Theorem \ref{THeorem Skorohod Jakubowski}, we can conclude that the sequence $\l\{ \mathcal{L}\l(m_n,\eta_n\r)\r\}_{n\in\mathbb{N}}$ is tight on the space $\mathcal{Z}_T \times \mathbb{M}_{\bar{\mathbb{N}}}\l([0,T] \times B\r)$, where $\mathbb{M}_{\bar{\mathbb{N}}}(\mathcal{S})$ denotes all the $\mathbb{N}\cup\{\infty\}$-valued measures on the space $\l(\mathcal{S},\mathscr{\varrho}\r)$ (see for instance \cite{ZB+UM_WeakSolutionSLLGE_JumpNoise}).
	By applying a generalized Jakubowski-Skorohod Representation Theorem \cite{ZB+EH+Paul_2018_StochasticReactionDiffusion_JumpProcesses}, we have the following lemma.

	\begin{lemma}\label{lemma use of Skorohod Theorem}
		There exists a probability space $$\l( \Omega^\p , \mathbb{F}^\p , \mathcal{F}^\p , \mathbb{P}^\p\r)$$ and another sequence of random variables on that space $\{ \m_n \}_{n\in\mathbb{N}}$, along with $\mathcal{Z}_T \times \mathbb{M}_{\bar{\mathbb{N}}}\l([0,T] \times B\r)$-valued random variables $\l(\m,\eta^\p\r) , \l(\m_n,\eta_n^\p\r),n\in\mathbb{N}$ on the aforementioned probability space, such that the following hold.
		\begin{enumerate}
			\item \begin{equation}
				\mathcal{L}\l( \m_n,\eta_n^\p \r) = \mathcal{L}\l( \m_n,\eta_n^\p \r), \forall n\in\mathbb{N},
			\end{equation}
			\item \begin{equation}
				\l( \m_n,\eta_n^\p \r) \to \l( \m,\eta_n^\p \r)\ \text{in}\ \mathcal{Z}_T \times \mathbb{M}_{\bar{\mathbb{N}}}\l([0,T] \times B\r),
			\end{equation}
			\item \begin{equation}
				\eta_n^\p(\omega^\p) = \eta^\p(\omega^\p),\ \text{for all} \  \omega^\p\in\Omega^\p.
			\end{equation}
			
		\end{enumerate}
	\end{lemma}
	
	For the remainder of the section, we fix $p = q = 4$ and $\beta = \frac{1}{2}$. Note that $X^{-\frac{1}{2}}$ can be identified with the space $(H^1)^\p$.

	\subsubsection{\textbf{Properties of the limiting process:}}\label{Section Properties of the limiting process}
	The space $\mathbb{D}([0,T]:H_n)$ can be defined in the spirit of Section \ref{Section Some Functional Spaces}.
	
	\begin{remark}
		The set $\mathbb{D}([0,T]:H_n)$ is a Borel subset of $\mathcal{Z}_T$ (see for example \cite{ZB+UM+AAP_2019_MartingaleSolutions_NematicLiquidCrystals_PureJumpNoise,UM+AAP_2021_LargeDeviationsSNSELevyNoise}). Therefore the newly obtained processes $\m_n$ satisfy the same bounds as the corresponding processes $m_n$.
	\end{remark}

	With the above remark in mind, we have the following lemma for the sequence $\l\{\m_n\r\}_{n\in\mathbb{N}}$.
	\begin{lemma}\label{lemma bounds 1 lemma mn prime}
		Let $p\geq 1$. Then, there exists a constant $C>0$, which can depend on $p$ but not on $n\in\mathbb{N}$ such that the following hold.
		\begin{equation}
			\mathbb{E}^\p \sup_{t\in[0,T]}\l| \m_n(t) \r|_{H^1}^{2p} \leq C,
		\end{equation}
		\begin{equation}
			\mathbb{E}^\p \l( \int_{0}^{t} \l| \m_n(t) \r|_{H^2}^2 \, ds \r)^{2p} \leq C.
		\end{equation}
		Here $\mathbb{E}^\p$ denotes the expectation with respect to the probability space $\l(\Omega^\p,\mathbb{P}^\p\r)$.
	\end{lemma}
	Similarly, we can show the following bounds for the limit process $\m$, in particular for $p = 1$.
	\begin{lemma}\label{Lemma FG approximations bounds on m prime}
		There exists a constant $C>0$ such that for every $n\in\mathbb{N}$, the following hold.
	
		\begin{equation}
			\mathbb{E}^\p \sup_{t\in[0,T]}\l| \m(t) \r|_{H^1}^{2} \leq C,
		\end{equation}
		\begin{equation}
			\mathbb{E}^\p \int_{0}^{T} \l| \m(t) \r|_{H^2}^2 \, dt  \leq C.
		\end{equation}
	\end{lemma}
	The above lemma shows that the obtained process $\m$ satisfies the bounds \eqref{eqn weak martingale solution definition bound 1} and \eqref{eqn weak martingale solution definition bound 2} in $(3)$ of Definition \ref{definition weak martingale solution}.
	We recall that by \ref{lemma use of Skorohod Theorem}, we have the following convergence $\mathbb{P}^\p$-a.s.
	\begin{equation}
		\m_n\to\m\ \text{in}\ \mathcal{Z}_T.
	\end{equation}
	See \eqref{eqn definition of the space ZT} for the definition of the space $\mathcal{Z}_T$.

	\subsubsection{\textbf{Convergence of the Approximate Solutions:}}\label{Section Convergence of the Approximate Solutions}
	
	The aim of this subsection is to show that the obtained process $\m$ is a weak martingale solution for the problem \eqref{eqn FG approximation of problem considered Marcus Form}. We skip the details and refer the reader to analogous calculations in  \cite{ZB+BG+Le_SLLBE,ZB+UM_WeakSolutionSLLGE_JumpNoise,UM+AAP_2021_LargeDeviationsSNSELevyNoise} for details. We give a structure of the proof and some arguments.

	\begin{lemma}\label{lemma convergence lemma 1 of FG approximates}
		The following convergences hold.
		\begin{enumerate}
			\item \begin{equation}\label{eqn L4L4L4 convergence mn prime}
				\m_n\to\m\ \text{in}\ L^4\l( \Omega^\p : L^4(0,T:L^4)\r),
			\end{equation}
			\item 
			
			\begin{equation}
				\m_n \to  \m\ \text{weakly in}\ L^2\l( \Omega^\p : L^2(0,T : H^2 )\r),
			\end{equation}

			\item For $r\in (1,\frac{4}{3})$,
			\begin{equation}
				\m_n \times \Delta \m_n \to \m \times \Delta \m\ \text{weakly in}\ L^r(\Omega^\p:L^r(0,T:L^2)),
			\end{equation}
			
			and for $\beta>\frac{1}{4}$
			\begin{equation}
				\m_n \times \Delta \m_n \to \m \times \Delta \m\ \text{weakly in}\  L^2(\Omega^\p:L^2(0,T:X^{-\beta})).
			\end{equation}

			\item  Let $V\in L^4(\Omega^\p:(0,T:H^1))$. Then
			\begin{equation}
				\lim_{n\to\infty}\mathbb{E} \int_{0}^{T} \l\langle P_n \l[ \l( 1 + \l| \m_n(s) \r|_{\mathbb{R}^3}^2 \r) \m_n(s) \r] 
				-\l( 1 + \l| \m(s) \r|_{\mathbb{R}^3}^2 \r) \m(s)  , V(s) \r\rangle_{L^2} \, ds = 0.
			\end{equation}
		\end{enumerate}
	\end{lemma}
	We skip the proof of this lemma and refer the reader to Section 5 in \cite{ZB+BG+Le_SLLBE} (see also \cite{ZB+BG+TJ_Weak_3d_SLLGE,ZB+BG+TJ_LargeDeviations_LLGE,ZB+UM_WeakSolutionSLLGE_JumpNoise,UM+AAP_2021_LargeDeviationsSNSELevyNoise,ZB+Dhariwal_2012_SNSE_LevyNoise})

	\begin{lemma}\label{lemma convergence for bn existence of weak martingale solution}
		Let $V\in L^4(\Omega^\p : L^2)$. Then the following convergence holds.
		\begin{equation}\label{eqn convergence for bn existence of weak martingale solution}
			\lim_{n\to\infty} \mathbb{E}^\p \l| \int_{0}^{T} \l\langle b_n\bigl(\m_n(s)\bigr) - b\bigl(\m(s)\bigr) , V \r\rangle_{L^2} \, ds \r|^2 = 0.
		\end{equation}
	\end{lemma}
	\begin{proof}[Proof of Lemma \ref{lemma convergence for bn existence of weak martingale solution}]
		It suffices to show that
		\begin{equation}
			\lim_{n\to\infty} \mathbb{E}^\p \l| \int_{0}^{T} \l\langle b\bigl(\m_n(s)\bigr) - b\bigl(\m(s)\bigr) , V \r\rangle_{L^2} \, ds \r|^2 = 0.
		\end{equation}

		\begin{align}
			\nonumber \mathbb{E}^\p \l| \int_{0}^{T} \l\langle b\big(\m_n(s)\big) - b\big(\m(s)\big) , V \r\rangle_{L^2} \r|^2\, ds 
			\nonumber \leq & \mathbb{E}^\p \l| \int_{0}^{T} \l| b\big(\m_n(s)\big) - b\big(\m(s)\big) \r|_{L^2} \l| V \r|_{L^2}   \, ds \r|^2 \\
			\nonumber \leq & C \l( \mathbb{E}^\p \l| V \r|_{L^2}^2 \r)^{\frac{1}{2}} 
			\l( \mathbb{E}^\p \int_{0}^{T} \l| b(\m_n(s)) - b\big(\m(s)\big) \r|_{L^2}^4 \, ds \r)^{\frac{1}{2}} \\
			\leq & C \l( \mathbb{E}^\p \int_{0}^{T} \l| \m_n(s) - \m(s) \r|_{L^2}^4 \, ds \r)^{\frac{1}{2}} .
		\end{align}
		The right hand side of the above inequality goes to $0$ as $n\to\infty$, thus concluding the proof of the lemma.
	\end{proof}
	
	\begin{lemma}\label{lemma convergence for Gn existence of weak martingale solution}
		Let $V\in L^4(\Omega^\p:L^2)$. Then the following convergence holds.
		\begin{equation}\label{eqn convergence for Gn existence of weak martingale solution}
			\lim_{n\to\infty} \mathbb{E}^\p\l[ \int_{0}^{T} \int_{B} \bigl| \l\langle G_n(\m_n(s)) - G(\m(s)) , V \r\rangle_{L^2} \bigr|^2 \, \nu(dl) \, ds \r] = 0.
		\end{equation}
	\end{lemma}
	\begin{proof}[Proof of Lemma \ref{lemma convergence for Gn existence of weak martingale solution}]
		As before, it suffices to show the calculations for $G(\m_n) - G(\m)$, that is, not considering the projection operator $P_n$.
		\begin{align}
			\nonumber & \mathbb{E}^\p\l[ \int_{0}^{T} \int_{B} \l| \l\langle G\big(\m_n(s)\big) - G\big(\m(s)\big) , V \r\rangle_{L^2} \r|^2 \, \nu(dl) \, ds \r] \\
			\nonumber \leq & \mathbb{E}^\p\l[ \int_{0}^{T} \int_{B} \l|  G_n\big(\m_n(s)\big) - G\big(\m(s)\big) \r|_{L^2}^2 \l| V \r|_{L^2}^2 \, \nu(dl) \, ds \r] \\
			\nonumber \leq & \mathbb{E}^\p\l[ \l| V \r|_{L^2}^2 \int_{0}^{T} \l|  \m_n(s) - \m(s) \r|_{L^2}^2  \int_{B} C(l)   \, \nu(dl) \, ds \r] \\
			\nonumber \leq & C \l[ \mathbb{E}^\p \l| V \r|_{L^2}^4 \r]^{\frac{1}{2}} \mathbb{E}^\p \l[\int_{0}^{T} \int_{B} C(l)^2 \l|  \m_n(s) - \m(s) \r|_{L^2}^4  \, \nu(dl) \, ds \r] \\
			\leq & \l[ \mathbb{E}^\p \int_{0}^{T} \l|  \m_n(s) - \m(s) \r|_{L^2}^4  \, ds \r]^{\frac{1}{2}}. 
		\end{align}
		The right hand side converges to $0$ as $n\to\infty$.
	\end{proof}
	The above lemma, combined with the It\^o-L\'evy Isometry concludes the convergence for the stochastic integral. That is,
	\begin{equation}
		\lim_{n\to\infty} \mathbb{E}^\p\l| \int_{0}^{T} \int_{B}  \l\langle G_n(\m_n(s)) - G(\m(s)) , V \r\rangle_{L^2}  \, \eta^\p(dl,ds) \r|^2 = 0.
	\end{equation}

	\subsubsection{\textbf{Identifying the Driving Jump Process:}}\label{Section Identifying the Driving Jump Process}

	We give a brief proof. Similar details can be found in Section 6.3 in \cite{ZB+UM_WeakSolutionSLLGE_JumpNoise}.
	
	Let $V\in L^4(\Omega^\p:H^1)$. Let us set the following notations for this subsection.
	\begin{align}
		\nonumber M_n(\m_n , \tilde{\eta}^\p_n , V)(t) : = & \int_{0}^{t} \bigg[ \l\langle \m_n(0) , V \r\rangle_{X^{\beta}} 
		+ \l\langle \Delta \m_n(s) , V \r\rangle_{L^2}
		+ \l\langle m_n(s) \times \Delta \m_n(s) , V \r\rangle_{L^2} \\
		\nonumber & - \l\langle \l( 1 + \l| \m_n(s) \r|^2\r) \m_n(s) , V \r\rangle_{H^1}  \bigg] \, ds
		+ \int_{0}^{t} \int_{B} \l\langle G\big(l,\m_n(s)\big) , V \r\rangle_{L^2} \tilde{\eta}^\p_n(dl , ds)\\
		& + \int_{0}^{t} \l\langle b\big(\m_n(s)\big) , V \r\rangle_{L^2} \, ds,
	\end{align}
	and
	\begin{align}
		\nonumber M(\m , \tilde{\eta}^\p , V)(t) : = & \int_{0}^{t} \bigg[ \l\langle \m(0) , V \r\rangle_{X^{\beta}} 
		+ \l\langle \Delta \m(s) , V \r\rangle_{L^2}
		+ \l\langle m(s) \times \Delta \m(s) , V \r\rangle_{L^2} \\
		\nonumber & - \l\langle \l( 1 + \l| \m(s) \r|^2\r) \m(s) , V \r\rangle_{H^1}  \bigg] \, ds
		+ \int_{0}^{t} \int_{B} \l\langle G\big(l,\m(s)\big) , V \r\rangle_{L^2} \tilde{\eta}^\p(dl , ds)\\
		& + \int_{0}^{t} \l\langle b\big(\m(s)\big) , V \r\rangle_{L^2} \, ds.
	\end{align}
	Following Section 6 in \cite{ZB+UM_WeakSolutionSLLGE_JumpNoise}, we write two results in the form of two steps.
	\begin{enumerate}
		\item[\textbf{Step 1:}]
		\begin{enumerate}
			\item[(a)] 
			\begin{equation}
				\lim_{n\to\infty} \mathbb{E}^\p \l[ \int_{0}^{T} \l| \m_n(t) - \m(t) \r|_{L^2}^2 \, dt \r]     = 0.
			\end{equation}
			This follows from Lemma \ref{lemma convergence lemma 1 of FG approximates} and the continuous embedding $L^4\hookrightarrow L^2$.
			
			\item[(b)]
			\begin{equation}
				\lim_{n\to\infty} \mathbb{E}^\p \l[ \int_{0}^{T} \l| M_n(\m_n , \tilde{\eta}^\p_n , V)(t) - M(\m , \tilde{\eta}^\p , V)(t) \r|^2 \, dt \r]     = 0.
			\end{equation}
			This is a combination of Lemma \ref{lemma convergence lemma 1 of FG approximates}, Lemma \ref{lemma convergence for bn existence of weak martingale solution} and Lemma \ref{lemma convergence for Gn existence of weak martingale solution}.
		\end{enumerate}
		
		\item[\textbf{Step 2:}]
		For each $t\in[0,T]$, $\mathbb{P}^\p$-a.s., the following equality holds.
		\begin{equation}
			\l\langle \m(t) , V \r\rangle_{H^1} = M(\m, \tilde{\eta}^\p , V)(t).
		\end{equation}
	\end{enumerate}
	\textbf{Conclusion.} The above equality shows that the process $\m$ satisfies the equality \eqref{eqn weak martingale solution definition eqn weak form}. Further, Lemma \ref{Lemma FG approximations bounds on m prime} implies that the obtained process $\m$ satisfies the bounds \eqref{eqn weak martingale solution definition bound 1} and \eqref{eqn weak martingale solution definition bound 2}. Combining the above steps, we can conclude (see \cite{ZB+UM_WeakSolutionSLLGE_JumpNoise,ZB+BG+TJ_Weak_3d_SLLGE,UM+AAP_2021_LargeDeviationsSNSELevyNoise}) that the obtained limit $\m$ is a weak martingale solution of the problem \eqref{eqn problem considered Marcus Form}.	
	This concludes the proof of Theorem \ref{theorem existence of weak martingale solution}.
	
	\section{Pathwise Uniqueness and Existence Of Strong Solution}\label{Section Pathwise Uniqueness and Existence Of Strong Solution}
	In the following section, we restrict ourselves to $d=1,2$. The pathwise uniqueness will be shown in a later section for a more general equation.	
	Before moving to the pathwise uniqueness part (Theorem \ref{theorem pathwise uniqueness existence of weak martingale solution}), we state a result that occurs as a corollary of the bounds established in Theorem \ref{theorem existence of weak martingale solution}. We refer the reader back to Remark \ref{Remark weak form of Laplacian and cross product terms} to see how the terms are to be understood in the space $(H^1)^\p$.
	
	\begin{corollary}\label{corollary weak martingale solution PDE strong sense}
		Let $\l( \Omega , \mathbb{F} , \mathcal{F} , \mathbb{P} , m , \eta \r)$ denote a weak martingale solution to \eqref{eqn problem considered Marcus Form}. Then the equation \eqref{eqn problem considered Marcus Form} makes sense in $(H^1)^\p$.
	\end{corollary}
	\begin{theorem}\label{theorem pathwise uniqueness existence of weak martingale solution}
		Let $\l( \Omega , \mathbb{F} , \mathcal{F} , \mathbb{P} , m_i , \eta \r)$, $i=1,2$ be two weak martingale solutions (on the same probability space) to the problem \eqref{eqn problem considered Marcus Form}. Then for each $t\in[0,T]$, we have
		\begin{equation}
			m_1(t) = m_2(t),\ \mathbb{P}-a.s.
		\end{equation}
	\end{theorem}
	\begin{proof}
		We do not give proof for this theorem. A more general result will be shown in a later section (see Theorem \ref{theorem pathwise uniqueness stochastic control equation}).
	\end{proof}
	Pathwise uniqueness shown in Theorem \ref{theorem pathwise uniqueness existence of weak martingale solution},  the existence of a martingale solution (Theorem \ref{theorem existence of weak martingale solution}) combined with the theory of Yamada and Watanabe \cite{Ikeda+Watanabe} gives us the existence a strong solution
	to the problem \eqref{eqn problem considered Marcus Form}, which is made formal in the following theorem.
	\begin{theorem}\label{theorem existence of strong solution}
		The problem \eqref{eqn problem considered Marcus Form} admits a strong solution.
	\end{theorem}

	\section{The Large Deviations Principle}\label{section The Large Deviations Principle}
	We remind the reader that we will be restricting ourselves to dimensions $1$ and $2$ (i.e. $d = 1,2$).
	\subsection{Background, Definitions, Formulation of the Problem}
	\begin{definition}[Good rate function]\label{definition rate function}
		Let $Z$ be a Polish space. A function $I:Z\to[0,\infty]$ is called a \textbf{rate function} if $I$ is lower semicontinuous. A rate function $I$ is a good rate function if for arbitrary $M\in[0,\infty)$, the level set $K_M = \{ x : I(x) \leq M \}$ is a compact subset of $Z$.
	\end{definition}

	\begin{definition}[Large deviations principle]\label{definition LDP}
		A family of probability measures $\{ \mathbb{P}_{\varepsilon} : \varepsilon > 0 \}$ is said to satisfy the Large Deviations Principle (LDP) on $Z$ with a good rate funciton $I:Z\to[0,\infty]$ if 
		\begin{enumerate}
			\item for each closed set $F_1\subset Z$,
			\begin{equation}
				\limsup_{\varepsilon\to0} \varepsilon \log \mathbb{P}_{\varepsilon}\l(F_1\r) \leq - \inf_{x\in F_1} I(x),
			\end{equation}
			\item for each open set $F_2\subset Z$,
			\begin{equation}
				\liminf_{\varepsilon\to0} \varepsilon \log \mathbb{P}_{\varepsilon}\l(F_2\r) \geq - \inf_{x\in F_2} I(x).
			\end{equation}
		\end{enumerate}
	\end{definition}	
	For $\varepsilon>0$, our aim is to establish Freidlin-Wentzell type large deviations of the solutions of the following equation on a given probability space $\l(\Omega,\mathcal{F},\mathbb{F},\mathbb{P}\r)$:
	\begin{align}\label{eqn problem considered with epsilon LDP with L}
		dm^{\varepsilon}(t) = \l( \Delta m^{\varepsilon}(t) + m^{\varepsilon}(t) \times \Delta m^{\varepsilon}(t) - \l( 1 + \l| m^{\varepsilon}(t) \r|^2 \r) m^{\varepsilon}(t) \r) \, dt + \varepsilon \bar{g}\l( m^{\varepsilon}(t) \r) \diamond dL^{\varepsilon^{-1}}(t),\ t\geq 0,
	\end{align}
	with $m^{\varepsilon}(0) = m_0$.\\
	We establish some notations here in order to introduce the stochastic control equation in Section \ref{The Stochastic Control Equation}. We refer the reader to \cite{Zhai+Zhang_2015_LargeDeviations_2DSNSE_MultiplicativeLevyNoise,UM+AAP_2021_LargeDeviationsSNSELevyNoise} for more details.
	
	Let $\eta$ denote the time homogeneous Poisson Random Measure (PRM) associated with $L$. Therefore
	\begin{equation*}
		\eta\bigl( [0,t] \times A \bigr) := \# \l\{ s\in[0,t] : L(s) - L(s-)\in A \r\}.
	\end{equation*}	
	If $\nu$ is the compensator, then
	\begin{equation*}
		t\nu(A) := \mathbb{E}\l[ \eta \bigl( [0,t] \times A \bigr) \r].
	\end{equation*}
	Let us fix $\varepsilon>0$ for now. The L\'evy process $L$ (with the continuous part 0) is scaled as follows.
	\begin{equation}
		L^{\varepsilon^{-1}}(t) := L(\varepsilon^{-1}t),\ t\geq 0.
	\end{equation}
	Let $\eta^{\varepsilon^{-1}}$ denote the PRM corresponding to the process $L^{\varepsilon^{-1}}$. One can show then that
	\begin{equation}
		\eta^{\varepsilon^{-1}} \l( [0,t] \times A \r) = \eta \bigl( [0,\varepsilon^{-1}t] \times A \bigr).
	\end{equation}
	Hence,
	\begin{equation}
		\nu^{\varepsilon^{-1}} = \varepsilon^{-1}\nu.
	\end{equation}
	The compensated PRM corresponding to $L^{\varepsilon^{-1}}$, denoted by $\tilde{\eta}$, is given by
	\begin{equation}
		\tilde{\eta}^{\varepsilon^{-1}} = \eta^{\varepsilon^{-1}} - t\nu^{\varepsilon^{-1}} = \eta^{\varepsilon^{-1}} - \varepsilon^{-1} t \nu.
	\end{equation}
	
	\subsection{Marcus Mapping for the Perturbed Equation}
	The subsection defines the Marcus mapping corresponding to the jump process (for the equation \eqref{eqn problem considered with epsilon LDP with L}).
	Let $\Phi^{\varepsilon}$ denote the solution to the ordinary differential equation
	\begin{equation}\label{eqn ODE for Marcus Mapping Phi epsilon}
		\frac{d\Phi^{\varepsilon}}{dt}\l(t,l,x\r) = \varepsilon l\bar{g}\l( \Phi^{\varepsilon}\l(t,l,x\r)\r),\ t\in\mathbb{R}_{+},\ l\in \mathbb{R}, x\in L^2,
	\end{equation}
	with $\phi\l(0,l,x\r) = x$.
	By change of variables, we can see that for $\varepsilon,t \geq 0$
	\begin{equation}
		\Phi^{\varepsilon}\l(t,l,x\r) = \Phi \l(\varepsilon t,l,x\r).
	\end{equation}
	We now write the equation \eqref{eqn problem considered with epsilon LDP with L} with the help of the Marcus mapping $\Phi^{\varepsilon}$.
	
	\begin{align}\label{eqn problem considered with epsilon Marcus Form 1}
		\nonumber dm^{\varepsilon}(t) = & \l( \Delta m^{\varepsilon}(t) + m^{\varepsilon}(t) \times \Delta m^{\varepsilon}(t) - \l( 1 + \l| m^{\varepsilon}(t) \r|^2 \r) m^{\varepsilon}(t) \r) \, dt \\
		\nonumber & + \int_{B} \bigl[ \Phi^{\varepsilon}\bigl( 1, l, m^{\varepsilon}(t) \bigr) - m^{\varepsilon}(t) \bigr] \, \tilde{\eta}^{\varepsilon^{-1}}(dl,dt) \\
		& + \int_{B} \bigl[ \Phi^{\varepsilon}\bigl( 1, l, m^{\varepsilon}(t) \bigr) - m^{\varepsilon}(t) - \varepsilon l\bar{g}\bigl( m^{\varepsilon} (t) \bigr)\bigr] \, \nu^{\varepsilon^{-1}}(dl)\,dt,
	\end{align}
	with $m^{\varepsilon}(0) = m_0$.
	The above equation can be written as follows.
	\begin{align}\label{eqn problem considered with epsilon Marcus Form 2}
		\nonumber dm^{\varepsilon}(t) = & \l( \Delta m^{\varepsilon}(t) + m^{\varepsilon}(t) \times \Delta m^{\varepsilon}(t) - \l( 1 + \l| m^{\varepsilon}(t) \r|^2 \r) m^{\varepsilon}(t) \r) \, dt \\
		\nonumber & + \int_{B} \bigl[ \Phi \bigl( \varepsilon, l, m^{\varepsilon}(t) \bigr) - m^{\varepsilon}(t) \bigr] \, \tilde{\eta}^{\varepsilon^{-1}}(dl,dt) \\
		& + \int_{B} \l[ \Phi\l( \varepsilon, l, m^{\varepsilon}(t) \r) - m^{\varepsilon}(t) - \varepsilon l\bar{g}\l( m^{\varepsilon} (t)  \r)\r] \, \nu^{\varepsilon^{-1}}(dl)\,dt.
	\end{align}
	We define the following operators similarly as done before in Section \ref{section Existence of a Weak Martingale Solution}.
	\begin{align*}
		G\l( \varepsilon , l , v \r) & = \Phi\l( \varepsilon , l , v \r) - v ,\\
		H\l( \varepsilon , l , v \r) & = \Phi\l( \varepsilon , l , v \r) - v - l\varepsilon\bar{g}(v),\\
		b(\varepsilon , v) & = \int_{B} H\l( \varepsilon , l , v \r) \, \nu^{\varepsilon^{-1}}(dl).
	\end{align*}
	Note that
	\begin{equation}
		b(\varepsilon , v) = \int_{B} H\l( \varepsilon , l , v \r) \, \nu^{\varepsilon^{-1}}(dl) = \varepsilon^{-1} \int_{B} H\l( \varepsilon , l , v \r) \, \nu(dl).
	\end{equation}
	Effectively, equation \eqref{eqn problem considered with epsilon Marcus Form 2} can be written as
	
	\begin{align}\label{eqn problem considered with epsilon Marcus Form 3}
		\nonumber m^{\varepsilon}(t) = & m_0 + \int_{0}^{t} \l( \Delta m^{\varepsilon}(s) + m^{\varepsilon}(s) \times \Delta m^{\varepsilon}(s) - \l( 1 + \l| m^{\varepsilon}(s) \r|^2 \r) m^{\varepsilon}(t) \r) \, ds \\
		& + \int_{0}^{t} \int_{B} \bigl[ G \bigl( \varepsilon , l , m^{\varepsilon}(s) \bigr) \bigr] \, \tilde{\eta}^{\varepsilon^{-1}}(dl,ds) 
		+ \varepsilon^{-1} \int_{0}^{t} b \bigl(\varepsilon , m^{\varepsilon}(s) \bigr) \,ds \ t\geq 0.
	\end{align}

	\begin{remark}
		A definition for a weak martingale solution for the problem \eqref{eqn problem considered with epsilon Marcus Form 3} ( for $\varepsilon>0$ ) can be given in the spirit of Definition \ref{definition weak martingale solution}. In fact, \eqref{eqn problem considered Marcus Form} is same as \eqref{eqn problem considered with epsilon Marcus Form 3} with $\varepsilon = 1$. Theorem \ref{theorem existence of weak martingale solution}, Theorem \ref{theorem existence of strong solution} can be generalized for the equation \eqref{eqn problem considered with epsilon Marcus Form 3} (replacing the equation \eqref{eqn problem considered Marcus Form}). Pathwise uniqueness for solutions is shown in Theorem \ref{theorem pathwise uniqueness stochastic control equation}.
	\end{remark}
	The following theorem is a consequence of pathwise uniqueness and the theory of Yamada and Watanabe (see \cite{Ondrejat_thesis} for instance).
	
	\begin{theorem}\label{theorem existence of strong solution and J epsilon}
		Let the assumptions of Theorem \ref{theorem existence of weak martingale solution} hold (with $d = 1,2$). Then the problem \eqref{eqn problem considered with epsilon Marcus Form 3} has a strong solution (which is also unique in law) in the following sense.
		\begin{enumerate}
			\item For any $\varepsilon>0$, if $\l( \Omega , \mathcal{F} , \mathbb{F} , m , \eta^{\varepsilon^{-1}}\r)$
			and $\l( \Omega^\p , \mathcal{F}^\p , \mathbb{F}^\p , m^\p , \l(\eta^{\varepsilon^{-1}}\r)^\p\r)$ are two martingale solutions to the problem \eqref{eqn problem considered with epsilon Marcus Form 3} then they have the same laws on $\mathbb{U}_T$.
			\item For every $\varepsilon>0$, there exists a Borel measurable function 
			\begin{equation}\label{eqn definition of J epsilon}
				J^{\varepsilon}: \mathbb{M}\to\mathbb{U}_T,
			\end{equation}
			such that the following holds.
			Let $\l(\Omega , \mathcal{F} , \mathbb{F} , \mathbb{P}\r)$ be a filtered probability space (satisfying the usual hypotheses). Let $\eta$ be an arbitrary $\mathbb{R}$ valued time homogeneous Poisson random measure defined on the said probability space and let $m^{\varepsilon}:= J^{\varepsilon}(\eta)$. Then $\l(\Omega , \mathcal{F} , \mathbb{F} , \mathbb{P} , m^{\varepsilon} , \eta^{\varepsilon^{-1}}\r)$ is a weak martingale solution to the problem \eqref{eqn problem considered with epsilon Marcus Form 3}.
		\end{enumerate}
	\end{theorem}

	\subsection{Notations, some spaces:}\label{Section Notations, Some Spaces} We recall that $B = B(0,1)\backslash\{0\}$. Let $B_T = [0,T] \times B$. Let $\mathbb{M} = \mathcal{M}(B_T)$ denote the space of all measures $\mathcal{V}$ on $\l( B_T,\mathcal{B}(B_T) \r)$ such that $\mathcal{V}(K)<\infty$, for any compact subset $K\subset B_T$.
	We endow $\mathbb{M}$ with the topology $\mathcal{T}\l(\mathbb{M}\r)$, which is the weakest topology such that for each $v\in C_C(B_T)$, the map
	\begin{equation}
		\mathcal{V} \mapsto \l(v,\mathcal{V}\r) : = \int_{B_T} v(l,s) \,  \mathcal{V}(dl,ds),
	\end{equation}
	is continuous. This topology can be metrized such that $\mathbb{M}$ is a Polish space.
	Let $X = B \times [0,\infty)$ and $X_T = B_T \times [0,\infty )$. The spaces $\bar{\mathbb{M}} = \mathcal{M}(X_T)$ and $\mathcal{T(\bar{\mathbb{M}})}$ are defined analogously.
	There exists a unique probability measure $\bar{\mathbb{P}}$ on $\l(\bar{\mathbb{M}},\mathcal{T}\l(  \bar{\mathbb{M}}  \r)\r)$ such that the canonical map $\bar{\eta}(\bar{m}) = \bar{m}$ is a Poisson random measure with intensity measure $\nu(dl)\,dt\,dr$. The corresponding compensated Poisson random measure is denoted by $\tilde{\bar{\eta}}$ and given by
	\begin{equation}
		\tilde{\bar{\eta}}\l(dt\,dl\,dr\r) := \bar{\eta} \l(dt\,dl\,dr\r) = \nu(dl)\,dt\,dr.
	\end{equation}
	Let $\bar{\mathcal{F}}$ denote the completion of $\sigma\l\{ \bar{\eta}\l( (0,s] \times D \r) : s\in[0,t],\ D\in\mathbb{X} \r\}$ under $\bar{\mathbb{P}}$.	
	Let $\bar{\mathbb{F}} : = \l\{\bar{\mathcal{F}}_t\r\}_{t\in[0,T]}$.	
	Let $\bar{\mathcal{P}}$ denote the $\bar{\mathbb{F}}$ predictable sigma field on $[0,T] \times \bar{\mathbb{M}}$, with the filtration $\l\{\bar{\mathcal{F}}_t\r\}_{t\in[0,T]}$ on $\l( \bar{\mathbb{M}} ,\mathcal{B}\l(\bar{\mathbb{M}}\r)\r)$.	
	Let $\bar{\mathcal{A}}$ denote the class of all $\l( \bar{\mathcal{P}} \times \mathcal{B}(B) \r)\backslash \mathcal{B}([0,\infty))$-measurable maps
	\begin{equation*}
		\phi: B_T\times \bar{\mathbb{M}} \to [0,\infty).
	\end{equation*}
	For $\phi\in\bar{\mathcal{A}}$, define a counting process $\eta_C^{\phi}$ (denoted by $\eta^{\phi}$ for brevity) on $B$ by
	\begin{equation}
		\eta^{\phi}\l( (0,t] \times D \r) = \int_{\l( (0,t] \times D \r)} \int_{(0,\infty)} \chi_{[0,\phi(s,l)]}(r) 
		\, \bar{\eta}\l(ds\,dl\,dr\r),\ t\in[0,T],\ D\in\mathcal{B}(B).
	\end{equation}	
	Similarly, we define
	\begin{equation}
		\tilde{\eta}^{\phi}\l( (0,t] \times D \r) = \int_{\l( (0,t] \times D \r)} \int_{(0,\infty)} \chi_{[0,\phi(s,l)]}(r) \, \tilde{\bar{\eta}}\l(ds\,dl\,dr\r),\ t\in[0,T],\ D\in\mathcal{B}(B).
	\end{equation}	
	$\eta^{\phi}$ is a controlled Poisson random measure. Clearly,
	\begin{equation}
		\tilde{\eta}^{\phi}\l( (0,t] \times D \r) = \eta^{\phi}\l( (0,t] \times D \r) - \int_{(0,t] \times D} \phi(s,l) \, \nu(dl) \, ds.
	\end{equation}	
	Let $K\in\mathbb{N}$.
	\begin{equation}
		\mathcal{S}^K := \l\{  \theta : B_T \to [0,\infty) : \mathcal{L}_T \leq K \r\},
	\end{equation}
	where
	\begin{equation}
		\mathcal{L}_T : = \int_{0}^{T} \int_{B} \biggl[ \biggl( \theta (t,l) \log \bigl(\theta (t,l)\bigr) \biggr) - \theta(t,l) + 1 \biggr] \nu(dl)\, dt.
	\end{equation}
	A function $\theta\in\mathcal{S}^K$ can be identified with a measure $\nu^{\theta}\in\mathbb{M}$, defined by
	\begin{equation}
		\nu^{\theta}(A_T) = \int_{A_T}\theta(t,l) \, \nu(dl) \, dt,\ A_T\in\mathcal{B}(B_T).
	\end{equation}
	That is,
	\begin{equation}
		\frac{\nu^{\theta}(dl,dt)}{\nu(dl)dt} = \theta.
	\end{equation}	
	This identification induces a topology on $\mathcal{S}^K$, under which $\mathcal{S}^K$ is a compact space. Let us denote
	\begin{equation}
		\mathbb{S} = \cup_{K\in\mathbb{N}}\, \mathcal{S}^K.
	\end{equation}
	For $\phi\in\mathbb{S}$, we have
	\begin{equation}
		\tilde{\eta}^{\phi}\l( A_T \r) = \eta^{\phi}\l( A_T \r) - \nu^{\phi}\l( A_T \r),\ A_T\in\mathcal{B}(B_T).
	\end{equation}
	The bounds (especially linear growth) established in Lemma \ref{lemma linear growth Lipschitz G and H} can also be shown for the operators $G(\varepsilon , l , \cdot),\ H(\varepsilon , l , \cdot)$. In particular, we make precise the dependence of the constants on $\varepsilon$ in the following lemma.
	\begin{lemma}\label{lemma linear growth G,H,b with epsilon}
		Let $\varepsilon>0$ and $l\in B$. Then there exists a constant $C>0$ such that the following inequalities hold.
		\begin{enumerate}
			\item \begin{equation}
				\l|G\l( \varepsilon , l , v \r) \r|_{L^2} \leq C \l| l \r| \l( e^{C\varepsilon} - 1 \r) \l( 1 + \l| v \r|_{L^2}\r),
			\end{equation}
			\item \begin{equation}
				\l| H\l( \varepsilon , l , v \r)\r|_{L^2} \leq C \l| l \r| \l( e^{C\varepsilon} - 1 \r) \l( 1 + \l| v \r|_{L^2}\r).
			\end{equation}
		\end{enumerate}
		Hence, there exists a constant $C>0$ such that
		\begin{equation}
			\l| b\l( \varepsilon , v \r)\r|_{L^2} \leq C \l( e^{C\varepsilon} - 1 \r) \l( 1 + \l| v \r|_{L^2}\r).
		\end{equation}
	\end{lemma}

	\begin{remark}
		We observe that in Lemma \ref{lemma linear growth Lipschitz G and H}, the operators $G(1 , \cdot , \cdot) ,\ H(1 , \cdot , \cdot)$ are a special case of the operators $G(\varepsilon , \cdot , \cdot) ,\ H(\varepsilon , \cdot , \cdot)$ defined above with $\varepsilon = 1$. Hence following the proof of Lemma \ref{lemma linear growth Lipschitz G and H}, we can conclude Lipschitz continuity for the above mentioned operators.
	\end{remark}

	\subsection{The Skeleton Equation (The Deterministic Control Equation)}
	
	Let $\theta\in\mathbb{S}$. We consider the equation
	\begin{align}\label{eqn skeleton equation}
		\nonumber dm^{\theta}(t) = & \l[ \Delta m^{\theta}(t) + m^{\theta}(t) \times \Delta m^{\theta}(t) - \l( 1 + \l| m^{\theta}(t) \r|_{\mathbb{R}^2}\r)m^{\theta}(t) \r] \, dt\\
		&+ \int_{B} l \bar{g}\l( m^{\theta}(t) \r) \l( \theta(t,l) - 1 \r) \nu(dl)dt,\ t\in[0,T],
	\end{align}
	with $m^{\theta}(0) = m_0$.	
	The following theorem outlines the criterion required for the skeleton equation \eqref{eqn skeleton equation} to admit a unique solution.

	\begin{theorem}\label{theorem unique solution for the skeleton equation}
	
		Let $\theta\in \mathbb{S}$, $0<T<\infty$ and $m_0\in H^1$. Then the equation \eqref{eqn skeleton equation} admits a unique solution in $C([0,T]:H^1)\cap L^2(0,T:H^2)$. In particular let $\theta\in\mathcal{S}^K$, for some $K\in\mathbb{N}$. Then there exists a constant $C_K>0$, which depends only on $T, K, h $ and the initial data $m_0$ such that
		\begin{equation}
			\int_{0}^{T} \l| m(t) \r|_{H^2}^2 \, dt \leq C_K,
		\end{equation}
		and
		\begin{equation}
			\sup_{t\in[0,T]} \l| m(t) \r|_{H^1}^2 \leq C_K.
		\end{equation}
		
	\end{theorem}
	We postpone the proof of this theorem to Appendix \ref{section Proof of Existence of a Solution for Skeleton Equation}.

	\begin{remark}\label{remark uniform bounds theta n}
		We observe in the proof (Section \ref{section Proof of Existence of a Solution for Skeleton Equation}) of Theorem \ref{theorem unique solution for the skeleton equation} that the bounds depend on $\int_{0}^{T} \int_{B} \l| \theta(t,l) \r| \, \nu(dl) \, dt$.
		By Proposition \ref{proposition uniform bound on theta SK} the bound  on $\int_{0}^{T} \int_{B} \l| \theta(t,l) \r| \, \nu(dl)$ depends only on $K\in\mathbb{N}$, where $\theta\in\mathcal{S}^K$. Hence it is uniformly bounded since $K$ is fixed. That is
		\begin{equation}
			\sup_{\theta\in\mathcal{S}^K} \int_{0}^{T} \int_{B} \l| \theta(t,l) \r| \, \nu(dl) < \infty.
		\end{equation}
	\end{remark}
	Let us denote by $\mathbb{U}_T$ the space $\mathbb{D}([0,T]:L^2)\cap L^2(0,T:H^1)$.\\
	\textbf{The Map $J^0$:} Define a map 
	\begin{equation*}
		J^0 : \mathbb{S} \to \mathbb{U}_T,
	\end{equation*}
	as follows. For $\theta\in\mathbb{S}$, $J^0(\theta) = m^{\theta}$ denotes the unique solution to \eqref{eqn skeleton equation}.
	Theorem \ref{theorem unique solution for the skeleton equation} in particular shows that the map $J^0$ is well defined.\\

	\subsection{The Stochastic Control Equation}\label{The Stochastic Control Equation}

	Let $\varepsilon>0$ and $\phi\in\bar{\mathcal{A}}$. We consider the following stochastic partial differential equation.
	\begin{align}\label{eqn stochastic control equation not simplified}
		\nonumber dm^{\varepsilon,\phi}(t) = & \l( \Delta m^{\varepsilon,\phi}(t) + m^{\varepsilon,\phi}(t) \times \Delta m^{\varepsilon,\phi}(t) - \l( 1 + \l| m^{\varepsilon,\phi} \r|_{\mathbb{R}^2} \r) m^{\varepsilon,\phi} \r)\, dt\\
		\nonumber & +  \int_{B} \l[ \Phi\l(\varepsilon , l , m^{\varepsilon,\phi}(t) \r) - m^{\varepsilon,\phi}(t) \r] \tilde{\eta}^{\varepsilon^{-1}\phi}(dl,dt) \\
		\nonumber & +  \int_{B} \l[ \Phi\l(\varepsilon , l , m^{\varepsilon,\phi}(t) \r) - m^{\varepsilon,\phi}(t) - \varepsilon  l \bar{g}\l(m^{ \varepsilon,\phi}(t) \r) \r] \nu^{\varepsilon^{-1},\phi}(dl) \, dt \\
		& + \varepsilon\int_{B} l \bar{g}\l( m^{\varepsilon,\phi}(t) \r) \l( \phi\l( t , l \r) - 1 \r)\nu^{\varepsilon^{-1}}(dl) \, dt,\ t\in[0,T].
	\end{align}
	The above equation \eqref{eqn stochastic control equation not simplified} can be written in the integral form and with the operators $G,H,b$ as follows. To simplify notation, we denote $m^{\varepsilon,\phi}(t)$ by $M(t)$.
	\begin{align}\label{eqn stochastic control equation}
		\nonumber M(t) = & m_0 + \int_{0}^{t} \l( \Delta M(s) + M(s) \times \Delta M(s) - \l( 1 + \l| M(s) \r|_{\mathbb{R}^3}^2 \r) M(s) \r) \, ds \\
		\nonumber & + \int_{0}^{t} \int_{B} G\l(\varepsilon , l , M(s)\r)\, \tilde{\eta}^{\varepsilon^{-1}\phi}(dl,ds) 
		+ \int_{0}^{t} \int_{B} G\l(\varepsilon , l , M(s)\r) \l( \phi(s,l) - 1 \r)\, \nu^{\varepsilon^{-1} }(dl) \, ds\\
		& + \varepsilon^{-1} \int_{0}^{t} b\l(\varepsilon , M(s)\r) \, ds.
	\end{align}
	Let $\{K_n\}_{n\in\mathbb{N}}$ be a sequence of compact sets such that $\cup_{n\in\mathbb{N}}K_n = B$.	
	Let
	\begin{equation*}
		\bar{\Omega} = \bar{\mathbb{M}},\ \bar{\mathcal{F}} = \mathcal{T}\l( \bar{\mathbb{M}} \r).
	\end{equation*}
	Let us denote
	\begin{align}
		\mathcal{A}_b : = & \bigg\{ \nonumber \phi\in\bar{\mathcal{A}}  : \phi\l( t,x,\omega \r)\in\l[\frac{1}{n},n\r] \text{ if } \l( t,x,\omega \r) \in [0,T] \times K_n \times \bar{\Omega} \\
		\text{ and } & \phi\l( t,x,\omega \r)  =  1 \text{ if } \l( t,x,\omega \r) \in [0,T] \times K_n^c \times \bar{\Omega} \bigg\}.
	\end{align}
	We now give a brief proof of the existence of a weak martingale solution to the above equation \eqref{eqn stochastic control equation}.
	
	\begin{theorem}\label{theorem existence of weak martingale solution stochastic control equation}
		Fix $0<T<\infty$. Let $\phi\in\bar{\mathcal{A}}_b$ be such that $\text{esssup}_{\omega\in\bar{\Omega}}\mathcal{L}_T(\phi)<\infty$. Then the problem \eqref{eqn stochastic control equation} admits a weak martingale solution.
	\end{theorem}
	\begin{proof}[Proof of Theorem \ref{theorem existence of weak martingale solution stochastic control equation}]
		Let $\phi\in\bar{\mathcal{A}}_b$. Let $\psi = \phi^{-1}$. Clearly, $\psi\in \bar{\mathcal{A}}$. Therefore there exists $n\in\mathbb{N}$ and a compact set $K_n$ such that
		\begin{align}
			\nonumber \psi \l( t, l , \omega \r) & \in \l[\frac{1}{n},n\r]\ \text{for}\ \l( t, l , \omega \r)\in [0,T] \times K_n \times \bar{\Omega}, \\
			\nonumber \text{and}&\\
			\psi \l( t, l , \omega \r) & = 1 \ \text{for}\ \l( t, l , \omega \r)\in [0,T] \times K_n^c \times \bar{\Omega},
		\end{align}

		\begin{align}
			\nonumber \mathcal{E}_t(\psi) : = & \exp \bigg\{ \int_{(0,t] \times B \times [0,\varepsilon^{-1}\phi(s,l)]}  \log \l( \psi(s,l) \r) \bar{\eta}(dl,ds,dr) \\
			\nonumber \nonumber & + \int_{(0,t] \times B \times [0,\varepsilon^{-1}\phi(s,l)]}  \l( -\psi(s,l) + 1 \r) \nu(dl)ds\,dr \bigg\} \\
			\nonumber = & \exp \bigg\{ \int_{(0,t] \times K_n \times [0,\varepsilon^{-1}\phi(s,l)]}  \log \l( \psi(s,l) \r) \bar{\eta}(dl,ds,dr) \\
			& + \int_{(0,t] \times K_n \times [0,\varepsilon^{-1}\phi(s,l)]}  \l( -\psi(s,l) + 1 \r) \nu(dl)ds\,dr \bigg\}.
		\end{align}
		\begin{enumerate}
			\item By Lemma 2.3 in \cite{Budhiraja+Dupuis+Maroulas_2011_VariationalRepresentationsContinuousTimeProcesses},  $\mathcal{E}_t(\psi)$ defined above is a $\bar{\mathbb{F}}$-martingale on $\l( \bar{\Omega} , \bar{\mathcal{F}} , \bar{\mathbb{F}} , \bar{\mathbb{P}} \r)$.

			\item Therefore by Theorem \ref{theorem existence of strong solution and J epsilon} (a similar formulation can be done for the problem \eqref{eqn stochastic control equation} with $\phi$ satisfying the required regularity criterion), the problem \eqref{eqn stochastic control equation} admits a unique martingale solution. The tuple
			\begin{equation}
				\l( \bar{\Omega} , \bar{\mathcal{F}} , \bar{\mathbb{F}} , \mathbb{P}^{\varepsilon}_{T} , M^{\varepsilon} := J^{\varepsilon}\l(\eta^{\varepsilon^{-1}\phi} \r) , \eta^{\varepsilon^{-1}\phi} \r),
			\end{equation}
			is a weak martingale solution to the above mentioned Stochastic Control Equation \eqref{eqn stochastic control equation}.
			\item The formula 
			\begin{equation}
				\mathbb{P}_T^{\varepsilon}(A) = \int_{A} \mathcal{E}_T(\psi)\, d\bar{\mathbb{P}},\ A\in\bar{\mathcal{F}},
			\end{equation}
			defines a probability measure on $\l( \bar{\Omega} , \bar{\mathcal{F}}\r)$.\\
			$\mathbb{P}_T^{\varepsilon}$ and $\bar{\mathbb{P}}$ are equivalent on $\l( \bar{\Omega} , \bar{\mathcal{F}}\r)$. Also, there exists a constants $C_1,C_2>0$ such that
			\begin{equation}
				\bar{\mathbb{P}}(A) \leq C_1 \mathbb{P}_T^{\varepsilon}(A) \leq C_2 \bar{\mathbb{P}}(A) ,\quad\ \forall A\in\bar{\mathcal{F}}.
			\end{equation}
			
			\item On $\l( \bar{\Omega} , \bar{\mathcal{F}}  , \bar{\mathbb{F}} , \bar{\mathbb{P}}_T \r)$, $\eta^{\varepsilon^{-1} , \phi}$ has the same law as that of $\eta^{\varepsilon^{-1}}$ on $\l( \bar{\Omega} , \bar{\mathcal{F}}  , \bar{\mathbb{F}} , \bar{\mathbb{P}} \r)$.

		\end{enumerate}
		Therefore by (2), Theorem \ref{theorem existence of strong solution and J epsilon}, the tuple 
		\begin{equation*}
			\l( \bar{\Omega} , \bar{\mathcal{F}}  , \bar{\mathbb{F}} , \bar{\mathbb{P}}_T , \tilde{\eta}^{\varepsilon^{-1},\phi} , J^{\varepsilon}\l(\tilde{\eta}^{\varepsilon^{-1}\phi}\r)\r),
		\end{equation*}
		is a weak martingale solution to the problem \eqref{eqn stochastic control equation}.

	\end{proof}
	\subsubsection{Pathwise Uniqueness:}
	We reiterate here that we restrict ourselves to $d = 1,2$. We now show that the obtained martingale solution in Theorem \ref{theorem existence of weak martingale solution stochastic control equation} is Pathwise unique and thereby also complete the proof of Theorem \ref{theorem pathwise uniqueness existence of weak martingale solution}.
	\begin{theorem}\label{theorem pathwise uniqueness stochastic control equation}
		Let $d = 1,2$. Let $\l(\Omega , \mathcal{F} , \mathbb{F} , \mathbb{P}\r)$ be a filtered probability space satisfying the usual assumptions. Let $M_1,M_2$ denote two solutions of \eqref{eqn stochastic control equation} corresponding to $\phi$ as in Theorem \ref{theorem existence of weak martingale solution stochastic control equation}. Then for all $t\in[0,T]$,
		\begin{equation}
			M_1(t) = M_2(t)\ \mathbb{P}-a.s.
		\end{equation}
	\end{theorem}
	\begin{proof}
		Let $M_1,M_2$ be as given in the statement of the theorem.
		Then the process $M = M_1 - M_2$ satisfies the following equality.
		\begin{align}
			\nonumber M(t) = & \int_{0}^{t}  \Delta M(s) \, ds + \int_{0}^{t} \l( M_1(s) \times \Delta M_1(s)  - M_2(s) \times \Delta M_2(s) \r) \, ds \\
			\nonumber & - \int_{0}^{t} \l[ \l( 1 + \l| M_1(s) \r|_{\mathbb{R}^3}^2 \r) M_1(s)  
			+ \l( 1 + \l| M_2(s) \r|_{\mathbb{R}^3}^2 \r) M_2(s) \r] \, ds \\
			\nonumber & + \int_{0}^{t} \int_{B} \l[ G\bigl(\varepsilon , l , M_1(s)\bigr) - G\bigl(\varepsilon , l , M_2(s)\bigr) \r]\, \tilde{\eta}^{\varepsilon^{-1}\phi}(dl,ds) \\
			\nonumber & + \int_{0}^{t} \int_{B} \l[ G\bigl(\varepsilon , l , M_1(s)\bigr) - G\bigl(\varepsilon , l , M_2(s)\bigr)\r] \l( \phi(s,l) - 1 \r)\, \nu^{\varepsilon^{-1} }(dl) \, ds\\
			\nonumber & + \varepsilon^{-1} \int_{0}^{t} \bigl[ b\bigl(\varepsilon , M_1(s)\bigr) - b\bigl(\varepsilon , M_2(s)\bigr)\bigr] \, ds \\
			= & \sum_{i=1}^{6}C_iI_{i}(t).
		\end{align}
		We apply the It\^o formula (see Gy\"ongi and Krylov \cite{Gyongy+Krylov_1982_StochasticEquationsSemimartingales}, see also Appendix in \cite{ZB+UM_WeakSolutionSLLGE_JumpNoise}) to the function $$\l\{L^2\ni v \mapsto \frac{1}{2}\l| v \r|_{L^2}^2 \in\mathbb{R}\r\}.$$
		The resulting equation is the following.
		\begin{align}
			\nonumber \frac{1}{2} \l| M(t) \r|_{L^2}^2 = & \int_{0}^{t}  \l\langle  \Delta M(s), M(s) \r\rangle_{L^2} \, ds \\
			\nonumber & + \int_{0}^{t} \l\langle \l( M_1(s) \times \Delta M_1(s)  - M_2(s) \times \Delta M_2(s) \r) , M(s) \r\rangle_{L^2} \, ds \\
			\nonumber & - \int_{0}^{t} \l\langle \l[ \l( 1 + \l| M_1(s) \r|_{\mathbb{R}^3}^2 \r) M_1(s)  
			+ \l( 1 + \l| M_2(s) \r|_{\mathbb{R}^3}^2 \r) M_2(s) \r] , M(s) \r\rangle_{L^2} \, ds \\
			\nonumber & + \varepsilon^{-1} \int_{0}^{t} \l\langle \l[ b\big(\varepsilon , M_1(s)\big) - b\big(\varepsilon , M_2(s)\big) \r] , M(s) \r\rangle_{L^2} \, ds \\
			\nonumber & + \int_{0}^{t} \int_{B} \l\langle \l[ G\big(\varepsilon , l , M_1(s)\big) - G\big(\varepsilon , l , M_2(s)\big) \r] , M(s) \r\rangle_{L^2} \tilde{\eta}^{\varepsilon^{-1}\phi}(dl,ds) \\
			\nonumber & + \frac{1}{2}\int_{0}^{t} \int_{B}  \l| G\big(\varepsilon , l , M_1(s)\big) - G\big(\varepsilon , l , M_2(s)\big) \r|_{L^2}^2 \tilde{\eta}^{\varepsilon^{-1}\phi}(dl,ds) \\
			\nonumber & + \frac{\varepsilon^{-1}}{2}\int_{0}^{t} \int_{B}  \phi(s , l)\l| G\big(\varepsilon , l , M_1(s)\big) - G\big(\varepsilon , l , M_2(s)\big) \r|_{L^2}^2\, \nu(dl)\,ds \\
			\nonumber & + \int_{0}^{t} \int_{B} \bigl(\phi(s , l) - 1\bigr)\l\langle \bigl[ G\big(\varepsilon , l , M_1(s)\big) - G\big(\varepsilon , l , M_2(s)\big) \bigr] , M(s) \r\rangle_{L^2} \, \nu(dl) \, ds\\
			= & \sum_{i=1}^{4}C_iI_i(t) + \mathcal{M}_1(t) + \mathcal{M}_2(t) + \sum_{i=5}^{6} I_i(t).
		\end{align}
		For the first term, we have
		\begin{align}
			I_1(t) = \int_{0}^{t}  \l\langle  \Delta M(s) , M(s) \r\rangle_{L^2} \, ds = - \int_{0}^{t}  \l| \nabla M(s) \r|_{L^2}^2 \, ds \leq 0.
		\end{align}
		
		Let $C>0$ denote a constant. Let us define an auxiliary function $\psi_C$ as follows. for $d = 2$
		\begin{align}
			\nonumber \psi_C(t) = & C\bigg( 1 + \l| M_2(t) \r|_{L^{\infty}} \l( \l| M_1(t) \r|_{L^{\infty}} + \l| M_2(t) \r|_{L^{\infty}} \r) \\
			& + \l| M_2(t) \r|_{H^1}^2\l| M_2(t) \r|_{H^2}^2 + \l| M_2(t) \r|_{H^1} \l| M_2(t) \r|_{H^2} \bigg).
		\end{align}
		The case of $d = 1$ can be tackled similarly, with $\psi$ defined as follows.
		\begin{equation}
			\psi_C(v) = C\l( 1 + \l| M_2(t) \r|_{L^{\infty}} \l( \l| M_1(t) \r|_{L^{\infty}} + \l| M_2(t) \r|_{L^{\infty}} \r) + \l| M_2(t) \r|_{H^1}^2 + \l| M_2(t) \r|_{H^1}^4 \r).
		\end{equation}
		Note that in both the cases ($d=1,2$), $\psi_C$ is integrable over $[0,T], \mathbb{P}$ -a.s.
		
		Fix $N\in\mathbb{N}$. Let us define a stopping time $\tau_N$ as follows.
		\begin{equation}
			\tau_N = \inf\l\{ t : \sup_{s\in[0,t]}\l|M_1(s)\r|_{H^1} \wedge \sup_{s\in[0,t]}\l|M_2(s)\r|_{H^1} + \int_{0}^{t} \l|M_1(s)\r|_{H^2}^2 \, ds \wedge \int_{0}^{t} \l|M_2(s)\r|_{H^2}^2 \, ds > N \r\} \wedge T   .
		\end{equation}
		Note that due to the energy estimates on the solutions $M_1,M_2$, we have the following for $i=1,2$.
		\begin{equation}
			\mathbb{E}\sup_{s\in[0,T]}\l|M_i(s)\r|_{H^1} + \mathbb{E}\int_{0}^{T} \l|M_i(s)\r|_{H^2}^2 \, ds < \infty.
		\end{equation}
		We recall that the mappings $b,\ G$ are Lipschitz continuous. Therefore there exists a constant $C>0$ such that
		\begin{align}
			\l|I_4(t)\r| \leq  C \int_{0}^{t} \l| M(s) \r|_{L^2}^2.
		\end{align}
		Similarly, the following holds for $i=5,6$, for some constant $C>0$.
		\begin{align}
			\l|I_i(t)\r| \leq  C \int_{0}^{t} \l| M(s) \r|_{L^2}^2.
		\end{align}
		Combining the above calculations, we get (by the choice of $\psi_C$)
		\begin{align}\label{eqn pathwise uniqueness stochastic control equation inequality for M with phi c}
			\nonumber \l| M(t) \r|_{L^2}^2 + \int_{0}^{t} \l| M(s) \r|_{H^1}^2 \, ds \leq & \int_{0}^{t} \psi_C(s)\l| M(s) \r|_{L^2}^2 \, ds + \mathcal{M}_1(t) + \mathcal{M}_2(t).
		\end{align}
		We apply the It\^o formula to the function $$\l|M(t)\r|_{L^2}^2\mapsto e^{-\int_{0}^{t}\phi_C(s) \, ds}\l|M(t)\r|_{L^2}^2.$$
		For a similar technique, see \cite{Schmalfuss_Qualitative_Properties_SNE,UM+AAP_2021_LargeDeviationsSNSELevyNoise} among others. As a result of the formula along with the stoping time $\tau_N$ and some simplification, we have the following.
		\begin{align}
			e^{-\int_{0}^{t\wedge \tau_N}\phi_C(s) \, ds}\l|M(t\wedge \tau_N)\r|_{L^2}^2 \leq & \int_{0}^{t\wedge \tau_N} e^{-\int_{0}^{s}\phi_C(r) \, dr} \bigl(\mathcal{M}_1(s) + \mathcal{M}_2(s)\bigr) \, ds.
		\end{align}
		Since $\phi_C$ is $\mathbb{P}$-a.s. integrable and non-negative, we have
		\begin{equation}
			e^{-\int_{0}^{t}\phi_C(r) \, dr} \leq 1,\ t\in[0,T].
		\end{equation}
		In particular, the process
		\begin{equation}
			\int_{0}^{t\wedge \tau_N} e^{-\int_{0}^{s}\phi_C(r) \, dr} \bigl(\mathcal{M}_1(s) + \mathcal{M}_2(s)\bigr) \, ds,\ t\in[0,T],
		\end{equation}
		is a martingale.
		Hence
		\begin{align}
			\mathbb{E}\l[ e^{-\int_{0}^{t\wedge \tau_N}\phi_C(s) \, ds}\l|M(t\wedge \tau_N)\r|_{L^2}^2\r] \leq & \mathbb{E}\l[\int_{0}^{t\wedge \tau_N} e^{-\int_{0}^{s}\phi_C(r) \, dr} \bigl(\mathcal{M}_1(s) + \mathcal{M}_2(s)\bigr) \, ds\r] = 0.
		\end{align}
		Due to the bounds assumed on the solutions $M_1,M_2$, we have 
		\begin{equation}
			\lim_{N\to\infty}\tau_N = T.
		\end{equation}
		Therefore (see for example \cite{UM+AAP_2021_LargeDeviationsSNSELevyNoise}), for each $t\in[0,T]$,
		\begin{equation}
			M_1(t)= M_2(t),\ \mathbb{P}\text{-a.s.}
		\end{equation}
		This concludes the proof of the theorem.
	\end{proof}
	
	Theorem \ref{theorem pathwise uniqueness stochastic control equation} shows that the problem \eqref{eqn stochastic control equation} admits a pathwise unique solution. Therefore we have the following result as a consequence of the existence of a weak martingale solution and its pathwise uniqueness.
	\begin{theorem}\label{Theorem existence of strong solution stochastic control equation}
		Let $d = 1,2$. The problem \eqref{eqn stochastic control equation} admits a pathwise unique, strong solution.
	\end{theorem}

	\section{Sufficient Conditions for LDP}\label{section Sufficient Conditions for LDP}
	\begin{enumerate}
		\item[\textbf{Condition 1:}]
		Let Assumption \ref{assumption main assumption} hold. For all $K\in\mathbb{N}$, let $\theta_n,\theta\in\mathcal{S}^K$ and $\theta_n\to\theta$ in $\mathcal{S}^K$ (that is  $\nu^{\theta_n}\to\nu^{\theta}$) as $n\to\infty$, then
		\begin{equation}
			J^0\l( \nu^{\theta_n} \r) \to J^0\l( \nu^{\theta} \r) \text{ in } \mathbb{U}_T.
		\end{equation}

		\item[\textbf{Condition 2:}]
		Let $K_n,n\in\mathbb{N}$ be an increasing sequence of compact subsets of $B$ such that $\cup_{n\in\mathbb{N}} = B$.		
		Let $\mathcal{U}^K$ denote the following.
		\begin{equation*}
			\mathcal{U}^K : = \l\{ \phi\in\bar{\mathcal{A}}_b : \phi \in \mathcal{S}^K, \bar{\mathbb{P}}\text{-a.s.} \r\}.
		\end{equation*}
		Let Assumption \ref{assumption main assumption} hold. Let $\{\varepsilon_n\}_{n\in\mathbb{N}}$ be a $(0,1]$-valued sequence converging to $0$. For all $K\in\mathbb{N}$, let $\phi_{\varepsilon_n},\phi\in\mathcal{S}^K$ be such that $\phi_{\varepsilon_n}$ converges in law to $\phi$ as $\varepsilon_n\to0$. Then
		\begin{equation}
			J^{\varepsilon_n}\l( \varepsilon_n \eta^{\varepsilon_n^{-1}\phi_{\varepsilon_n}}\r) \text{ converges in law to } J^0\l( \nu^{\phi} \r) \text{ in } \mathbb{U}_T.
		\end{equation}
	\end{enumerate}

	\subsection{Statement of the Main Theorem}
	\begin{theorem}\label{theorem LDP}
		Let $d=1,2$. Let the initial data $m_0$ and the given function $h$ be as in Assumption \ref{assumption main assumption}. Then the family  of laws of   $\l\{ m^{\varepsilon} \r\}_{\varepsilon\in(0,1]}$ (the family of solutions to the problem \eqref{eqn problem considered with epsilon LDP with L}) satisfies the large deviations principle on $\ \mathbb{U}_T$ with the good rate function $I$, given by
		\begin{equation}
			I(v) : = \inf_{\theta\in\mathbb{S}\,,m^{\theta} = v}\l\{ \mathcal{L}_T(\theta) \r\}\ v\in\mathbb{U}_T,
		\end{equation}
		where $m^{\theta}$ denotes the unique solution to the skeleton equation (deterministic control equation) \eqref{eqn skeleton equation} with the initial condition $m_0$. Note that $I(v) = \infty$ if the set $\l\{\theta\in\mathbb{S}:m^{\theta} = v\r\}$ is empty.
	\end{theorem}
	
	\begin{proof}[Proof of Theorem \ref{theorem LDP}]
		To prove that the family of laws of $m^{\varepsilon}$ satisfies the LDP, it suffices to verify the Conditions 1 and 2 mentioned in Section \ref{section Sufficient Conditions for LDP} (see \cite{ZB+UM+Zhai_Preprint_LDP_LLGE_JumpNoise,Budhiraja+Chen+Dupuis_2013_LDPDrivenByPoissonProcess,UM+AAP_2021_LargeDeviationsSNSELevyNoise}). Verifying the said conditions is the aim of the following Section \ref{section Verification of Conditions 1 and 2}.
	\end{proof}

	\section{Verification of Conditions 1 and 2}\label{section Verification of Conditions 1 and 2}
	The section focuses on verifying Conditions 1 and 2, thereby proving Theorem \ref{theorem LDP}.
	\subsection{Verification of Condition 1}
	We recall the mapping $J^0$ here. For $\theta\in\mathbb{S}$ (more specifically $\theta\in\mathcal{S}^K$, for some $K\in\mathbb{N}$), $J^0(\nu^{\theta}) = J^0(\theta)$ denotes the unique solution to \eqref{eqn skeleton equation}. The equality $J^0(\nu^{\theta}) = J^0(\theta)$ is justified by the identification of $\theta$ with the measure $\nu^{\theta}$, and the space $\mathcal{S}^K$ having the corresponding topology.
	
	Condition 1 is a consequence of Lemma \ref{lemma cdn1} given below. We recall the space $$\mathbb{U}_T = \mathbb{D}([0,T]:L^2)\cap L^2(0,T:H^1).$$
	\begin{lemma}\label{lemma cdn1}
		Let $K\in\mathbb{N}$ and let $\theta_n,\theta\in\mathcal{S}^K$, be such that
		\begin{equation}
			\theta_n\to\theta\text{ as }n\to\infty\text{ in }\mathbb{S}.
		\end{equation}
		Then 
		\begin{equation}
			J^0\l(\nu^{\theta_n}\r) \to J^0\l(\nu^{\theta}\r)\text{ as }n\to\infty\text{ in }\mathbb{U}_T
		\end{equation}
		Moreover, let $\theta,\tilde{\theta}$ be two $\mathcal{S}^K$-valued random variables, possibly defined on different probability spaces $\Omega,\tilde{\Omega}$, with the same laws. then the laws of the random variables 
		\begin{equation}
			\Omega \ni \omega \mapsto J^0\l(v^{\theta}\r)\in \mathbb{U}_T;\ \tilde{\Omega} \ni \tilde{\omega} \mapsto J^0\l(\nu^{\tilde{\theta}}\r)\in\mathbb{U}_T
		\end{equation}
		are equal.
	\end{lemma}
	\begin{proof}[Proof of Lemma \ref{lemma cdn1}]
		Let us fix the following notations for this proof. For $n\in\mathbb{N}$, let $m_n$ denote the unique solution to the problem \eqref{eqn skeleton equation} corresponding to $\theta_n$. Similarly, let $m$ denote the unique solution to the problem \eqref{eqn skeleton equation} corresponding to $\theta$. The existence and uniqueness of these solutions are given by Theorem \ref{theorem unique solution for the skeleton equation}. We give an outline of the proof here. Similar details can be found in \cite{ZB+UM+Zhai_Preprint_LDP_LLGE_JumpNoise,UM+AAP_2021_LargeDeviationsSNSELevyNoise} among others. Some calculations are similar to the proof of Theorem \ref{theorem unique solution for the skeleton equation}.
		\begin{enumerate}
			\item[\textbf{Step 1}] Uniform energy estimates for the sequence $m_n$, followed by compactness arguments, resulting in a limit $\bar{m}$.\\    
			Notice that in Theorem \ref{theorem unique solution for the skeleton equation} (see also Remark \ref{remark uniform bounds theta n}), the bounds on the solutions $m_n$ depend only on the terminal time $T$, the initial data $m_0$ and $K$, where $\theta_n\in\mathcal{S}^K,n\in\mathbb{N}$. Therefore the bounds established in Theorem \ref{theorem unique solution for the skeleton equation} (with $\theta$ replaced by $\theta_n$) are uniform in $n\in\mathbb{N}$.
			\begin{lemma}\label{lemma bounds lemma 1 existence of skeleton equation}
				There exists a constant $C>0$ such that the following hold.
				\begin{equation}
					\sup_{t\in[0,T]} \l| m_n(t) \r|_{H^1}^2 \leq C,
				\end{equation}
				\begin{equation}
					\int_{0}^{T} \l| m_n(t) \r|_{H^2}^2 \, dt \leq C,
				\end{equation}
			\end{lemma}
			
			\begin{lemma}\label{lemma bounds lemma 2 existence of skeleton equation}
				Let $p\geq 2,\beta>\frac{1}{4}$, $\alpha\in(0,\frac{1}{2})$. There exists a constant $C>0$ such that
				\begin{equation}
					\l| m_n \r|_{W^{\alpha , p}\l([0,T]:X^{-\beta}\r)} \leq C.
				\end{equation}
			\end{lemma}
			Using standard compactness results (for instance see Section \ref{section Proof of Existence of a Solution for Skeleton Equation}), we can obtain a subsequence $\{ m_n \}_{n\in\mathbb{N}}$ (using the same notation), along with some $\bar{m}\in L^2(0,T:H^2)\cap L^{\infty}(0,T:H^1) \cap C([0,T]:L^2)$ such that
			
			\begin{align}
				m_n\to \bar{m}&\ \text{weakly in} \ L^2(0,T:H^2),\\
				m_n\to \bar{m}&\ \text{weakly* in} \ L^{\infty}(0,T:H^1),
			\end{align}
			and
			\begin{align}\label{cdn1 convergence of mn to bar m}
				m_n \to \bar{m} \text{ in } L^4(0,T:L^4)\cap L^2(0,T:H^1) \cap C([0,T]:L^2).
			\end{align}
			
			\item[\textbf{Step 2}] Following the proof of Lemma 6.1 in \cite{UM+AAP_2021_LargeDeviationsSNSELevyNoise} (see also Section 4 in \cite{LE_Deterministic_LLBE}), we can show that this obtained limit $\bar{m}$ is a solution for the problem \eqref{eqn skeleton equation} corresponding to $\theta$.\\
			By Theorem \ref{theorem unique solution for the skeleton equation}, the solution $\bar{m}$ is unique (corresponding to $\theta$). Hence $m = \bar{m}$.
			In particular, we can conclude from the convergence in \eqref{cdn1 convergence of mn to bar m} that $m_n\to m$ in $L^4(0,T:L^4)\cap L^2(0,T:H^1) \cap C([0,T]:L^2)$, and therefore in $\mathbb{U_T}$.
		\end{enumerate}
		This concludes the proof of Lemma \ref{lemma cdn1}.
	\end{proof}

	\subsection{Condition 2}
	Let $n,N\in\mathbb{N}$. Let us define a stopping time $\tau_n^N$ as follows.
	\begin{equation}\label{eqn cdn2 definition of stopping time}
		\tau_n^N = \inf\l\{ t : \sup_{s\in[0,t]}\l|Y_n(s)\r|_{H^1} + \sup_{s\in[0,t]}\l|y_n(s)\r|_{H^1} + \int_{0}^{t} \l|Y_n(s)\r|_{H^2}^2 \, ds + \int_{0}^{t} \l|y_n(s)\r|_{H^2}^2 \, ds > N \r\}\wedge T    .
	\end{equation}
	For the sake of convenience, we supress the dependence on $N$ and write the stopping time as $\tau_n$.
	We first state and prove an auxiliary result.
	\begin{lemma}\label{lemma cdn2 convergence of norm}
		Let $Y_n = J^{\varepsilon_n}(\varepsilon_n \eta^{\varepsilon_n^{-1}\phi_{\varepsilon_n}})$ and $y_n = J^0(\phi_{\varepsilon_n})$. Then for each $N\in\mathbb{N}$,
		\begin{equation}
			\lim_{n\to\infty} \mathbb{E} \l[ \sup_{t\in[0,T\wedge \tau_n]} \l| Y_n - y_n \r|_{L^2}^2 + \int_{0}^{T\wedge\tau_n} \l| Y_n - y_n \r|_{H^1}^2 \, dt \r] = 0.
		\end{equation}
	\end{lemma}
	\begin{proof}[Proof of Lemma \ref{lemma cdn2 convergence of norm}]
		We denote $\phi_{\varepsilon_n}$ by $\phi_n$. The process $Y_n$ satisfies the following equality.
		\begin{align}\label{eqn cdn2 equation for Capital yn}
			\nonumber Y_n(t) = & m_0 + \int_{0}^{t} \l( \Delta Y_n(s) + Y_n(s) \times \Delta Y_n(s) - \l( 1 + \l| Y_n(s) \r|_{\mathbb{R}^3}^2 \r) Y_n(s) \r) \, ds \\
			\nonumber & + \int_{0}^{t} \int_{B} G\bigl(\varepsilon_n , l , Y_n(s)\bigr)\, \tilde{\eta}^{\varepsilon_n^{-1}\phi_n}(dl,dt) 
			+ \int_{0}^{t} \int_{B} G\bigl(\varepsilon_n , l , Y_n(s)\bigr) \bigl( \phi_n(s,l) - 1 \bigr)\, \nu^{\varepsilon_n^{-1} }(dl) \, ds\\
			& + \varepsilon_n^{-1} \int_{0}^{t} b\bigl(\varepsilon_n , Y_n(s)\bigr) \, ds.
		\end{align}
		Similarly, the process $y_n$ satisfies the following equality.
		\begin{align}\label{eqn cdn2 equation for small yn}
			\nonumber y_n(t) = &  m_0 + \int_{0}^{t} \biggl[ \Delta y_n(s) + y_n(s) \times \Delta y_n(s) - \bigl( 1 + \l| y_n(s) \r|_{\mathbb{R}^2}^2 \bigr) y_n(s) \biggr] \, dt\\
			& + \int_{0}^{t} \int_{B} l \bar{g}\bigl( y_n(s) \bigr) \bigl( \phi_n(t,l) - 1 \bigr) \nu(dl) \, ds.
		\end{align}
		Let $X_n = Y_n - y_n$. Then $X_n$ satisfies the following equation.
		\begin{align}\label{eqn cdn2 equation for Xn}
			\nonumber X_n(t) = & \int_{0}^{t}  \Delta X_n(s) \, ds + \int_{0}^{t} \l[ Y_n(s) \times \Delta Y_n(s) - y_n(s) \times \Delta y_n(s) \r] \, ds\\
			\nonumber & - \int_{0}^{t} \l[ \l( 1 + \l| Y_n(s) \r|_{\mathbb{R}^3}^2 \r) Y_n(s)  - \bigl( 1 + \l| y_n(s) \r|_{\mathbb{R}^3}^2 \bigr) y_n(s)  \r] \, ds \\
			\nonumber & + \int_{0}^{t} \int_{B} G\bigl(\varepsilon_n , l , Y_n(s)\bigr)\, \tilde{\eta}^{\varepsilon_n^{-1}\phi_n}(dl,dt) \\
			\nonumber & + \int_{0}^{t} \int_{B} G\bigl(\varepsilon_n , l , Y_n(s)\bigr) \bigl( \phi_n(s,l) - 1 \bigr)\, \nu^{\varepsilon_n^{-1} } \, (dl) \, ds\\
			\nonumber & + \varepsilon_n^{-1} \int_{0}^{t} b\bigl(\varepsilon_n , Y_n(s)\bigr) \, ds \\
			& +  \int_{0}^{t} \int_{B} l \bar{g}\bigl( y_n(s) \bigr) \bigl( \phi_n(t,l) - 1 \bigr) \, \nu(dl) \, ds    .
		\end{align}
		Applying the It\^o formula to the function 
		\begin{equation*}
			\l\{L^2\ni v\mapsto \frac{1}{2}\l| v \r|_{L^2}^2\in\mathbb{R}\r\},	
		\end{equation*}
		followed by the use of integration by parts gives
		
		\begin{align}\label{eqn cdn2 Ito formula 1 on Xn}
			\nonumber \frac{1}{2} \l| X_n(t) \r|_{L^2}^2 = & \int_{0}^{t}  \l\langle \Delta X_n(s) , X_n(s) \r\rangle_{L^2} \, ds + \int_{0}^{t} \l\langle    y_n(s) \times \Delta X_n(s)  , X_n(s) \r\rangle_{L^2} \, ds\\
			\nonumber & - \int_{0}^{t} \bigl\langle \l\langle Y_n(s) + y_n(s) , X_n(s) \r\rangle_{\mathbb{R}^3}Y_n(s)   , X_n(s) \bigr\rangle_{L^2}  \, ds \\
			\nonumber & + \int_{0}^{t} \l\langle  \l( 1 + \l| y_n(s) \r|_{\mathbb{R}^3}^2 \r) X_n(s)  , X_n(s) \r\rangle_{L^2}  \, ds \\
			\nonumber & + \frac{1}{2} \int_{0}^{t} \int_{B} \l| G\bigl(\varepsilon_n , l , Y_n(s)\bigr) \r|_{L^2}^2 \, \tilde{\eta}^{\varepsilon_n^{-1}\phi_n}(dl,ds) \\
			\nonumber &   + \int_{0}^{t} \l\langle G\bigl(\varepsilon_n , l , Y_n(s)\bigr) , X_n(s) \r\rangle_{L^2} \, \tilde{\eta}^{\varepsilon^{-1}_n\phi{_{\varepsilon_n}}}(dl,ds)\\
			\nonumber & + \frac{\varepsilon_n^{-1}}{2} \int_{0}^{t} \int_{B} \phi_{n}(s,l) \l| G\bigl(\varepsilon_n , l , Y_n(s)\bigr) \r|_{L^2}^2 \, \nu(dl) \, ds \\
			\nonumber & + \varepsilon_n^{-1} \int_{0}^{t} \int_{B} \l\langle G\bigl(\varepsilon_n , l , Y_n(s)\bigr) \bigl( \phi_n(s,l) - 1 \bigr) , X_n(s) \r\rangle_{L^2} \, \nu(dl) \, ds\\
			\nonumber & + \varepsilon_n^{-1} \int_{0}^{t} \l\langle b\bigl(\varepsilon_n , Y_n(s) \bigr) , X_n(s) \r\rangle_{L^2}  \, ds \\
			\nonumber & +  \int_{0}^{t} \int_{B} l \l\langle \bar{g}\bigl( y_n(s) \bigr) \bigl( \phi_n(s,l) - 1 \bigr) , X_n(s) \r\rangle_{L^2} \nu(dl)ds \\
			= & \sum_{i=1}^{4} C_iI_i(t) + \mathcal{M}_1(t) + \mathcal{M}_2(t) + \sum_{i=5}^{8} C_iI_i(t).
		\end{align}
		The terms $I_i,i=1,\dots,4$ can be calculated as done in, for example, \cite{ZB+BG+Le_SLLBE}.\\
		\textbf{Calculations for $I_5$:}
		\begin{align*}
			\l| C_5 I_5(t) \r| =  & \l| \frac{\varepsilon_n^{-1}}{2} \int_{0}^{t} \int_{B} \phi_{n}(s,l) \l| G\bigl(\varepsilon_n , l , Y_n(s)\bigr) \r|_{L^2}^2 \, \nu(dl) \, ds \r| \\
			\leq & C \frac{\varepsilon_n^{-1}}{2} \int_{0}^{t} \int_{B} \l| \phi_{n}(s,l) \r| \l| G\bigl( \varepsilon_n , l , Y_n(s) \bigr)  \r|_{L^2}^2 \, \nu(dl) \, ds  \\
			\leq & C \frac{C_{\varepsilon_n}^2}{2\varepsilon_n} \int_{0}^{t} \int_{B} \l| \phi_{n}(s,l) \r| \bigl( 1 + \l|  Y_n(s) \r|_{L^2}^2 \bigr) \, \nu(dl) \, ds.
		\end{align*}
		\textbf{Calculations for $I_6$:}
		
		We recall the constant $C_{\varepsilon_n}$ here (see Lemma \ref{lemma linear growth G,H,b with epsilon}).
		\begin{align}
			C_{\varepsilon_n} = & \, C\l( e^{C\varepsilon_n} - 1 \r) .
		\end{align}

		\begin{align}\label{eqn cdn2 inequality for I5}
			\nonumber \l| C_6 I_6(t) \r| = & \varepsilon_n^{-1} \l| \int_{0}^{t} \int_{B} \l\langle G\bigl(\varepsilon_n , l , Y_n(s)\bigr) \bigl( \phi_n(s,l) - 1 \bigr) , X_n(s) \r\rangle_{L^2} \, \nu(dl) \, ds \r| \\
			\nonumber \leq & \varepsilon_n^{-1}  \int_{0}^{t} \int_{B} \l| G\bigl(\varepsilon_n , l , Y_n(s)\bigr) \r|_{L^2} \l| \bigl( \phi_n(s,l) - 1 \bigr) \r|  \l| X_n(s) \r|_{L^2} \, \nu(dl) \, ds  \\
			\nonumber \leq &   \int_{0}^{t} \int_{B} \varepsilon_n\l| G\bigl(\varepsilon_n , l , Y_n(s)\bigr) \r|_{L^2} \l| \bigl( \phi_n(s,l) - 1 \bigr) \r|  \varepsilon_n^{-2} \l| X_n(s) \r|_{L^2} \, \nu(dl) \, ds  \\
			\nonumber \leq &   \int_{0}^{t} \int_{B} \varepsilon_n \l( 1 + \l| Y_n(s) \r|_{L^2} \r) \l| \bigl( \phi_n(s,l) - 1 \bigr) \r| C_{\varepsilon_n} \varepsilon_n^{-2} \l| X_n(s) \r|_{L^2} \, \nu(dl) \, ds  \\
			\leq & C \frac{\varepsilon_n^2}{2} \int_{0}^{t} \l( 1 + \l| Y_n(s) \r|_{L^2}^2 \r) \, ds + \frac{C_{\varepsilon_n}^2}{\varepsilon_n^{4}} \int_{0}^{t} \l| X_n(s) \r|_{L^2}^2  \, ds.
		\end{align}
		\textbf{Calculations for $I_7$:}
		\begin{align}\label{eqn cdn2 inequality for I6}
			\nonumber \l| C_7I_7(t) \r| \leq & \varepsilon_n^{-1} \l| \int_{0}^{t} \l\langle b\bigl(\varepsilon_n , Y_n(s)\bigr) , X_n(s) \r\rangle_{L^2}  \, ds \r| \\
			\nonumber = & \varepsilon_n^{-2}\varepsilon_n \l| \int_{0}^{t} \l\langle b\bigl(\varepsilon_n , Y_n(s)\bigr) , X_n(s) \r\rangle_{L^2}  \, ds \r| \\
			\nonumber \leq & \varepsilon_n^{-2}\varepsilon_n  \int_{0}^{t} \l| b\bigl(\varepsilon_n , Y_n(s)\bigr) \r|_{L^2} \l| X_n(s) \r|_{L^2}  \, ds  \\
			\nonumber \leq & C_{\varepsilon_n}\varepsilon_n^{-2}\varepsilon_n  \int_{0}^{t} \l( 1 + \l| Y_n(s) \r|_{L^2}\r) \l| X_n(s) \r|_{L^2}  \, ds  \\
			\leq & \varepsilon_n^2  \frac{1}{2} \int_{0}^{t} \l( 1 + \l| Y_n(s) \r|_{L^2}^2 \r) \, ds 
			+ \frac{C_{\varepsilon_n}^2}{2\varepsilon_n^4} \int_{0}^{t}  \l| X_n(s) \r|_{L^2}^2  \, ds.
		\end{align}
		Note that due to the definition of the constant $C_{\varepsilon_n}$, the term $\frac{C_{\varepsilon_n}^2}{2\varepsilon_n^4}$ is bounded even as $n\to\infty$ (i.e. as $\varepsilon_n\to 0$).\\
		\textbf{Calculations for $I_8$:}
		For the last term, we have the following observation first.
		\begin{equation}
			l\varepsilon_n\bar{g}(v) = G(\varepsilon_n , l , v) - H(\varepsilon_n , l , v).
		\end{equation}
		Therefore
		\begin{equation}
			l\bar{g}(v) = \varepsilon_n^{-1}\l[ G(\varepsilon_n , l , v) - H(\varepsilon_n , l , v) \r].
		\end{equation}
		Therefore
		\begin{align}
			\nonumber \l|C_8I_8(t)\r| = &  \l| \varepsilon_n^{-1} \int_{0}^{t} \int_{B} l \l\langle \bar{g}\bigl( y_n(s) \bigr) \bigl( \phi_n(t,l) - 1 \bigr) , X_n(s) \r\rangle_{L^2} \, \nu(dl) \, ds \r| \\
			\nonumber = & \l| \varepsilon_n^{-1} \int_{0}^{t} \int_{B}  \l\langle G\bigl(\varepsilon_n , l , y_n(s)\bigr) - H\bigl(\varepsilon_n , l , y_n(s)\bigr) \bigl( \phi_n(s,l) - 1 \bigr) , X_n(s) \r\rangle_{L^2} \, \nu(dl) \, ds \r| \\
			\nonumber \leq & \varepsilon_n^{-1} \int_{0}^{t} \int_{B}  \l| G\bigl(\varepsilon_n , l , y_n(s)\bigr) - H\bigl(\varepsilon_n , l , y_n(s)\bigr) \bigl( \phi_n(s,l) - 1 \bigr) \r|_{L^2}      \l|  X_n(s) \r|_{L^2}    \, \nu(dl) \, ds \\
			\nonumber \leq & \varepsilon_n^{-1} \int_{0}^{t} \int_{B}  \l| G\bigl(\varepsilon_n , l , y_n(s)\bigr)  \l( \phi_n(s,l) - 1 \r) \r|_{L^2}      \l|  X_n(s) \r|_{L^2}     \nu(dl) \, ds \\
			& +  \varepsilon_n^{-1} \int_{0}^{t} \int_{B}  \l|  H\bigl(\varepsilon_n , l , y_n(s)\bigr) \bigl( \phi_n(s,l) - 1 \bigr) \r|_{L^2}      \l|  X_n(s) \r|_{L^2}    \, \nu(dl) \, ds .
		\end{align}
		
		Therefore combining the above calculations with \eqref{eqn cdn2 inequality for I5}, \eqref{eqn cdn2 inequality for I6}, we have the following inequality.
		\begin{align}\label{eqn cdn2 inequality for I7}
			\nonumber & \l| \int_{0}^{t} \int_{B} l \l\langle \bar{g}\bigl( y_n(s) \bigr) \l( \phi_n(t,l) - 1 \r) , X_n(s) \r\rangle_{L^2} \nu(dl) \, ds \r| \\ 
			\leq &   \frac{\varepsilon_n^2}{2} \int_{0}^{t} \l( 1 + \l| y_n(s) \r|_{L^2}^2 \r) \, ds + \frac{1}{2}  C_{\varepsilon_n}^2  \varepsilon_n^{-4} C \int_{0}^{t}  \l| X_n(s) \r|_{L^2}^2  \, ds.
		\end{align}

		For now, we set aside the above inequalities and return to them at a later stage.
		Consider an auxiliary function $f:[0,T] \times \mathbb{R}\to\mathbb{R}$ given by $f(t,x) = e^{\int_{0}^{t} - \Psi(r) \, dr    }x$. The function $\Psi$ is described in the following lines.\\
		For convenience, we rewrite \eqref{eqn cdn2 Ito formula 1 on Xn} as follows
		\begin{align}
			\l| X_n(t) \r|_{L^2}^2 = \int_{0}^{t} f_1(s) \, ds + \int_{0}^{t} f_2(s) \, ds + \int_{0}^{t}f_3(s) \, \tilde{\eta}^{\varepsilon_n^{-1}\phi_n}(dl,ds),
		\end{align}
		with $f_1$ representing the integrands for the summands $I_i,i=1,\dots4$, $f_2$ representing the integrands for the summands $I_i,i=5 , \dots 8$ and $f_3$ representing the integrands for $\mathcal{M}_1 + \mathcal{M}_2$.
		Let us define a function 
		$$\Psi_C : [0,T]\to\mathbb{R}$$
		for some constant $C>0$ as follows.
		\begin{equation}
			\Psi_C(s) = C\l( 1 + \l| y_n(t) \r|_{L^{\infty}}^2 \l( \l| Y_n(t) \r|_{L^{\infty}}^2 + \l| y_n(t) \r|_{L^{\infty}}^2 \r) + \l| y_n(t) \r|_{H^1}^2\l| y_n(t) \r|_{H^2}^2 + \l| y_n(t) \r|_{H^1} \l| y_n(t) \r|_{H^2} \r)
		\end{equation}
		Note that the above expression is for dimension $2$, i.e. $d = 2$. The case $d = 1$ can be handled similarly with a slight change (see for example \cite{ZB+BG+Le_SLLBE}) in the above mentioned function $\Psi_C$.
		
		From the calculations done for the terms $I_i,i=1,\dots 7$, we can conclude that there exists a constant $C>0$ such that for each $s\in[0,T]$,
		\begin{align}\label{eqn cdn2 inequality for Xn L2 H1 norm }
			f_1(s) + f_2(s) + \l| \nabla X_n(s) \r|_{L^2}^2 \leq \Psi_C(s) \l| X_n(s) \r|_{L^2}^2 + C \l[ \frac{C_{\varepsilon_n}^2}{\varepsilon_n} + \frac{\varepsilon_n^2}{2}\r]\l( 1 + \l| Y_n(s) \r|_{L^2}^2 \r).
		\end{align}
		The third term on the left hand side of the above inequality is non-negative. Therefore, neglecting the third term on the left hand side for now, we get the following inequality.
		\begin{align}
			\nonumber e^{\int_{0}^{s} - \Psi_C(r) \, dr  } \l( f_1(s) + f_2(s) \r) \leq & e^{\int_{0}^{s} - \Psi_C(r) \, dr    } \Psi_C(s) \l| X_n(s) \r|_{L^2}^2 
			+ e^{\int_{0}^{s} - \Psi_C(r) \, dr    } C \bigg[ \frac{C_{\varepsilon_n}^2}{\varepsilon_n} \\
			& + \frac{\varepsilon_n^2}{2}\bigg]\l( 1 + \l| Y_n(s) \r|_{L^2}^2 \r).
		\end{align}		
		Applying the It\^o-L\'evy formula (see for example \cite{Oksendal+Sulem_StochasticControlJumpDiffusionsBook}) to the function $f$ gives us the following equation.
		\begin{align}
			\nonumber f(t,\l| X_n(t) \r|_{L^2}^2) = & \int_{0}^{t} e^{\int_{0}^{s} - \Psi_C(r) \, dr    } \l[- \Psi_C(s) \l| X_n(s) \r|_{L^2}^2 \, ds + f_1(s) + f_2(s) \r] \, ds \\
			& + \int_{0}^{t} e^{\int_{0}^{t} - \Psi_C(r) \, dr    } f_3(s) \, \tilde{\eta}^{\varepsilon_n^{-1}\phi_n}(dl,ds).
		\end{align}		
		Therefore
		\begin{align}
			\nonumber e^{\int_{0}^{t} - \Psi_C(r) \, dr    }\l| X_n(t) \r|_{L^2}^2 \leq & \int_{0}^{t} e^{\int_{0}^{t} - \Psi_C(r) \, dr    } f_3(s) \, \tilde{\eta}^{\varepsilon_n^{-1}\phi_n}(dl,ds) \\
			& + C \l[ \frac{C_{\varepsilon_n}^2}{\varepsilon_n} + \frac{\varepsilon_n^2}{2}\r] \int_{0}^{t} e^{\int_{0}^{s} - \Psi_C(r) \, dr    } \l( 1 + \l| Y_n(s) \r|_{L^2}^2 \r) \, ds.
		\end{align}		
		Stopping the process at $\tau_n$, (see \eqref{eqn cdn2 definition of stopping time}) gives
		\begin{align}
			\nonumber & e^{\int_{0}^{t \wedge \tau_n } - \Psi_C(r) \, dr    }\l| X_n(t \wedge \tau_n ) \r|_{L^2}^2 \\
			\nonumber \quad &\leq \int_{0}^{t \wedge \tau_n } e^{\int_{0}^{t} - \Psi_C(r) \, dr    } f_3(s) \, \tilde{\eta}^{\varepsilon_n^{-1}\phi_n}(dl,ds) \\
			&  \quad +  C \l[ \frac{C_{\varepsilon_n}^2}{\varepsilon_n} + \frac{\varepsilon_n^2}{2}\r] \int_{0}^{t \wedge \tau_n } e^{\int_{0}^{s} - \Psi_C(r) \, dr    } \l( 1 + \l| Y_n(s) \r|_{L^2}^2 \r) \, ds.
		\end{align}
		Therefore
		\begin{align}
			\nonumber & \l| X_n(t \wedge \tau_n ) \r|_{L^2}^2 \\
			\nonumber \quad &\leq e^{\int_{0}^{t \wedge \tau_n } \Psi_C(r) \, dr    } \bigg[ \int_{0}^{t \wedge \tau_n } e^{\int_{0}^{s} - \Psi_C(r) \, dr    } f_3(s) \, \tilde{\eta}^{\varepsilon_n^{-1}\phi_n}(dl,ds) 
			\\
			& +  C \l[ \frac{C_{\varepsilon_n}^2}{\varepsilon_n} + \frac{\varepsilon_n^2}{2}\r] \int_{0}^{t \wedge \tau_n } e^{\int_{0}^{s} - \Psi_C(r) \, dr    } \l( 1 + \l| Y_n(s) \r|_{L^2}^2 \r) \, ds\bigg].
		\end{align}
		Due to the embedding $H^2\hookrightarrow L^{\infty}$ and the definition of the stopping time $\tau_n$, there exists a constant $C_N$ such that
		\begin{align}
			e^{\int_{0}^{t \wedge \tau_N } \Psi_C(r) \, dr    } \leq C_N.
		\end{align}
		Moreover since $\Psi_C$ is non-negative, for each $s\in[0,T]$ we have the following.
		\begin{align}
			e^{\int_{0}^{s} - \Psi_C(r) \, dr    } \leq 1.
		\end{align}
		Therefore
		\begin{align}\label{eqn cdn2 inequality for Xn with e Psi before calculating f3}
			\nonumber\l| X_n(t \wedge \tau_n ) \r|_{L^2}^2 
			\leq & C_N \bigg[ \int_{0}^{t \wedge \tau_n } e^{\int_{0}^{s} - \Psi_C(r) \, dr    } f_3(s) \, \tilde{\eta}^{\varepsilon_n^{-1}\phi_n}(dl,ds) \\
			& +  C \l[ \frac{C_{\varepsilon_n}^2}{\varepsilon_n} + \frac{\varepsilon_n^2}{2}\r] \int_{0}^{t \wedge \tau_n } e^{\int_{0}^{s} - \Psi_C(r) \, dr    } \l( 1 + \l| Y_n(s) \r|_{L^2}^2 \r) \, ds \bigg].
		\end{align}
		We now do estimates for the term $f_3$.
		First we establish some bounds on the term $f_3$.
		\begin{align*}
			\l| f_3 \r|_{L^2}^2 = & \l| \frac{1}{2}  \l| G\l(\varepsilon_n , l , Y_n\r) \r|_{L^2}^2 + \l\langle G(\varepsilon_n , l , Y_n) , X_n(s) \r\rangle_{L^2} \r|^2 \\
			\leq & C_{\varepsilon_n}^2   \l| G\l(\varepsilon_n , l , Y_n\r) \r|_{L^2}^4  + C_{\varepsilon_n}\l| G\l(\varepsilon_n , l , Y_n\r) \r|_{L^2}^2 \l| X_n \r|_{L^2}^2  \\
			\leq & C C_{\varepsilon_n}^2 \l( 1 + \l| Y_n \r|_{L^2}^4 \r) + C C_{\varepsilon_n} \l| X_n \r|_{L^2}^2 \l( 1 + \l| Y_n \r|_{L^2}^2 \r).  
		\end{align*}
		Therefore by the Burkh\"older-Davis-Gundy Inequality, we have

		\begin{align}
			\nonumber & \mathbb{E} \sup_{t\in[0,T]} \l| \int_{0}^{t \wedge \tau_n }  e^{\int_{0}^{s} - \Psi_C(r) \, dr    } f_3(s) \, \tilde{\eta}^{\varepsilon_n^{-1}\phi_n}(dl,ds) \r| \\
			\nonumber \leq & C \mathbb{E} \bigg( \int_{0}^{T \wedge \tau_n } \int_{B} e^{ - 2\int_{0}^{s} \Psi_C(r) \, dr    } \bigg[  
			\l| G\bigl(\varepsilon_n , l , Y_n(s)\bigr) \r|_{L^2}^2 \\
			\nonumber & +  \l\langle G\big(\varepsilon_n , l , Y_n(s)\big) , X_n(s) \r\rangle_{L^2} \bigg]^2
			\, \nu^{\varepsilon_n^{-1}}(dl) \, ds \bigg)^{\frac{1}{2}} \\
			\nonumber \leq & C C_{\varepsilon_n}^2 \mathbb{E} \l( \int_{0}^{T \wedge \tau_n } \int_{B} e^{ - 2\int_{0}^{s} \Psi_C(r) \, dr    }     \l( 1 + \l| Y_n(s) \r|_{L^2}^4 \r)  \, \nu^{\varepsilon_n^{-1}}(dl) \, ds \r)^{\frac{1}{2}} \\
			\nonumber & + C  C_{\varepsilon_n}   \mathbb{E} \l( \int_{0}^{T \wedge \tau_n } \int_{B} e^{ - 2\int_{0}^{s} \Psi_C(r) \, dr } \l| X_n(s) \r|_{L^2}^2 \l( 1 + \l| Y_n(s) \r|_{L^2}^2 \r)  \, \nu^{\varepsilon_n^{-1}}(dl) \, ds \r)^{\frac{1}{2}} \\
			\nonumber \leq & \varepsilon_n^{-\frac{1}{2}} C C_{\varepsilon_n}^2 \mathbb{E} \l( \int_{0}^{T \wedge \tau_n } \int_{B} e^{ - 2\int_{0}^{s} \Psi_C(r) \, dr    }     \l( 1 + \l| Y_n(s) \r|_{L^2}^4 \r)  \, \nu(dl) \, ds \r)^{\frac{1}{2}} \\
			\nonumber & + \varepsilon_n^{-\frac{1}{2}} C  C_{\varepsilon_n}   \mathbb{E} \l( \int_{0}^{T \wedge \tau_n } \int_{B} e^{ - 2\int_{0}^{s} \Psi_C(r) \, dr } \l| X_n(s) \r|_{L^2}^2 \l( 1 + \l| Y_n(s) \r|_{L^2}^2 \r)  \, \nu(dl) \, ds \r)^{\frac{1}{2}} \\
			\leq & C \l(  \frac{C_{\varepsilon_n}^2}{\varepsilon_n^{\frac{1}{2}}} +  \frac{C_{\varepsilon_n}}{\varepsilon_n} \mathbb{E} \int_{0}^{T} \l| X_n(s) \r|_{L^2}^2 \, ds \r).
		\end{align}
	We now go back to the inequality \eqref{eqn cdn2 inequality for Xn with e Psi before calculating f3}. Taking the supremum over $[0,T]$, followed by taking the expectation of both sides and then combining the resulting inequality with the calculations done so far, there exists a constant $C_N$ that is independent of $n$ such that
		\begin{align}
			\mathbb{E} \sup_{t\in[0,T]} \l| X_n(t \wedge \tau_n ) \r|_{L^2}^2 
			\leq C_N \l( \frac{C_{\varepsilon_n}^2}{\varepsilon_n^{\frac{1}{2}}} + \frac{C_{\varepsilon_n}^2}{\varepsilon_n} + C_N \frac{\varepsilon_n^2}{2}  \r) + \frac{C_{\varepsilon_n}}{\varepsilon_n} \mathbb{E} \int_{0}^{T} \l| X_n(s) \r|_{L^2}^2 \, ds.
		\end{align}
		Note that again, by the definition of the constant $C_{\varepsilon_n}$, the term $\frac{C_{\varepsilon_n}}{\varepsilon_n}$ is bounded independent of $n$. Using the Gronwall inequality, there exists another constant $C_N$ such that
		\begin{align}\label{eqn cdn2 inequality for L2 norm of Xn going to zero}
			\mathbb{E} \sup_{t\in[0,T]} \l| X_n(t \wedge \tau_n ) \r|_{L^2}^2 \leq C_N \l( \frac{C_{\varepsilon_n}^2}{\varepsilon_n^{\frac{1}{2}}} + \frac{C_{\varepsilon_n}^2}{\varepsilon_n} + \frac{\varepsilon_n^2}{2} \r).
		\end{align}
		For each fixed $N\in\mathbb{N}$, the right hand side of the above inequality goes to $0$ as $n$ goes to infinity.		
		Going back to the inequality \eqref{eqn cdn2 inequality for Xn L2 H1 norm } and applying the stopping time $\tau_N$, we have
		\begin{align}
			\nonumber \mathbb{E} \int_{0}^{T \wedge \tau_n } \l|  X_n(s) \r|_{H^1}^2 \, ds \leq & C
			\mathbb{E}\int_{0}^{T \wedge \tau_n } \l( \Psi_C(s) + 1 \r) \l| X_n(s) \r|_{L^2}^2 \\
			& + C \l[ \frac{C_{\varepsilon_n}^2}{\varepsilon_n} + \frac{\varepsilon_n^2}{2}\r] \mathbb{E}\int_{0}^{T \wedge \tau_n } \l( 1 + \l| Y_n(s) \r|_{L^2}^2 \r).
		\end{align}
		Using \eqref{eqn cdn2 inequality for L2 norm of Xn going to zero}, for each fixed $N\in\mathbb{N}$, the right hand side, and hence the left hand side of the above inequality goes to $0$ as $n$ goes to infinity.		
		This concludes the proof of the lemma.

	\end{proof}

	\begin{lemma}\label{lemma cdn2 convergence in probability}
		Let Assumption \ref{assumption main assumption} hold. Let $\varepsilon_n$ be a $(0,1]$-valued sequence converging to $0$. Let $\phi_{\varepsilon_{n}}$ be a $\mathcal{U}^K$-valued sequence, along with $\phi\in\mathcal{U}^K$ such that $\mathcal{L}\l(\phi_{\varepsilon_n}\r)$ converges to $\mathcal{L}(\phi)$ on $\mathcal{S}^K$. Here $\mathcal{L}(\phi)$ (and respectively $\mathcal{L}\l(\phi_{\varepsilon_n}\r)$) denotes the law of $\phi$ (and respectively $\phi_{\varepsilon_n}$). Then the sequence of random variables
		\begin{equation}
			\mathbb{M} \ni \nu \mapsto J^{\varepsilon_n}\l( \varepsilon_n \eta^{ \varepsilon_n^{-1}\phi_{\varepsilon_n}} \r) - J^0\l(\nu^{\varepsilon_n}\r) \in \mathbb{U}_T,
		\end{equation}
		converges to $0$ in probability.
	\end{lemma}
	\begin{proof}[Proof of Lemma \ref{lemma cdn2 convergence in probability}]
		As done in Lemma \ref{lemma cdn2 convergence of norm}, we denote $Y_n = J^{\varepsilon_n}\l( \varepsilon_n \eta^{\varepsilon_n^{-1}\phi_{\varepsilon_n}}\r)$ and $y_n = J^0\l( \nu^{\phi_{\varepsilon_n}} \r)$.
		For a sequence of random variables $Z_n,n\in\mathbb{N}$ and $Z$, in order to show that $Z_n\to Z$ in probability, it suffices to show that each $\delta>0$,
		\begin{equation}
			\lim_{n\to\infty}\mathbb{P}\l( \l[ \sup_{t\in[0,T]}\l| Z_n(t) - Z(t) \r|_{L^2}^2 + \int_{0}^{T} \l| Z_n(t) - Z(t) \r|_{H^1}^2 \, dt \r] > \delta \r) = 0, \text{ for } \delta>0.
		\end{equation}
		Let $\delta,\alpha>0$ be given.		
		For the sake of the following calculations, let us fix the following notation (short-hand).
		\begin{align}
			\l| v \r|_{Y_T} := \sup_{s\in[0,T]}\l|v(s)\r|_{H^1}^2   + \int_{0}^{T} \l|v(s)\r|_{H^2}^2 \, ds .
		\end{align}
		Now, for each $n,N\in\mathbb{N}$, the following holds.		
		
		\begin{align}
			\nonumber \mathbb{P} \l(\l| Y_n - y_n \r|_{Y_{T}} \geq\delta \r) \leq &
			\mathbb{P} \l(\l| Y_n - y_n \r|_{Y_{T}} \geq \delta : \tau_n = T \r)\\
			\nonumber  &+ \mathbb{P} \l(\l| Y_n - y_n \r|_{Y_{T}} \geq \delta : \tau_n < T \r) \\
			\leq &  \frac{1}{\delta} \mathbb{E}  \l| Y_n - y_n \r|_{Y_{T\wedge \tau_n}} + \frac{1}{N} \l[ \mathbb{E}\l|Y_n\r|_{Y_T} + \mathbb{E}\l|y_n\r|_{Y_T}\r].
		\end{align}		
		We choose $N_{\alpha}$ large enough so that
		\begin{equation}
			\frac{1}{N_{\alpha}} \sup_{n\in\mathbb{N}} \bigl[ \mathbb{E}\l|Y_n\r|_{Y_T} + \mathbb{E}\l|y_n\r|_{Y_T} \bigr] < \frac{\alpha}{2}.
		\end{equation}
		For this $N_{\alpha}$ chosen above, by Lemma \ref{lemma cdn2 convergence of norm} we can choose $N_0$ large enough so that
		\begin{equation}
			\frac{1}{\delta} \mathbb{E}  \l| Y_n - y_n \r|_{Y_{T\wedge \tau_n}} < \frac{\alpha}{2}, \forall n\geq N_0.
		\end{equation}
		Hence for any given $\delta,\alpha>0$, we can choose $N_0$ large enough so that
		\begin{equation}
			\mathbb{P} \l(\l| Y_n - y_n \r|_{Y_{T}} \geq\delta \r) < \alpha,\ \forall n\geq N_0.
		\end{equation}
		This concludes the proof of Lemma \ref{lemma cdn2 convergence in probability}.
	\end{proof}
	
	\subsection{Verification of Condition 2}
	Condition two is a consequence of Lemma \ref{lemma cdn2 convergence in probability}.
	We recall that we have to show that
	\begin{equation}
		J^{\varepsilon_n}\l( \varepsilon_n \eta^{\varepsilon_n^{-1},\phi_{\varepsilon_n}} \r) \to J^{0}\l(\phi\r),
	\end{equation}
	in law on $\mathbb{U}_T$.
	Let $f: \mathbb{U}_T\to\mathbb{R}$ be globally Lipschitz continuous and bounded.
	\begin{align}
		\nonumber & \l| \int_{\mathbb{U}_T} f(x) \, d\mathcal{L}\l( J^{\varepsilon_n}\l( \varepsilon_n \eta^{\varepsilon_n^{-1},\phi_{\varepsilon_n}} \r) \r) - \int_{\mathbb{U}_T} f(x) \, d\mathcal{L}\l( J^{0}\l( \phi \r) \r) \r| \\
		\nonumber = & \l| \int_{\Omega} f\l(  J^{\varepsilon_n}\l( \varepsilon_n \eta^{\varepsilon_n^{-1},\phi_{\varepsilon_n}} \r) \r) \, d\mathbb{P} - \int_{\Omega} f\l(  J^{0}\l( \phi \r) \r) \, d\mathbb{P} \r| \\
		\nonumber = & \l| \int_{\tilde{\Omega}} f\l(  J^{\varepsilon_n}\l( \varepsilon_n \eta^{\varepsilon_n^{-1},\phi_{\varepsilon_n}} \r) \r) \, d\tilde{\mathbb{P}} - \int_{\tilde{\Omega}} f\l(  J^{0}\l( \phi \r) \r) \, d\tilde{\mathbb{P}} \r| \\
		\nonumber \leq & \l| \int_{\tilde{\Omega}} f\l(  J^{\varepsilon_n}\l( \varepsilon_n \eta^{\varepsilon_n^{-1},\tilde{\phi}_{\varepsilon_n}} \r) \r) \, d\tilde{\mathbb{P}} - \int_{\tilde{\Omega}} f\l(  J^{0}\l( \tilde{\phi}_{\varepsilon_n} \r) \r) \, d\tilde{\mathbb{P}} \r| \ \boxed{\text{Term 1}}\\
		\nonumber & + \l| \int_{\tilde{\Omega}} f\l(  J^{0}\l( \tilde{\phi}_{\varepsilon_n} \r) \r) \, d\tilde{\mathbb{P}} - \int_{\tilde{\Omega}} f\l(  J^{0}\l( \phi \r) \r) \, d\tilde{\mathbb{P}} \r|.\ \boxed{\text{Term 2}}
	\end{align}
	For Term 1, $f$ is globally Lipschitz continuous and bounded and $J^{\varepsilon_n}\l( \varepsilon_n \eta^{\varepsilon_n^{-1},\tilde{\phi}_{\varepsilon_n}} \r) - J^{0}\l( \tilde{\phi}_{\varepsilon_n} \r)$ converges to $0$ (Lemma \ref{lemma cdn2 convergence in probability}). Term 1 goes to $0$ as $n$ goes to $\infty$ by dominated convergence theorem.	
	For Term 2, that $f$ globally Lipschitz continuous and bounded along with Lemma \ref{lemma cdn1} imply that Term 2 goes to $0$ as $n$ goes to $\infty$.
	
	This concludes the verification of Condition 2.

	\appendix

	\section{Proof of Existence of a Solution for Skeleton Equation}\label{section Proof of Existence of a Solution for Skeleton Equation}

	\begin{proof}[Proof of Theorem \ref{theorem unique solution for the skeleton equation}]

		\begin{proof}[\textbf{Brief idea of the proof:}]
			The proof is standard. We approximate the equation \eqref{eqn skeleton equation} by finite dimensional equations (Faedo-Galerkin approximations). The structure of the proof is similar to the proof of Theorem \ref{theorem existence of weak martingale solution}. The existence of approximates is followed by establishing uniform energy estimates on the obtained solutions. Standard compactness arguments give convergence (possibly along a subsequence). The obtained limit is shown to be a solution of \eqref{eqn skeleton equation}.
		\end{proof}
		\begin{enumerate}
			\item[\textbf{Step 1:}] 
			Let, as in Section \ref{section Faedo Galerkin Approximations}, $H_n\subset L^2$ denote the linear span of the eigenfunctions corresponding to the first $n$ eigenvalues of the Neumann Laplacian. Let $P_n:L^2\to H_n$ denote the orthogonal projection operator onto $H_n$.\\
			We approximate the equation \eqref{eqn skeleton equation} by the following finite dimensional equation.
			\begin{align}\label{eqn skeleton FG eqn 1}
				\nonumber dm_n^{\theta}(t) = & \biggl[ P_n \Delta m_n^{\theta}(t) + P_n\l( m_n^{\theta}(t) \times \Delta m_n^{\theta}(t) \r) - P_n\bigl( \l( 1 + \l| m_n^{\theta}(t) \r|_{\mathbb{R}^2}\r) m_n^{\theta}(t) \bigr) \biggr] \, dt\\
				&+ \int_{B} l P_n\biggl( \bar{g}  \bigl( m_n^{\theta}(t) \bigr) \biggr)\bigl( \theta\l(t,l\r) - 1 \bigr) 
				\, \nu(dl) \, dt,\ t\geq 0,
			\end{align}
			with $m_n^{\theta}(0) = P_n(m_0)$.\\
			Let us write down a few notations, which will be used only in the proof of this theorem.
			\begin{align}
				F_n^1 & : H_n \ni m \mapsto  \Delta m    \in H_n, \\
				F_n^2 & : H_n \ni m \mapsto P_n \l( m \times \Delta m \r)  \in H_n, \\
				F_n^3 & : H_n \ni m \mapsto P_n\l( \l( 1 + \l| m \r|_{\mathbb{R}^3}^2\r) m \r) \in H_n, \\
				\bar{g}_n & : H_n \ni m \mapsto P_n \bigl( \bar{g}\l( m \r) \bigr)    \in H_n.
			\end{align}
			For the proof of this theorem, we suppress the notation $\theta$, that is, we replace $m^{\theta}$ be $m$, $m_n^{\theta}$ by $m_n$, etc. Using the above notations, equation \eqref{eqn skeleton FG eqn 1} can be written in the integral form as
			\begin{align}\label{eqn skeleton FG eqn 2}
				\nonumber m_n (t) = &   m_n(0) + \int_{0}^{t} \biggl[ F_n^1\bigl(m_n(s)\bigr) + F_n^2\bigl(m_n(s)\bigr) - F_n^3\bigl(m_n(s)\bigr) \biggr] \, ds \\
				& + \int_{0}^{t} \int_{B} l \bar{g}_n(m_n(s)) \l( \theta(s,l) - 1 \r) \, \nu(dl)ds.
			\end{align}
			The first three mappings, viz. $F_n^i,i=1,2,3$ are locally Lipschitz (see for example \cite{LE_Deterministic_LLBE}).
			
			Let $v_1,v_2\in H_n,$ and $s\in[0,T]$. Then there exists a constant $C_{\text{s}} $ such that
			\begin{align}
				\nonumber \l| \int_{B}l \l( \bar{g}_n(v_1) - \bar{g}_n(v_2) \r) \l( \theta(s,l) - 1 \r) \, \nu(dl) \r|_{L^2} \leq & \l| h \r|_{L^{\infty}} \int_{B} \l| l \r| \l| v_1 - v_2 \r|_{L^2} \l| \theta(s,l) - 1 \r| \, \nu (dl) \\
				\nonumber \leq & C_{\text{s}} \l| v_1 - v_2 \r|_{L^2} \int_{B} \l| l \r| \l| \theta(s,l) - 1 \r| \, \nu (dl) \\
				\leq & C_{\text{s}} \l| v_1 - v_2 \r|_{L^2} .
			\end{align}
			The justification for the last inequality is as follows. By Proposition \ref{proposition uniform bound on theta SK}, the quantity $\int_{B} \l| l \r| \l| \theta(s,l) - 1 \r| \, \nu (dl)$ is finite, for a.a. $s\in[0,T]$. Moreover, by the same Proposition \ref{proposition uniform bound on theta SK}, integrating the above inequality over $[0,t],\ t\in[0,T]$, there exists a constant $C>0$ such that
			\begin{align}
				\int_{0}^{t} \l| \int_{B}l \l( \bar{g}_n(v_1) - \bar{g}_n(v_2) \r) \l( \theta(s,l) - 1 \r) \, \nu(dl) \r|_{L^2} \, ds \leq & C \l| v_1 - v_2 \r|_{L^2}.
			\end{align}
			In particular, the mapping
			\begin{equation}
				\l\{ F_n^4 : H_n \ni v\mapsto \int_{B}l  \bar{g}_n(v)  \bigl( \theta(s,l) - 1 \bigr) \, \nu(dl) \in H_n \r\},
			\end{equation}
			is Lipschitz continuous. Therefore we can also conclude linear growth property for the above mentioned mapping.
			Hence the problem \eqref{eqn skeleton FG eqn 2} admits a unique global solution in $H_n$ (for example, see \cite{ZB+Albeverio_2010_ExistenceGlobalSolutionSDEPoissonNoise}).
			
			\item[\textbf{Step 2:}] Uniform Energy Estimates.

			\begin{lemma}\label{lemma skeleton equation existence bounds lemma 2}
				There exists a constant $C>0$, depending on $T,m_0$ but not on $n\in\mathbb{N}$ such that the following hold.
				\begin{equation}
					\sup_{t\in[0,T]} \l| m_n(t) \r|_{L^2}^2 \leq C,
				\end{equation}
				\begin{equation}
					\int_{0}^{T} \l| m_n(t) \r|_{H^1}^2 \, dt \leq C,
				\end{equation}
				\begin{equation}
					\int_{0}^{T} \l| m_n(t) \r|_{L^4}^4 \, dt \leq C.
				\end{equation}
				Note that the constant $C$ depends on the initial data $m_0$, the terminal time $T$ and $K$ in the space $\mathcal{S}^K$ (see Proposition \ref{proposition uniform bound on theta SK}).
			\end{lemma}
			\begin{proof}[Proof of Lemma \ref{lemma skeleton equation existence bounds lemma 2}]
				
				Let $t\in[0,T]$ and $n\in\mathbb{N}$. Multiplying the equation \eqref{eqn skeleton FG eqn 2} by $m_n$ gives
				\begin{align}
					\frac{1}{2}\l|m_n(t)\r|_{L^2}^2 = \frac{1}{2}\l|m_0\r|_{L^2}^2 + \int_{0}^{t} \l\langle F_n^1(s) + F_n^2(s) - F_n^3(s) + \l(\int_{B}l\bar{g}_n(s)\nu(dl)\r) , m_n(s) \r\rangle_{L^2} \, ds.
				\end{align}
	We first show calculations for the last term on the left hand side of the inner product. For $t\in[0,T]$,
				\begin{align}
					\nonumber & \l| \int_{0}^{t} \l\langle \int_{B} \l| l \r| \bar{g}_n(m_n(s)) \bigl( \theta(s,l) - 1 \bigr) , m_n(s) \r\rangle_{L^2} \, \nu(dl)  \, ds \r| \\
					\nonumber \leq & \int_{0}^{t} \l( \l[ \frac{1}{2} \l| h \r|_{L^2}^2 + \frac{1}{2} \l| m_n(s) \r|_{L^2}^2 \r] \int_{B}  \l| l \r|  \l| \theta(s,l) - 1 \r| \, \nu(dl) \r) \, ds \\
					\nonumber \leq & \int_{0}^{t} \l( \l[ \frac{1}{2} \l| h \r|_{L^2}^2 \r] \int_{B}  \l| l \r| \l| \theta(s,l) - 1 \r| \, \nu(dl) \r) \, ds \\
					\nonumber & + \int_{0}^{t} \l( \frac{1}{2} \l| m_n(s) \r|_{L^2}^2 \int_{B}  \l| l \r| \l| \theta(s,l) - 1 \r| \, \nu(dl) \r) \, ds \\
					\nonumber \leq &  \frac{1}{2} \l| h \r|_{L^2}^2  \int_{0}^{t}   \int_{B}  \l| l \r|  l \l| \theta(s,l) - 1 \r| \, \nu(dl)  \, ds \\
					& +  \frac{1}{2} \int_{0}^{t}  \l|   m_n(r) \r|_{L^2}^2 \int_{B}  \l| l \r| \l| \theta(s,l) - 1 \r| \, \nu(dl)  \, ds .
				\end{align}
				The first three terms (viz. $F_n^i,i=1,2,3$) can be handled in the spirit of Lemma \ref{proposition properties of F existence of weak martingale solution} (see also \cite{LE_Deterministic_LLBE}). This, combined with the above calculation, along with Proposition \ref{proposition uniform bound on theta SK} we have the following inequality.
				\begin{align}
					\nonumber \l| m_n(t) \r|_{L^2}^2 & + \int_{0}^{t} \l|m_n(s)\r|_{H^1}^2 \, ds + \int_{0}^{t} \l|m_n(s)\r|_{L^4}^4 \, ds \\
					& \leq  C \bigg[ \l| m_0 \r|_{L^2}^2 + 1  + \int_{0}^{t}   \sup_{r\in[0,s]} \l|m_n(s)\r|_{L^2}^2 \int_{B}  \l| l \r| \l| \theta(s,l) - 1 \r| \, \nu(dl) \, ds\bigg].
				\end{align}
				Note again that the second term on the left hand side of the above inequality contains the full 
				$H^1$ norm, which is obtained as the sum of $\l|\nabla m_n(s)\r|_{L^2}^2$ and $\l|m_n(s)\r|_{L^2}^2$. First, we observe that all the three terms on the left hand side of the above inequality are non-negative. Hence the second and the third can be neglected, for now, keeping the inequality as it is. Taking supremum over $[0,T]$ and using the Gronwall inequality gives us the required result.
				Note that after using the Gronwall inequality, the term $\l( e^{ \int_{0}^{T} \int_{B} \l| l \r| \l| \theta(s,l) - 1 \r| \, \nu(dl) \, ds } \r)$ appears as a multiplier. Since this term is finite (depending on $T$ and $K$, where $\theta\in\mathcal{S}^K$, see Proposition \ref{proposition uniform bound on theta SK}), the inequality can be proved. Proof for the remaining two inequalities can be given following the idea of the proof of Lemma \ref{lemma bounds 1}.
			\end{proof}

			\begin{lemma}\label{lemma skeleton equation existence bounds lemma 3}
				There exists a constant $C>0$ such that the following hold.
				\begin{equation}
					\sup_{t\in[0,T]} \l| m_n(t) \r|_{H^1}^2 \leq C,
				\end{equation}
			
				\begin{equation}
					\int_{0}^{T} \l| \Delta m_n(t) \r|_{L^2}^2 \, dt \leq C,
				\end{equation}
				
			\end{lemma}
			\begin{proof}[Proof of Lemma \ref{lemma skeleton equation existence bounds lemma 3}]
				The proof is similar in structure to the proof of Lemma \ref{lemma skeleton equation existence bounds lemma 2} and is hence skipped.
			
			\end{proof}
			
			\begin{lemma}\label{lemma skeleton equation existence bounds lemma 4}
				Let $\beta>\frac{1}{4}$, $\alpha\in(0,\frac{1}{2})$ and $p\geq 2$. There exists a constant $C>0$, which can depend on $\alpha,p,\beta$ but not on $n\in\mathbb{N}$, such that
		
				\begin{equation}
					\l| m_n \r|_{W^{\alpha , p}\l(0,T:X^{-\beta}\r)} \leq C.
				\end{equation}
			\end{lemma}
			\begin{proof}[Proof of Lemma \ref{lemma skeleton equation existence bounds lemma 4}]
				For a similar result, See Lemma 4.9 in the \cite{ZB+UM+Zhai_Preprint_LDP_LLGE_JumpNoise}. 
				We rewrite the equality \eqref{eqn skeleton FG eqn 2} as follows.
				\begin{align}
					\nonumber \l| m_n(t) \r|_{L^2}^2 = & \l| m_0 \r|_{L^2}^2 + 2\int_{0}^{t} F_n^1(m_n(s)) \, ds 
					+ 2\int_{0}^{t} F_n^2(m_n(s)) \, ds\\
					\nonumber & - 2\int_{0}^{t} F_n^3(m_n(s)) \, ds
					+ 2\int_{0}^{t} F_n^4(m_n(s)) \, ds \\
					= &  \sum_{i=1}^{4}  J_i(t).
				\end{align}
				
				For $i=1,2,3$, using the bounds established in Lemma \ref{lemma bounds lemma 2 existence of skeleton equation}, one can show that there exists a constant $C>0$, which can depend on $\alpha,p$, but not on $n\in\mathbb{N}$ such that
				\begin{equation}
					\l|J_i\r|_{W^{1,2}(0,T:X^{-\beta})} \leq C.
				\end{equation}
				For $p$ such that $\frac{1}{2} + \frac{1}{p} > \alpha $, we have the following continuous embedding \cite{Simon_Compact_Sets}
				\begin{equation}
					W^{1,2}(0,T:X^{-\beta}) \hookrightarrow W^{\alpha,p}(0,T:X^{-\beta}).
				\end{equation}
				Further, following Lemma A.2, Appendix A in \cite{UM+AAP_2021_LargeDeviationsSNSELevyNoise}, one can show that there exists a constant $C>0$ independent of $n\in\mathbb{N}$ such that
				\begin{equation}
					\l|J_4\r|_{W^{\alpha,p}(0,T:X^{-\beta})} \leq C.
				\end{equation}
				
			\end{proof}
			
			\item[\textbf{Step 3:}] \textbf{Compactness Arguments:}
			Using Theorem 2.2 of Flandoli and Gatarek \cite{Flandoli_Gatarek}, the space $W^{\alpha,p}(0,T:X^{-\beta^\p})$ is compactly embedded into the space $C([0,T]:X^{-\beta})$, for $\beta > \beta^\p$ and $\alpha p >1$. 
			Again, using Theorem 2.1 of \cite{Flandoli_Gatarek}, we have the following compact embedding
			\begin{equation}
				L^2(0,T:H^2)\cap L^{\infty}(0,T:H^1)\hookrightarrow L^2(0,T:H^1).
			\end{equation}
			Using the compactness arguments  and uniform bounds (Lemmas \ref{lemma skeleton equation existence bounds lemma 2}, \ref{lemma skeleton equation existence bounds lemma 3}, \ref{lemma skeleton equation existence bounds lemma 4}) above, there exists an element 
			\begin{align*}
				m\in L^{\infty}(0,T:H^1)\cap L^2(0,T:H^2)\cap C([0,T]:X^{-\beta})\cap L^4(0,T:L^4),
			\end{align*}
			such that
			\begin{align}
				m_n\to m\ \text{strongly in}\ L^2(0,T:H^1)\cap L^4(0,T:L^4),
			\end{align}
			\begin{align}
				m_n\to m\ \text{weakly in}\ L^2(0,T:H^2).
			\end{align}
			For $r\in(1,\frac{4}{3})$, there exists a constant $C>0$ such that
			\begin{equation}
				\int_{0}^{T} \l| m_n(t) \times \Delta m_n(t) \r|_{L^{2}}^r \, dt \leq C.
			\end{equation}
			In particular, we can conclude (see \cite{LE_Deterministic_LLBE}) that
			\begin{equation}
				m_n \times \Delta m_n \to m \times \Delta m\ \text{weakly in}\ L^2(0,T : (H^1)^\p).
			\end{equation}
			Using standard arguments (see for example \cite{Temam}), we can show that the obtained limit $m$ is a solution to the equation \eqref{eqn skeleton equation}.
			By Lemma 1.2 in \cite{Temam} (see page. 176), that $m\in L^2(0,T:H^2)$ and $\frac{du}{dt}\in L^2(0,T:(H^1)^\p)$  implies $m\in C([0,T]: H^1)$. This concludes the proof of the existence part in Theorem \ref{theorem unique solution for the skeleton equation}. The proof of uniqueness can be done along the lines of the proof of Theorem \ref{theorem pathwise uniqueness stochastic control equation}.

		\end{enumerate}

	\end{proof}

	\section{An Auxiliary Result}\label{Section An Auxiliary Result}
	We refer the reader to \cite{Zhai+Zhang_2015_LargeDeviations_2DSNSE_MultiplicativeLevyNoise}  for the following proposition (see also Lemma 4.5, \cite{ZB+UM+Zhai_Preprint_LDP_LLGE_JumpNoise}, Lemma 3.4, \cite{Budhiraja+Chen+Dupuis_2013_LDPDrivenByPoissonProcess}, Remark 5.4, \cite{UM+AAP_2021_LargeDeviationsSNSELevyNoise}).
	\begin{proposition}\label{proposition uniform bound on theta SK}
		Let us define the following space.
		\begin{equation}
			\mathcal{H}:= \l\{ v : B\to\mathbb{R}^{+} : \exists \delta>0, \forall \Gamma \in \mathcal{B} \text{ with } \nu(\Gamma)< \infty, \int_{\Gamma} e^{\delta v^2}\nu(dl)<\infty \r\}.
		\end{equation}
		Let $f\in \mathcal{H}\cap L^2(\nu)$. Then for every $K\in\mathbb{N}$,
		\begin{equation}
			\sup_{\theta\in\mathcal{S}^K}\int_{0}^{T} \int_{B} f(l) \l| \theta(l,s) - 1 \r|\, \nu(dl) \, ds : = C_K < \infty.
		\end{equation}
		In particular, for $f(l) = \l| l \r| \in \mathcal{H}\cap L^2(\nu)$,
		\begin{equation}
			\sup_{\theta\in\mathcal{S}^K}\int_{0}^{T} \int_{B} \l|l\r| \l| \theta(l,s) - 1 \r|\, \nu(dl) \, ds < \infty.
		\end{equation}
	\end{proposition}

	\par\medskip\noindent
	\bibliographystyle{plain}
	\bibliography{References_GokhaleSoham}
\end{document}